\documentclass[12pt, reqno, a4paper]{amsart}

\usepackage{amssymb}
\usepackage{color}
\definecolor{darkgreen}{rgb}{0,0.45,0} 
\usepackage[colorlinks,citecolor=blue,linkcolor=blue]{hyperref}
\usepackage[matrix, arrow, curve]{xy}
\usepackage{bbm}
\usepackage{xcolor}
\usepackage{comment}
\usepackage{enumitem}

\usepackage[english]{babel}

\binoppenalty=\maxdimen
\relpenalty=\maxdimen

\theoremstyle{plain}
\newtheorem{theorem}{Theorem}[section]
\newtheorem{lemma}[theorem]{Lemma}
\newtheorem{proposition}[theorem]{Proposition}
\newtheorem{corollary}[theorem]{Corollary}

\theoremstyle{remark}
\newtheorem{remark}[theorem]{Remark}
\theoremstyle{remark}
\newtheorem{remarks}[theorem]{Remarks}

\theoremstyle{definition}
\newtheorem{example}[theorem]{Example}
\newtheorem{examples}[theorem]{Examples}

\newtheorem{definition}[theorem]{Definition}

\numberwithin{equation}{section}

\DeclareMathOperator{\id}{id}

\DeclareMathOperator{\End}{End}
\DeclareMathOperator{\Aut}{Aut}

\DeclareMathOperator{\supp}{supp}
\DeclareMathOperator{\cosupp}{cosupp}

\DeclareMathOperator{\Hom}{Hom}

\DeclareMathOperator{\colim}{colim}
\newcommand{\LIO}{\mathsf{LIO}}

\newcommand{\Cc}{\mathcal{C}}

\newcommand{\Xx}{\mathcal{X}}

\newbox\skewpullbackbox
\setbox\skewpullbackbox=\hbox{\xy 0;<1mm,0mm>: \POS(4,0)\ar@{-} (-4,0) \ar@{-} (8,4)
	\endxy}

\newbox\skwepullbackbox
\setbox\skwepullbackbox=\hbox{\xy 0;<1mm,0mm>: \POS(16,0)\ar@{-} (10,0) \ar@{-} (12,4)
	\endxy}

\newbox\ksewpullbackbox
\setbox\ksewpullbackbox=\hbox{\xy 0;<1mm,0mm>: \POS(0,-8)\ar@{-} (0,-4) \ar@{-} (4,-4)
	\endxy}

\newbox\pullbackbox
\setbox\pullbackbox=\hbox{\xy 0;<1mm,0mm>: \POS(4,0)\ar@{-} (0,0) \ar@{-} (4,4)
	\endxy}
\newcommand{\pullback}{\copy\pullbackbox}

\newbox\pullbackabox
\setbox\pullbackabox=\hbox{\xy 0;<1mm,0mm>: \POS(-4,-6)\ar@{-} (-8,-6) \ar@{-} (-4,-2)
	\endxy}

\newbox\pullbackbbox
\setbox\pullbackbbox=\hbox{\xy 0;<1mm,0mm>: \POS(-4,-4)\ar@{-} (-8,-4) \ar@{-} (-4,0)
	\endxy}

\newbox\pullbackcbox
\setbox\pullbackcbox=\hbox{\xy 0;<1mm,0mm>: \POS(4,-7)\ar@{-} (0,-7) \ar@{-} (4,-3)
	\endxy}
\newcommand{\pullbackc}{\copy\pullbackcbox}

\newbox\pullbackdbox
\setbox\pullbackdbox=\hbox{\xy 0;<1mm,0mm>:\POS(-10,-6)\ar@{-} (-14,-6) \ar@{-} (-10,-2)
	\endxy}
\newcommand{\pullbackd}{\copy\pullbackdbox}

\newbox\pushoutbox
\setbox\pushoutbox=\hbox{\xy 0;<1mm,0mm>: \POS(0,4)\ar@{-} (0,0) \ar@{-} (4,4)
	\endxy}
\def\pushout{\copy\pushoutbox}

\newbox\pushoutabox
\setbox\pushoutabox=\hbox{\xy 0;<1mm,0mm>: \POS(2,8)\ar@{-} (2,4) \ar@{-} (8,8)
	\endxy}

\DeclareRobustCommand{\No}{\ifmmode{\nfss@text{\textnumero}}\else\textnumero\fi} 

\newcounter{dummy}
\makeatletter
\newcommand\myitem[1][]{\item[(#1)]\refstepcounter{dummy}\def\@currentlabel{#1}}
\makeatother

\usepackage{graphicx}
\usepackage[geometry]{ifsym}

\renewcommand{\square}{\text{\normalfont\scalebox{.6}{\SmallSquare}}}

\makeatletter
\@namedef{subjclassname@2020}{\textup{2020} Mathematics Subject Classification}
\makeatother

\begin{document}

\title[Lifting of locally initial objects]{Lifting of locally initial objects and universal (co)acting Hopf algebras}

\author{A.L. Agore}
\address{Simion Stoilow Institute of Mathematics of the Romanian Academy, P.O. Box 1-764, 014700 Bucharest, Romania}
\email{ana.agore@gmail.com, ana.agore@imar.ro}

\author{A.S. Gordienko}
\address{Department of Higher Algebra,
	Faculty of Mechanics and  Mathematics,
	M.\,V.~Lomonosov Moscow State University,
	Leninskiye Gory, d.1,  Main Building, GSP-1, 119991 Moskva, Russia }
\email{alexey.gordienko@math.msu.ru}

\author{J. Vercruysse}
\address{D\'epartement de Math\'ematiques, Facult\'e des sciences, Universit\'e Libre de Bruxelles, Boulevard du Triomphe, B-1050 Bruxelles, Belgium}
\email{joost.vercruysse@ulb.be}


\begin{abstract}	 The universal (co)acting bi/Hopf algebras introduced by Yu.\,I.~Manin, M.~Sweedler and D.~Tambara, the universal Hopf algebra of a given (co)module structure, as well as the universal group of a grading, introduced by J.~Patera and H.~Zassenhaus, find their applications in the classification of quantum symmetries. Typically, universal (co)acting objects are defined as initial or terminal in the corresponding categories and, as such, they do not always exist. In order to ensure their existence, we introduce the support of a given object, which generalizes the support of a grading and is used to restrict the class of objects under consideration. The existence problems for universal objects are formulated and studied in a purely categorical manner by seeing them as particular cases of the lifting problem for a locally initial object. We prove the existence of a lifting and, consequently, of the universal (co)acting objects under some assumptions on the base (braided or symmetric monoidal) category. In contrast to existing constructions, our approach is self-dual in the sense that we can use the same proof to obtain the existence of universal actions and coactions. In particular, when the base category is the category of vector spaces over a field, the category of sets or their duals, we recover known existence results for the aforementioned universal objects. The proposed approach allows us to apply our results not only to the classical categories of sets and vectors spaces and their duals but also to (co)modules over bi/Hopf algebras, differential graded vector spaces, $G$-sets and graded sets. 
\end{abstract}


\thanks{The study carried out by the second author was conducted under the state assignment of Lomonosov Moscow State University. The third author thanks the FNRS for support via the research project (PDR) T.0318.25 “Redisclosure”.}

\maketitle

\tableofcontents

\section{Introduction}

When an affine algebraic group $G$ is acting morphically on an affine algebraic variety $X$, this action $G\times X \to X$ corresponds to a homomorphism of algebras $\mathcal O(X) \to \mathcal O(X) \otimes \mathcal O(G)$ where $\mathcal O(X)$ and $\mathcal O(G)$ are the coordinate algebras of regular functions (functions that can be represented by polynomials in coordinates) on $X$ and $G$, respectively. Recall that the group structure on $G$ turns $\mathcal O(G)$ into a Hopf algebra. Moreover $\mathcal O(X)$ is an $\mathcal O(G)$-comodule algebra and a $U(\mathfrak g)$-module algebra where 
$U(\mathfrak g)$ is the universal enveloping algebra of the Lie algebra $\mathfrak g$ of $G$. In this classical situation,
$\mathcal O(G)$ and $\mathcal O(X)$ are commutative and $U(\mathfrak g)$ is cocommutative. 
More generally, if $H$ is a (neither necessarily commutative
nor necessarily cocommutative) Hopf algebra (co)acting on a (not necessarily commutative algebra) $A$, this can be interpreted in view of the above
as an action of a ``quantum group'' by ``quantum symmetries'' on some (not necessarily commutative) algebraic variety $X$ where the algebra $A$ plays the role of the coordinate algebra on $X$. As the structure of a general Hopf algebra can be more involved than the one arising from a usual (algebraic) group, the associated (co)actions do not only allow us to describe classical automorphisms and derivations, but also gradings and much wilder types of ``measurings''. Therefore, the problem of classifying (co)module structures on a given algebra $A$, can be seen geometrically as the problem of classifying quantum symmetries of $A$ and is as such a natural continuation and unification of the classifications of automorphisms, derivations and gradings of $A$.

Usually, gradings are classified either up to an isomorphism (when the grading group is fixed) or up to equivalence (when it is not important by elements of which group the graded components are marked), see e.g.  \cite{ElduqueKochetov}. The classification up to equivalence may seem coarser, however the notion of the universal group of the grading makes it possible to recover all groups that realize a concrete grading. 
Inspired by this, the notions of equivalence and universal Hopf algebras of (co)module structures on algebras were introduced in~\cite{AGV1} as a natural generalization of the aforementioned universal group of a grading. 
This construction was unified in \cite{AGV2} with the 
universal (co)acting bi/Hopf algebras of Sweedler~--- Manin~--- Tambara, introducing the $V$-universal (co)acting bi/Hopf algebras for a given algebra $A$,
where $V$ is a unital subalgebra of $\End_{\mathbbm k}(A)$ and $\mathbbm k$ is the base field. 
The advantage of such a unified theory is that it simplifies the classification of (co)module structures (or even in certain cases makes it possible at all) by providing duality theorems. Moreover, it is known that for a general infinite dimensional algebra $A$, the universal coacting bi/Hopf algebra
of Manin~--- Tambara do not always exist~(see \cite[Section 4.5]{AGV2}) and the use of $V$ provides the necessary restriction on the class of comodule structures under consideration to ensure the existence of the universal Hopf algebra for this class.
Furthermore, following Sweedler \cite{Sweedler}, rather than considering just (universal) (co)actions of a bi/Hopf algebra on an algebra, one can also consider (universal) measurings of coalgebras between different algebras, similarly to considering homomorphisms between algebras rather than just endomorphisms of a single algebra.

Particular interest has been shown recently for studying properties of Manin's universal coacting Hopf algebra (see e.g. \cite{HNUVVW, HNUVVW2, HNUVVW3, RV}). Furthermore, various other instances of universal (co)acting bi/Hopf algebras and universal (co)measurings surfaced in recent years motivated by different applications in e.g., subfactor theory \cite{BhaChiGos}, superpotential algebras \cite{ChiWalWan} or for studying comodules over bialgebroids via chain complexes of modules in \cite{ArdKaoMen}. Moreover, universal coacting weak bi/Hopf algebras are introduced in \cite{HWWW} as a generalization of the Manin-Tambara construction while measurings of Hopf algebroids  were recently considered in \cite{BK} for their applications in cyclic (co)homology theories whereas measurings by Hopf algebroids  as introduced in \cite{BB}, were used in \cite{BKS} for proving structure theorems for crossed products. All this motivates once more the need for a unified theory for such universal objects.

In this paper, we introduce the categorical foundations for the notion of universal bi/Hopf algebra. Instead of working with algebras and coalgebras over a field, we are considering $\Omega$-magmas (see Section~\ref{SubsectionOmegaMagmas}), i.e. objects
endowed with a collection of morphisms $\omega_A \colon A^{{}\otimes s(\omega)} \to A^{{}\otimes t(\omega)}$, indexed by $\omega \in \Omega$, in a braided monoidal category $\mathcal C$ throughout referred to as the \textit{base category}. We provide sufficient conditions for the existence of universal bi/Hopf monoids in terms of the category $\mathcal C$. Such a categorical approach allows us to treat algebras and coalgebras over fields, as well as monoids in $\mathbf{Sets}$ by a single theory. Furthermore, we may now consider the cases when all (co)algebras  under consideration as well as bi/Hopf algebras (co)acting on them are endowed with an additional structure: an action of a group or a Lie algebra, group grading, (co)action of a fixed Hopf algebra, structure of a Yetter-Drinfel'd module algebra, etc., and (co)actions of bi/Hopf algebras are compatible with this structure (see Section~\ref{ExamplesBaseCaategory} and \cite[Section 5]{AGV3}). 

Usually the conditions imposed on a the category $\Cc$ to obtain the existence of universal objects are rather strong. A common condition for example (see e.g. \cite{HylLopVas, GV, NorthPeroux, Porst1, Vasilakopoulou}) is to require that $\Cc$ is locally presentable. However, since a category and its dual cannot be at the same time locally presentable without being ``almost trivial'' (isomorphic to a poset category to be precise), such an approach
cannot be applied for universal coactions. In contrast to this, we propose a self-dual approach that allows us to provide a single proof for both the existence of universal actions and coactions, which can be applied to many categories of interest, such as Sets, Vector spaces, Yetter-Drinfel'd modules etc, as well as to their duals. The price to pay for this self-dual approach is that we have to introduce restrictions on (co)actions in terms of their (co)supports. In addition, we need to suppose the existence of (co)free (co)monoids, rather than deducing it from the conditions on $\Cc$. Indeed, the construction of a cofree coalgebra is not just a formal dual of the construction of a free monoid.

The main difference between existing literature and the approach we advocate here, is that all known general categorical results prove the existence of \emph{globally} universal (i.e. universal among \emph{all} measurings) (co)acting objects, and as already pointed out above, it is known \cite[Section 4.5]{AGV2}, that even in the case $\mathcal C = (\mathbf{Vect}_{\mathbbm k},\otimes)$, the category of vector spaces over a field $\mathbbm k$, such global objects do not exist. Therefore, we propose 
to restrict the class of (co)measurings and (co)actions under consideration which allows for many new categories to be considered, including the ones previously mentioned and not covered by the existing literature.
More precisely, a universal (co)action of a Hopf monoid in our sense can be defined as an initial object in some suitable full subcategory of the category of all (co)actions on a fixed $\Omega$-magma. Consequently, for the whole category of (co)actions this universal (co)action is then a {\em locally} initial object, by which we mean an object which admits {\em at most} one arrow into any other object. In contrast to (classical) initial objects, a category (in our case, the category of all coactions on a fixed $\Omega$-magma) can contain several non-isomorphic locally initial objects (in our case, each of these will be a coaction that is universal with respect to a different `support', see below). The aforementioned full subcategory consists then exactly of those objects whose images under some forgetful functor admit an arrow from the considered locally initial object. This suggests that the problem of finding universal Hopf monoids can be described as a lifting problem for locally initial objects and we provide several useful methods of lifting locally initial objects for various classes of functors (e.g., functors which admit an adjoint, functors which are full and faithful on some hom-sets) in Section~\ref{SectionLiftingProblem}.
To properly describe the full subcategory in which a locally initial object becomes initial, we introduce the notion of an absolute value of an object $x$ as the ''minimum'' among the locally initial objects which admit an arrow into $x$. This leads naturally to what we call the (co)support of a morphism which can be used to restrict the class of  objects under consideration to obtain the desired full subcategory.  In a subsequent paper \cite{AGV3}, we will show that in case of closed monoidal categories the universal (co)actions and (co)measurings obtained in this way, correspond with the $V$-universal (co)measurings as mentioned above and studied in \cite{AGV2}.

Concretely, the lifting of locally initial objects is carried out in several steps as we will outline now. We denote by
\begin{itemize}
\item $A$ and $B$ some $\Omega$-magmas;
\item $\mathbf{MorTens}(A,B)$ and $\mathbf{TensMor}(A,B)$ the categories of, respectively, morphisms $A\to B \otimes Q$
and morphisms $P\otimes A \to B$ for some objects $P$ and $Q$; 
\item $\mathbf{Comeas}(A,B)$ and $\mathbf{Meas}(A,B)$ the categories of, respectively, comeasurings $A\to B \otimes Q$
and measurings $P\otimes A \to B$ where $P$ is a comonoid and $Q$ a monoid; 
\item $\mathbf{ComodStr}(A)$ and $\mathbf{ModStr}(A)$ the categories of, respectively, comodule structures $A\to A \otimes P$ and module structures $Q\otimes A \to A$ where $P$ is a comonoid and $Q$ a monoid;  
\item $\mathbf{Coact}(A)$ and $\mathbf{Act}(A)$  the categories of bimonoid (co)actions on $A$;
\item $\mathbf{HCoact}(A)$ and $\mathbf{HAct}(A)$  the categories of  Hopf monoid (co)actions on $A$;
\item  $G$, $G_1$--$G_4$, $G'$, $G_1'$--$G_4'$ the corresponding forgetful/embedding functors, as in the diagram below.
$$\xymatrix{
	& \mathbf{HCoact}(A)  \ar[d]^{G_4} & \,\, & \mathbf{HAct}(A)^{\mathrm{op}} \ar[d]^{G'_4} \\
\mathbf{Comeas}(A,A) \ar[d]^{G_1}	& \mathbf{Coact}(A) \ar[l]_-{G_3} \ar[d]^{G_2} &  \,\, & \mathbf{Act}(A)^{\mathrm{op}} \ar[d]^{G'_2} \ar[r]^-{G'_3} & \mathbf{Meas}(A,A)^{\mathrm{op}} \ar[d]^{G'_1}   \\
\mathbf{MorTens}(A,A)	& \mathbf{ComodStr}(A) \ar[l]_-{G}  & \,\, & \mathbf{ModStr}(A)^{\mathrm{op}}  
\ar[r]^-{G'} & \mathbf{TensMor}(A,A)^{\mathrm{op}}
}$$	
\end{itemize}
In Theorems~\ref{TheoremMonUnivComeasExistence} and~\ref{TheoremDUALComonUnivMeasExistence}
we lift locally initial objects along the forgetful functors $G_1 \colon \mathbf{Comeas}(A,B) \to \mathbf{MorTens}(A,B)$ and $G_1' \colon \mathbf{Meas}(A,B)^\mathrm{op} \to \mathbf{TensMor}(A,B)^\mathrm{op}$, respectively.
As opposed to the existing results from~\cite{HylLopVas, GV, NorthPeroux, Porst1, Vasilakopoulou}, we do not require our base category to be locally presentable nor monoidal closed. The precise conditions that we impose are detailed in Sections~\ref{SubsectionSupportCoactingConditions} and~\ref{SubsectionDUALCosupportActingConditions}.
When $A=B$, one can consider bi/Hopf monoid (co)actions as well.
Corollaries~\ref{CorollaryBimonUnivCoactingExistence} and~\ref{CorollaryDUALBimonUnivActingExistence}
provide a method of lifting objects in $\mathbf{ComodStr}(A)$ and $\mathbf{ModStr}(A)^{\mathrm{op}}$, whose images under $G$
and $G'$  are locally initial, along the functors $G_2$ and $G_2'$, respectively. The actual lifting is made along $G_3$ and $G_3'$.

Along the way, in order to construct universal (co)acting Hopf monoids, we prove in Theorem~\ref{TheoremBimonHopfLeftAdjoint}
and Theorem~\ref{TheoremDUALBimonHopfRightAdjoint}, respectively, that under certain conditions on the base braided monoidal category $\mathcal C$ the category $\mathsf{Hopf}(\mathcal C)$ of Hopf monoids in $\mathcal C$ is a reflective and coreflective
subcategory of the category $\mathsf{Bimon}(\mathcal C)$ of bimonoids in $\mathcal C$. The constructions follow the general outline proposed in \cite{Takeuchi, CH} (for Hopf algebras over fields) and \cite{Porst1}, see also \cite{GV}. However, we do not assume $\mathcal C$ to be symmetric nor locally finitely presentable as in \cite{Porst1} and \cite{GV}. Instead, the conditions we impose on the base category $\Cc$ make the construction in Theorem~\ref{TheoremBimonHopfLeftAdjoint} dualizable.
Finally, in Theorems~\ref{TheoremHopfMonUnivCoactingExistence} and~\ref{TheoremDUALHopfMonUnivActingExistence}
we lift the locally initial objects obtained in $\mathbf{Coact}(A)$ and $\mathbf{Act}(A)^{\mathrm{op}}$
to $\mathbf{HCoact}(A)$ and $\mathbf{HAct}(A)^{\mathrm{op}}$, respectively.

As suggested by the above diagrams, universal (co)acting objects give rise to a very rich duality theory which is thoroughly investigated in \cite{AGV3}.

For all undefined notions and further details on (braided monoidal) categories see \cite{MacLaneCatWork} (resp. \cite{JoyStr}) and for unexplained concepts in Hopf algebra theory we refer the reader to \cite{DNR}.

\section{Lifting of locally initial objects}\label{SectionLiftingProblem}

\subsection{Locally initial objects}\label{SubsectionLIO}
An object $x_0$ in a category $X$ is \textit{locally initial} if for every object $x$ in $X$ there exists
at most one morphism $x_0 \to x$.  If for every object $x$ in $X$ there exists
exactly one morphism $x_0 \to x$, then $x_0$ is just \textit{initial}. It is well known that the initial object is unique up to an isomorphism.

\begin{example} Let $a$ be an object in a category $C$.
	Then locally initial objects in the comma category $(a \downarrow C)$ are just epimorphisms $a\to b$ for arbitrary objects $b$ in $C$.
\end{example}	

Locally initial objects form a preorder $\LIO(X)$ where $x_1 \succcurlyeq x_2$ if there exists a morphism $x_1 \to x_2$.
For a given locally initial object $x_0$ denote by $X(x_0)$ the full subcategory of $X$ consisting
of all objects $x$ such that there exists a morphism $x_0 \to x$. Then $x_0$ is the initial object in $X(x_0)$.

Fix a functor $G \colon Y \to X$ where $Y$ is a category. For  $x_0 \in \LIO(X)$ denote by  $Y_G(x_0)$ the full subcategory of $Y$ consisting of all objects $y$
such that $Gy$ is an object in $X(x_0)$:
$$ \xymatrix{ 
	\ Y_G(x_0) \ \ar@{^{(}->}[r] \ar[d]  \ar@{}[rd]|<{\pullback} & \ Y \  \ar[d]^G \\
	\ X(x_0) \  \ar@{^{(}->}[r] &  \ X \   \\       
}
$$

\noindent\textbf{Lifting Problem.} \textit{Given $x_0 \in \LIO(X)$, find an initial object $y_0$ in $Y_G(x_0)$.}

Note that the terminology is used loosely since our definition does not require that $Gy_0 = x_0$. However, a slightly different equality may hold, see Proposition~\ref{PropositionLIOAbsValueLiftCriterion} below.
	A criterion for $Gy_0 = x_0$ to hold is given in Proposition~\ref{PropositionLIOLiftingArrowsFromLIO}.

\begin{proposition} Let $G \colon Y \to X$ be a functor and let $y_0$ be an initial object in $Y_G(x_0)$
	for some $x_0 \in \LIO(X)$.
	Then $y_0 \in \LIO(Y)$.	
\end{proposition}
\begin{proof}
	Let $y_0 \to y$ be an arrow in $Y$. Then there is an arrow $x_0 \to Gy_0 \to Gy$ in $X$. Thus $y$ is an object in $Y_G(x_0)$.
	Therefore, the arrow $y_0 \to y$ is unique since the object $y$ is initial in $Y_G(x_0)$.
\end{proof}	

In most cases below we will lift initial objects using \textit{ad hoc} methods. However, in the case when $G$ admits a left adjoint $F$, the lifting can be done by simply applying the functor $F$.

\begin{theorem}\label{TheoremLIOLiftAdjointExistence}
	Let $\tilde G \colon \tilde Y \to Y$ and $G \colon Y \to X$ be functors for some categories $X,Y,\tilde Y$.
	Let $y_0$ be the initial object in $Y_G(x_0)$ for some $x_0 \in \LIO(X)$.
	Suppose $\tilde G$ admits a left adjoint $\tilde F \colon Y \to \tilde Y$:	
	$$\xymatrix{ \tilde Y \ar@<-6pt>[r]^{\perp}_{\tilde G} & Y \ar@<-6pt>[l]_{\tilde F} \ar[r]^G & X}$$
 Then $\tilde F y_0$ is the initial object in $\tilde Y_{G \tilde G}(x_0)$.
 Moreover,  $\tilde Fz \in  \LIO(\tilde Y)$ for every $z \in \LIO(Y)$.
\end{theorem}
\begin{proof}
By the definition of $\tilde Y_{G \tilde G}(x_0)$, the functor $\tilde G$ maps $\tilde Y_{G \tilde G}(x_0)$ to $Y_G(x_0)$. Moreover, if $y$ is an object in $Y_G(x_0)$,
then we have an arrow $x_0 \to Gy \to G \tilde G \tilde F y$, where $y \to \tilde G \tilde F y$ is the unit of the adjunction $\tilde F \dashv \tilde G$,
so $\tilde F y$ is an object in $\tilde Y_{G \tilde G}(y_0)$. In other words, the adjunction $\tilde F \dashv \tilde G$ restricts to the subcategories $Y_G(x_0)$
and $\tilde Y_{G \tilde G}(x_0)$. As $\tilde F$ is a left adjoint, it preserves all colimits and, in particular, the initial object, which is the colimit of the functor from the empty category. Hence $\tilde Fy_0$ is indeed the initial object in $\tilde Y_{G \tilde G}(x_0)$.
If $z \in \LIO(Y)$, then the set $Y(z, \tilde G\tilde y)$ of morphisms $z \to G\tilde y$ consists
of no more than a single element for every $\tilde y$ in $\tilde Y$. But the adjunction
implies that the sets $Y(z, \tilde G\tilde y)$ and $\tilde Y(\tilde Fz, \tilde y)$ have the same cardinality.
Hence $\tilde Fz \in  \LIO(\tilde Y)$ too.
\end{proof}	
Applying Theorem~\ref{TheoremLIOLiftAdjointExistence} for the case when $X=Y$ and $G=\id_X$, we get 
\begin{corollary}\label{CorollaryLIOLiftAdjointExistence} Let  $x_0 \in \LIO(X)$ for some category $X$.
	Suppose a functor $G \colon Y \to X$, where $Y$ is a arbitrary category, admits a left adjoint $F \colon X \to Y$. Then $Fx_0$ is the initial object in $Y_G(x_0)$. Moreover, $Fx_0 \in \LIO(Y)$.
\end{corollary}

\begin{remark}
	Corollary~\ref{CorollaryLIOLiftAdjointExistence} hints at another approach to the Lifting Problem.
	It is well known (see e.g. \cite[Chapter IV, Section 1, Theorem 2, (ii)]{MacLaneCatWork})
	that a functor $G \colon Y \to X$ admits a left adjoint if and only if for every object $x$ in $X$
	the comma category $(x \downarrow G)$ admits an initial object. Hence the Lifting Problem above is the restriction
	of the problem of constructing of a left adjoint functor to objects $x\in \LIO(X)$.
\end{remark}

The most obvious candidate to check, whether it is an initial object in  $Y_G(x_0)$ for a given $x_0 \in \LIO(X)$,
is some preimage $y_0$ of $x_0$ in $Y$ if that preimage indeed exists. Proposition~\ref{PropositionLIOLiftingArrowsFromLIO}
shows that such $y_0$ is a solution of the Lifting Problem if $G$ is full and faithful on certain hom-sets.

\begin{proposition}\label{PropositionLIOLiftingArrowsFromLIO}
	Let $G \colon Y \to X$ be a functor and let $y_0$ be an object in $Y$ such that $Gy_0 \in \LIO(X)$.
	Then $y_0$ is an initial object in $Y_G(Gy_0)$ if and only if for every object $y$ in $Y_G(Gy_0)$ the map $Y(y_0, y) \to X(Gy_0, Gy)$, defined by $G$,
	 is a bijection.
 \end{proposition}	
\begin{proof} For every
	 object $y$ in $Y_G(Gy_0)$ the set $X(Gy_0, Gy)$ consists of a single arrow $Gy_0 \to Gy$ since $Gy_0 \in \LIO(X)$.
	 Then $Y(y_0, y) \to X(Gy_0, Gy)$ is a bijection if and only if the set $Y(y_0, y)$ consists of a single element.
	  However, the latter holds for every $y$ in  $Y_G(Gy_0)$ if and only if $y_0$ is an initial object in $Y_G(Gy_0)$.
\end{proof}	
	
 \subsection{Absolute values}\label{SubsectionLIOAbsValue} Let $X$ be again a category.
 For a given object $x$ denote by $|x|$ an object in $\LIO(X)$
 such that $x$ is an object in $X(|x|)$ and for any other object $x_1$ in $\LIO(X)$
 such that $x$ is an object in $X(x_1)$ we have $|x| \preccurlyeq x_1$.
 In other words, $|x|$ is the absolute minimum of such $x_1 \in \LIO(X)$
for which there exists a morphism $x_1 \to x$:
 $$\xymatrix{ x &  \ar[l] \ar@{-->}[d] x_1 \\
 	& |x| \ar[lu]
 }$$
 
 The object $|x|$ (if it exists) is unique up to an isomorphism. We will call $|x|$ the \textit{absolute value} of $x$.
 
 The following example motivates the terminology used:

 \begin{example} Consider the category $\mathbb C$ where the objects are all complex numbers
	and $$\mathbb C(a,b)=\left\lbrace
	\begin{array}{cl}
	\left\lbrace 1_a, e_a \right\rbrace & \text{ if } a=b \text{ and } a\notin \mathbb{R}_+, \\
	\left\lbrace f_{a,b} \right\rbrace & \text{ if } a\in \mathbb{R}_+ \text{ and } |b|\leqslant a, \\
	\varnothing & \text{ otherwise.}
	\end{array}
	\right.$$ Define $$f_{b,c} f_{a,b} = f_{a,c},\
	1_c^2=1_c,\ e_c^2 = 1_c e_c = e_c 1_c = e_c,\ 
	e_c f_{b,c}= 1_c f_{b,c} = f_{b,c}$$	 
	for all $c\in \mathbb C$, $a,b \in \mathbb{R}_+$,
	$|c|\leqslant b \leqslant a$.
	Then $\LIO(\mathbb C)=\mathbb{R}_+$ and for every $c\in \mathbb C$
	the object $|c|\in \mathbb{R}_+$ coincides with the absolute value of $c$ in the usual sense.
\end{example}	
 
 \begin{remarks}
 \hspace{0.1cm}
 \begin{enumerate}
 \item The arrow $|x| \to x$ is just the terminal object in the comma category~$(\LIO(X)\downarrow x)$.
\item For every $x \in \LIO(X)$ we have $x = |x|$. 
\end{enumerate}
\end{remarks}

\begin{example} Let $f \colon a \to b$ be a morphism in an abelian category $A$ 
	and let $X := (a \downarrow A)$. Then $|f| = \mathop\mathrm{coim} f := \mathop\mathrm{coker}(\mathop\mathrm{ker} f)$.
\end{example}

 Now return to the situation when we have a functor $G \colon Y \to X$ for some categories $X$ and $Y$.
 
 Given an object $y$ in $Y$, denote $|y|_G := |Gy|$.
 
 For objects $y_1, y_2$ in $Y$ we write $y_1 \preccurlyeq y_2$
 and say that $y_1$ is \textit{coarser} than $y_2$ and $y_2$ is \textit{finer} than $y_1$ if $|y_1|_G \preccurlyeq |y_2|_G$.
 If $y_1 \preccurlyeq y_2$ and $y_2 \preccurlyeq y_1$, then we say that $y_1$ and $y_2$ are \textit{(support) equivalent}.
 Since the absolute value is defined up to an isomorphism, $y_1$ and $y_2$ are equivalent if and only if $|y_1|_G = |y_2|_G$.

Now we prove some simple criteria that can be used to check whether $y_1 \succcurlyeq y_2$ for given
objects $y_1$ and $y_2$.

\begin{proposition}\label{PropositionLIOCoarserFinerEquivalent} 
	Let $G \colon Y \to X$ be a functor for some categories $X$ and $Y$
	and let $y_1$ and $y_2$ be some objects in $Y$ such that there exist $|y_1|_G$ and $|y_2|_G$.
	Then $y_1 \succcurlyeq y_2$ holds if and only if there exists an arrow $|y_1|_G\to Gy_2$ in $X$.
\end{proposition}
\begin{proof}  If $y_1 \succcurlyeq y_2$, then there exists an arrow $|y_1|_G\to |y_2|_G$.
	Composing this arrow with $|y_2|_G\to Gy_2$, we get the desired arrow $|y_1|_G\to Gy_2$.
	
	Conversely, suppose there exists an arrow $|y_1|_G\to Gy_2$ in $X$. By the definition of $|y_1|_G$, this arrow must factor through
	$|y_2|_G$, which yields an arrow $|y_1|_G\to |y_2|_G$. Hence $y_1 \succcurlyeq y_2$.
\end{proof}		
\begin{corollary}\label{CorollaryLIOCoarserFinerSufficientConditions}
Let $G \colon Y \to X$ be a functor for some categories $X$ and $Y$.
Suppose there is an arrow $y_1 \to y_2$ in $Y$ such that there exist $|y_1|_G$ and $|y_2|_G$.
Then $y_1 \succcurlyeq y_2$.
\end{corollary}	
\begin{proof} We consider the composition of $|y_1|_G\to Gy_1$ and $Gy_1 \to Gy_2$  and apply Proposition~\ref{PropositionLIOCoarserFinerEquivalent}.
\end{proof}	

\begin{remark} 
 Using the notion of an absolute value, we see that if absolute values of all objects in $X$ exist, then, given $x_0 \in \LIO(X)$, the category $Y_G(x_0)$ consists of all objects $y$ in $Y$ such that $|y|_G \preccurlyeq x_0$.
\end{remark} 
 
 \begin{proposition}\label{PropositionLIOAbsValueLiftCriterion} 
 	Let $G \colon Y \to X$ be a functor for some categories $X$ and $Y$. Suppose that in $X$ there exist absolute values of all objects. Let $y_0$ be the initial object in $Y_G(x_0)$ for some $x_0 \in \LIO(X)$.
 	Then $x_0 = |y_0|_G$  if and only if $x_0 = |y|_G$ for some object $y$ in $Y$.
 \end{proposition}
 \begin{proof}
 	Suppose $x_0 = |y|_G$ for some object $y$ in $Y$. Then there exists an arrow $x_0 \to Gy$ in $X$
 	and	the object $y$ belongs to $Y_G(x_0)$. Hence there is an arrow $y_0 \to y$ in $Y$.
 	Note that by the definition of the absolute value, the arrow $x_0 \to Gy_0$
 	factors through $|y_0|_G$. 
 	We get the following diagram:
 	$$\xymatrix{
x_0 \ar[d] \ar@{-->}[r] \ar[rd]   & |y_0|_G \ar[d] \\
Gy & \ar[l]   Gy_0	
}$$
 	In particular, $|y_0|_G \preccurlyeq x_0$. On the other hand, as the arrow $x_0 \to Gy$
 	factors through $|y_0|_G$ and $|y|_G = x_0$, we have $|y_0|_G \succcurlyeq x_0$.
 	Therefore, $|y_0|_G$ and $x_0$ are isomorphic, since both belong to $\LIO(X)$. As a consequence,
 	$x_0$ is the absolute value of $y_0$ too. (Recall that the absolute value is defined up to an isomorphism.)
 	
 	The converse is trivial.
 \end{proof}

  If for an object~$y$ in $Y$
 there exists an initial object $y_\mathrm{univ}$ in $Y_G(|y|_G)$, then we call $y_\mathrm{univ}$ the \textit{universal object of $y$}. Note that, by Proposition~\ref{PropositionLIOAbsValueLiftCriterion}, we have $|y_\mathrm{univ}|_G = |y|_G$, so~$y_\mathrm{univ}$ is indeed universal among all objects $y_{1}$ in $Y$ (support) equivalent to $y$, i.e. for every $y_{1}$
 with $|y_{1}|_G = |y|_G$ there exists a unique arrow $y_\mathrm{univ} \to y_{1}$ in $Y$.
 
 In Proposition~\ref{PropositionLIOLiftingArrowsFromLIO} we considered the particular case when it was possible to lift certain arrows from $X$ to $Y$ in a unique way. If, in addition, one is able to find preimages under $G$ of absolute values of objects in $Y$,
 then the corresponding universal objects exist and this particular case of the Lifting Problem has a solution:
 
 \begin{proposition}\label{PropositionLIOLiftingAbsValue}
 	Let $G \colon Y \to X$ be a functor and let $y_0,y_1$ be objects in $Y$ such that $Gy_0 = |y_1|_G$
 	and for every object $y$ in $Y_G(|y_1|_G)$ the map $Y(y_0, y) \to X(Gy_0, Gy)$, defined by $G$,
 	is a bijection.
 	Then $y_0$ is an initial object in $Y_G(|y_1|_G)$.
 \end{proposition}
\begin{proof} Note that $G y_0 = |y_1|_G \in \LIO(X)$ and apply Proposition~\ref{PropositionLIOLiftingArrowsFromLIO}.
\end{proof}		
 
 \subsection{Group gradings}\label{SubsectionGroupGradings}
 
 While the main examples will follow in the next sections, we provide here an example to justify the terminology introduced.
 
  Let $A$ be a (neither necessarily unital, nor necessarily associative) algebra over a field $\mathbbm k$.
 	
 	Let $Y$ be the category where \begin{itemize} \item the objects are all gradings $$A=\bigoplus\limits_{h\in H} A^{(h)} \text{ (direct sum of subspaces)}$$
 	on $A$ by arbitrary groups $H$ where $A^{(g)}A^{(h)}\subseteq A^{(gh)}$ for all $g,h \in H$;
 	\item
 	the morphisms between group gradings $$A=\bigoplus\limits_{h\in H_1} A^{(h)}_1 \text{
 	and }A=\bigoplus\limits_{h\in H_2} A^{(h)}_2$$
 	are all group homomorphisms  $\varphi\colon H_1 \to H_2$ such that $A_1^{(h)}\subseteq A_2^{\bigl(\varphi(h)\bigr)}$
 	for all $h\in H_1$.\end{itemize}
 
 	Let $X$ be the category  where \begin{itemize} \item the objects are all gradings 
 	on $A$ by arbitrary sets $T$ that can be realized as group gradings, i.e. such decompositions $\Gamma \colon A=\bigoplus\limits_{t\in T} A^{(t)}$
 	that there exists a group grading $A=\bigoplus\limits_{h\in H} A_0^{(h)}$
 	and an embedding $$\tau \colon \supp \Gamma := \lbrace t\in T \mid A^{(t)}\ne 0  \rbrace \hookrightarrow H$$
 	such that $A^{(t)}=A_0^{\bigl(\tau(t)\bigr)}$ for all $t\in\supp \Gamma$;
 	\item the morphisms between set gradings $$\Gamma_1 \colon A=\bigoplus\limits_{t\in T_1} A^{(t)}_1\text{
 		and }\Gamma_2 \colon A =\bigoplus\limits_{t\in T_2} A^{(t)}_2$$
 	are all maps $\varphi\colon T_1 \to T_2$ such that $A_1^{(t)}\subseteq A_2^{\bigl(\varphi(t)\bigr)}$
 	for all $t\in T_1$.
 \end{itemize}

Note that if $\Gamma \colon A=\bigoplus\limits_{t\in T} A^{(t)}$, then $\Gamma \in \LIO(X)$ if and only if $T = \supp \Gamma$.

Now let $G \colon Y \to X$ be the functor that forgets the group structure on then grading group. Then
for $\Gamma \colon A=\bigoplus\limits_{h\in H} A^{(h)}$ we have
$|\Gamma|_G = \Gamma_0$ where $\Gamma_0 \colon A=\bigoplus\limits_{g\in \supp \Gamma} A^{(g)}$.
In other words, in this example $|\Gamma|_G$ turns out to be
the restriction of the grading $\Gamma$ to its support $ \supp \Gamma$.
The finer/coarser relation and the equivalence relation introduced in Section~\ref{SubsectionLIOAbsValue} coincide here with the usual ones for gradings. 

Finally, the left adjoint functor $F \colon X \to Y$ is constructed as follows. Let $\Gamma \colon A=\bigoplus\limits_{t\in T} A^{(t)}$. Denote $\mathsf{G}_\Gamma := \mathcal F(T)/N$ where $\mathcal F(T)$ is the free group with the set $T$ of free generators
and $N$ is the normal closure in $\mathcal F(T)$  of the words $hst^{-1}$ where $h,s,t \in T$
are all such elements that $0 \ne A^{(h)}A^{(s)} \subseteq A^{(t)}$.
Now let $$F\Gamma \colon A=\bigoplus\limits_{g\in \mathsf{G}_\Gamma} A_1^{(g)}$$
where $A_1^{(g)} := A^{(t)}$ if $g\in \mathsf{G}_\Gamma$ is the image of some $t\in T$ in $\mathsf{G}_\Gamma$
and $A_1^{(g)} := 0$ otherwise.

Given a group grading $\Gamma_1$, the object $F(|\Gamma_1|_G)$
is initial in $Y_G(|\Gamma_1|_G)$ by Corollary~\ref{CorollaryLIOLiftAdjointExistence}. The corresponding group $\mathsf{G}_{|\Gamma_1|_G}$ is called the \textit{universal group}
of the grading $\Gamma_1$ and it was first introduced by J.~Patera and H.~Zassenhaus~\cite{PZ89} in 1989.

\section{(Co)modules, $\Omega$-magmas, (co)measurings and (co)actions}\label{SectionDefinitionsComodulesOmegaMeasurings}

In order to proceed to the main cases of categories and functors where we solve the Lifting Problem,
we recall definitions that generalize well known notions for the categorical setting.

\subsection{Modules over monoids}

Let $(A, \mu, u)$ be a monoid in a monoidal category $\mathcal C$
and let $\psi \colon A \otimes M \to M$ be a morphism for some object $M$ in $\mathcal C$.
Recall that the pair $(M,\psi)$ is a \textit{(left) $A$-module} if the diagrams below are commutative:

$$\xymatrix{ M \ar@{=}[d] \ar[r]^(0.4){\sim} & \mathbbm{1}\otimes M \ar[d]^{u\,\otimes\, \id_M} \\
	M  & A \otimes M \ar[l]^(0.6)\psi
} \qquad
\xymatrix{ (A \otimes A) \otimes M \ar[rr]^{\sim} \ar[d]^{\mu\, \otimes\, \id_M} & & A \otimes (A \otimes M) 
	\ar[d]^{\id_A \otimes\, \psi} \\
	A \otimes M 	\ar[rd]^{\psi}  & & A \otimes M \ar[ld]_{\psi} \\
	& M \\
}
$$

\begin{examples} 
\hspace{0.1cm}
\begin{enumerate}
\item Let $\mathcal C = \mathbf{Vect}_\mathbbm{k}$ for a field $\mathbbm{k}$ where the monoidal product $\otimes$ is the usual tensor product $\otimes$ and the monoidal unit $\mathbbm{1}$ is just the base field $\mathbbm{k}$.
Then we recover the usual definition of a unital module over a unital associative algebra $A$;
\item Let $\mathcal C = \mathbf{Sets}$ where the monoidal product $\otimes$ is just the usual Cartesian product $\times$ and the monoidal unit $\mathbbm{1}$ is a one element set $\lbrace * \rbrace$.
	Then we get the usual definition of a set $M$ with an action of a monoid $A$.
\end{enumerate}	
\end{examples}

\subsection{Comodules over comonoids}

Dually, let $(C, \Delta, \varepsilon)$ be a comonoid in a monoidal category $\mathcal C$
and let $\rho \colon M \to M \otimes C$ be a morphism for some object $M$ in $\mathcal C$.
Recall that the pair $(M,\rho)$ is a \textit{(right) $C$-comodule} if the diagrams below are commutative:

$$\xymatrix{ M \ar[r]^(0.4)\rho \ar@{=}[d] &  M \otimes C  \ar[d]^{\id_M \otimes\, \varepsilon} \\
	M   &  M \otimes \mathbbm{1}\ar[l]_(0.6){\sim}  \\
} \qquad
\xymatrix{ 	& M \ar[ld]_{\rho} \ar[rd]^{\rho} \\
	M \otimes C \ar[d]^{ \id_M \otimes\, \Delta} & & M \otimes C \ar[d]^{\rho \, \otimes\, \id_C} \\
	M  \otimes (C \otimes C)   & & (M \otimes C) \otimes C  \ar[ll]_{\sim}
	\\
}
$$

\begin{examples} 
\hspace{0.1cm}
\begin{enumerate}
\item If $\mathcal C = \mathbf{Vect}_\mathbbm{k}$ for a field $\mathbbm{k}$ we obtain the usual definition of a (counital) comodule over a (counital coassociative) coalgebra $C$;
\item Recall that in $\mathbf{Sets}$ all comonoids $C$ are just sets endowed with the diagonal map $\Delta \colon C \to C\times C$, $\Delta(c)  = (c,c)$. Therefore, if $\mathcal C = \mathbf{Sets}$ all $C$-comodules are just sets~$M$ endowed with maps $\rho_1 \colon M \to C$. Namely, $\rho(m)=(m,\rho_1(m))$ for all $m\in M$.
\end{enumerate}	
\end{examples}

\subsection{$\Omega$-magmas}\label{SubsectionOmegaMagmas}

Let $\Omega$ be a set together with maps $s,t \colon \Omega \to \mathbb Z_+$. 

\begin{definition} An \textit{$\Omega$-magma} in a monoidal category $\mathcal C$ is 
	an object $A$ endowed with morphisms $\omega_A \colon A^{\otimes s(\omega)} \to
	A^{\otimes t(\omega)}$ for every $\omega \in \Omega$. (We will usually drop the subscript $A$ and denote
	the map just by $\omega$.) Here we use the convention that $A^{\otimes 0} := \mathbbm{1}$, the monoidal unit in $\mathcal C$.
\end{definition}

\begin{remark}
	
	Note that $\omega_A$ is not required to satisfy any identities.
	
\end{remark}

\begin{examples}\label{comod} 
\hspace{0.1cm}
\begin{enumerate}
\item\label{ExampleAlgebraOmega} Every (neither necessarily associative, nor necessarily unital) algebra over a field $\mathbbm{k}$
	is just an $\Omega$-magma in $\mathbf{Vect}_\mathbbm{k}$	for $\Omega=\lbrace\mu \rbrace$, $s(\mu)=2$, $t(\mu)=1$; 
\item\label{ExampleUnitalAlgebraOmega} Every unital algebra $A$ over a field $\mathbbm{k}$ is an example of an $\Omega$-magma in $\mathbf{Vect}_\mathbbm{k}$	for $\Omega=\lbrace\mu, u \rbrace$, $s(\mu)=2$, $t(\mu)=1$, $s(u)=0$, $t(u)=1$, where $u_A \colon \mathbbm{k} \to A$ is defined by $u_A(\alpha)=\alpha 1_A$
	for $\alpha \in \mathbbm{k}$. An ordinary monoid is an example of an $\Omega$-magma in $\mathbf{Sets}$ for the same $\Omega$;
\item\label{ExampleCoalgebraOmega} Every coalgebra $C$ over a field $\mathbbm{k}$ is an example of an $\Omega$-magma in $\mathbf{Vect}_\mathbbm{k}$	for $\Omega=\lbrace\Delta, \varepsilon \rbrace$, $s(\Delta)=1$, $t(\Delta)=2$, $s(\varepsilon)=1$, $t(\varepsilon)=0$.
In general, $\Omega$-magmas in $\mathbf{Vect}_\mathbbm{k}$ are called \textit{$\Omega$-algebras} over $\mathbbm{k}$~\cite{AGV2}.
\item\label{ExampleBraidingOmega} An object $A$ endowed with a braiding $c_A \colon A \otimes A \to A \otimes A$
	is an example of an $\Omega$-magma for $\Omega=\lbrace c \rbrace$, $s(c)=2$, $t(c)=2$.
\end{enumerate}
\end{examples}	

\subsection{Measurings}\label{SubsectionMeasurings}

Fix a braided monoidal category $\mathcal C$ with a braiding $c$ and a monoidal unit~$\mathbbm{1}$. Let $P$ be a comonoid in $\mathcal C$ with a comultiplication
$\Delta \colon P \to P \otimes P$ and a counit $\varepsilon \colon P \to \mathbbm{1}$. Consider the monoidal category $\mathsf{TensMor}(P)$ where the objects are morphisms $P\otimes A \to B$ and morphisms between objects $\psi_1 \colon P\otimes A_1 \to B_1$ and $\psi_2 \colon P\otimes A_2 \to B_2$
are pairs of morphisms $\alpha \colon A_1 \to A_2$ and $\beta \colon B_1 \to B_2$ making the diagram below
commutative:

$$\xymatrix{  P\otimes A_1 \ar[d]_{\id_P\, \otimes \, \alpha }\ar[r]^(0.6){\psi_1} & B_1 \ar[d]_{\beta } \\
	P\otimes A_2 \ar[r]^(0.6){\psi_2} & B_2
}$$

The monoidal product $\psi_1 \mathbin{\widetilde\otimes} \psi_2 \colon P \otimes (A_1 \otimes A_2) \to B_1 \otimes B_2$
of objects $\psi_1 \colon P\otimes A_1 \to B_1$ and $\psi_2 \colon P\otimes A_2 \to B_2$
in $\mathsf{TensMor}(P)$ is defined as the composition of the morphisms below:

$$\xymatrix{ P \otimes (A_1 \otimes A_2) \ar[rr]^(0.45){\Delta \, \otimes \,  \id_{A_1 \otimes A_2}} &  & (P \otimes P) \otimes (A_1 \otimes A_2) \ar[rr]^{\id_{P} \, \otimes \,  c_{P, A_1} \, \otimes \,  \id_{A_2}} & \quad & (P \otimes A_1) \otimes (P \otimes A_2)
	\ar[d]^{\psi_1 \otimes \psi_2} \\ & & & & B_1 \otimes B_2 }$$

The monoidal unit of $\mathsf{TensMor}(P)$ is the composition $P \otimes \mathbbm{1} \mathop{\widetilde{\to}} P \xrightarrow{\varepsilon} \mathbbm{1}$.
The axioms of a monoidal category for $\mathsf{TensMor}(P)$ are consequences of those for $\mathcal C$ and the fact that $P$ is a comonoid.

A \textit{measuring} of $\Omega$-magmas is an $\Omega$-magma $\psi \colon P\otimes A \to B$ in the category $\mathsf{TensMor}(P)$. Note that the structure of an $\Omega$-magma on $\psi$ endows the objects $A$ and $B$ with structures
of $\Omega$-magmas in $\mathcal C$ and the morphism $\psi$ relates these structures in a special way.

\begin{examples}\label{meas_exmp}
\hspace{0.1cm}
\begin{enumerate}
\item If $\mathcal C = \mathbf{Vect}_\mathbbm{k}$ for a field $\mathbbm{k}$ and $\Omega$
	is as in Examples~\ref{comod}, (\ref{ExampleAlgebraOmega}) and~(\ref{ExampleUnitalAlgebraOmega}), we recover the traditional definition of a measuring of (now not necessarily associative)
	algebras \cite[Chapter VII]{Sweedler};
\item If $\mathcal C = \mathbf{Sets}$ with $\otimes = \times$, then
	a measuring $\psi \colon P \otimes A \to B$ is just a map $\psi \colon P \times A \to B$
	such  that for every fixed $p\in P$ the map $\psi(p, -)$ is an ordinary $\Omega$-magma homomorphism, i.e. a map compatible with all operations from the set $\Omega$.
\end{enumerate}
\end{examples}

For a monoidal category~$\mathcal C$ denote by $\mathsf{Mon}(\mathcal C)$ and $\mathsf{Comon}(\mathcal C)$
the categories of monoids and comonoids in~$\mathcal C$, respectively.

Recall that if the category $\mathcal C$ is braided, then $\mathsf{Mon}(\mathcal C)$ is a monoidal category too.
Objects of the category $\mathsf{Comon}(\mathsf{Mon}(\mathcal C))$ (which is isomorphic to $\mathsf{Mon}(\mathsf{Comon}(\mathcal C))$) are called \textit{bimonoids} in $\mathcal C$.

If $P$ is a bimonoid, then the category ${}_P\mathsf{Mod}$ of left $P$-modules
is a subcategory of $\mathsf{TensMor}(P)$ that inherits from $\mathsf{TensMor}(P)$ the monoidal structure.
An $\Omega$-magma in ${}_P\mathsf{Mod}$ is called a \textit{$P$-module $\Omega$-magma}
and the corresponding morphism $\psi \colon P\otimes A \to A$ is called a \textit{$P$-action} on $A$.

\begin{examples}\label{mod_alg_exmp}
\hspace{0.1cm}
\begin{enumerate}
\item In the case $\mathcal C = \mathbf{Vect}_\mathbbm{k}$ for a field $\mathbbm{k}$ and $\Omega$
is as in Examples~\ref{comod}, (\ref{ExampleAlgebraOmega}) and~(\ref{ExampleUnitalAlgebraOmega}), we recover the traditional definition of a  (now not necessarily associative) module algebra over a bialgebra;
\item Recall that, since comonoids in $\mathbf{Sets}$ are  trivial, all bimonoids in $\mathbf{Sets}$
	are just monoids. Hence if $\mathcal C = \mathbf{Sets}$, then
	an action $\psi \colon P \times A \to A$ is just such an action of $P$ on $A$ by $\Omega$-magma homomorphisms.
\end{enumerate}
\end{examples}

\subsection{Comeasurings}\label{SubsectionComeasurings}

Dually,  let $Q$ be a monoid in $\mathcal C$ with a multiplication
$\mu \colon Q \otimes Q \to Q$ and a unit $u \colon \mathbbm{1} \to Q$. Consider the monoidal category $\mathsf{MorTens}(Q)$ where the objects are morphisms $ A \to B \otimes Q$ and morphisms between objects $\rho_1 \colon 
A_1 \to B_1 \otimes Q$ and $\rho_2 \colon 
A_2 \to B_2 \otimes Q$
are pairs of morphisms $\alpha \colon A_1 \to A_2$ and $\beta \colon B_1 \to B_2$ making the diagram below
commutative:

$$\xymatrix{  A_1 \ar[d]^{ \alpha }\ar[r]^(0.4){\rho_1} & B_1 \otimes Q \ar[d]^{\beta \, \otimes \, \id_Q} \\
	A_2 \ar[r]^(0.4){\rho_2} & B_2 \otimes Q
}$$

The monoidal product $\rho_1 \mathbin{\widetilde\otimes} \rho_2 \colon A_1 \otimes A_2 \to (B_1 \otimes B_2) \otimes Q$
of objects $\rho_1 \colon A_1 \to B_1  \otimes Q$ and $\rho_2 \colon A_2 \to B_2  \otimes Q$
in $\mathsf{MorTens}(Q)$ is defined as the composition of the morphisms below:

$$\xymatrix{ A_1 \otimes A_2 \ar[r]^(0.35){\rho_1 \otimes \rho_2} &
	(B_1 \otimes Q) \otimes (B_2 \otimes Q) \ar[rr]^{\id_{B_1} \, \otimes \,  c_{Q, B_2} \, \otimes \,  \id_{Q}}
	&\qquad & (B_1 \otimes B_2) \otimes (Q \otimes Q)   	\ar[d]^{\id_{B_1 \otimes B_2} \, \otimes \, \mu}
	\\ & & & (B_1 \otimes B_2) \otimes Q }$$

The monoidal unit of $\mathsf{MorTens}(Q)$ is the composition $ \mathbbm{1}  \xrightarrow{u}
Q \mathop{\widetilde{\to}} \mathbbm{1} \otimes Q$.
The axioms of a monoidal category for $\mathsf{MorTens}(Q)$ are consequences of those for $\mathcal C$ and the fact that $Q$ is a monoid.

A \textit{comeasuring} of $\Omega$-magmas is an $\Omega$-magma $\rho \colon A \to B \otimes Q$ in the category $\mathsf{MorTens}(Q)$. Note that the structure of an $\Omega$-magma on $\rho$ endows the objects $A$ and $B$ with structures
of $\Omega$-magmas in $\mathcal C$ and the morphism $\rho$ relates these structures in a special way.

\begin{example}\label{ExampleComeasSets}
	If $\mathcal C = \mathbf{Sets}$, then
	a comeasuring $\rho \colon A \to B \times Q$ is just a pair of maps $\rho_0 \colon A \to B$
	and $\rho_1 \colon A \to Q$ such that $\rho(a)=(\rho_0(a), \rho_1(a))$ for all $a\in A$,
	the map $\rho_0$ is an $\Omega$-magma homomorphism and $\rho_1$ defines on $A$
	a $Q$-grading: $A = \bigsqcup\limits_{q\in Q} A^{(q)}$ where $\rho_0(a)=q$
	for $a\in A^{(q)}$ and $q\in Q$ and all the operations from $\Omega$
	are compatible with this grading.
	If $\Omega$ is the one from Examples~\ref{comod} (\ref{ExampleUnitalAlgebraOmega}), and $A$ and $B$ are ordinary monoids, then
a comeasuring $\rho \colon A \to B \times Q$ is just a monoid homomorphism.
\end{example}

If $Q$ is a bimonoid, then the category $\mathsf{Comod}^Q$ of right $Q$-comodules
is a subcategory of $\mathsf{MorTens}(Q)$ that inherits from $\mathsf{MorTens}(Q)$ the monoidal structure.
An $\Omega$-magma in $\mathsf{Comod}^Q$ is called a \textit{$Q$-comodule $\Omega$-magma}
and the corresponding morphism $\rho \colon  A \to A \otimes Q$ is called a \textit{$Q$-coaction} on $A$.

\begin{examples}\label{comodule} 
\hspace{0.1cm}
\begin{enumerate}
\item In the case $\mathcal C = \mathbf{Vect}_\mathbbm{k}$ for a field $\mathbbm{k}$ and $\Omega$ is as in Examples~\ref{comod}, (\ref{ExampleAlgebraOmega}) and~(\ref{ExampleUnitalAlgebraOmega}), we recover the traditional definition of a  (now not necessarily associative) comodule algebra over a bialgebra;
\item\label{ExampleCoactionSets} If $\mathcal C = \mathbf{Sets}$,
	then a coaction $\rho \colon  A \to A \times Q$ is just a map $a\mapsto (a, \rho_1(a))$
	where $\rho_1$ defines a $Q$-grading on $A$.
	In particular, when $\Omega$
	is as in Examples~\ref{comod} (\ref{ExampleUnitalAlgebraOmega}) and $A$ is an ordinary monoid, then
	 $\rho_1 \colon A \to Q$ is just a monoid homomorphism.
\end{enumerate} 
\end{examples}

\subsection{Hopf monoids}\label{SubsectionHopfMonoids}

Recall that if $(A,\mu_A, u_A)$ is a monoid and $(C,\Delta_C,\varepsilon_C)$ is a comonoid in a monoidal category $\mathcal C$,
then the set $\mathcal C (C,A)$ of all morphisms $C\to A$ in $\mathcal C$ admits a structure of an ordinary monoid:
the multiplication is defined by
$$\varphi * \psi := \mu_A (\varphi \otimes \psi)\Delta_C \text{ for all } \varphi,\psi \in \mathcal C (C,A)$$
and $u_A\varepsilon_C$ is the identity element.
The monoid $\mathcal C (C,A)$ is called the \textit{convolution monoid}.

A bimonoid $H$ in a braided monoidal category $\mathcal C$ is called a \textit{Hopf monoid}
if $\id_H \in \mathcal C (H,H)$ admits an inverse $S \colon H \to H$, which is called the \textit{antipode}. We denote the category of Hopf monoids in $\mathcal C$ by $\mathsf{Hopf}(\mathcal C)$.

\begin{examples}
\hspace{0.1cm}
\begin{enumerate}
\item	Hopf monoids in $\mathbf{Vect}_\mathbbm{k}$, where $\mathbbm{k}$ is a field, are exactly Hopf algebras over~$\mathbbm{k}$;
\item Hopf monoids in $\mathbf{Sets}$ are exactly groups.
\end{enumerate}
\end{examples}

Denote by $c_{X,Y} \colon X \otimes Y \mathrel{\widetilde\to} Y \otimes X$ the braiding in $\mathcal C$.

For a monoid $(A,\mu_A, u_A)$ denote by $A^{\left(\mathrm{op}^n\right)}$, where $n \in \mathbb Z$,
the monoid  $(A,\mu_A(c_{A,A})^n, u_A)$. Analogously, for a comonoid $(C,\Delta_C,\varepsilon_C)$ denote by $C^{\left(\mathrm{cop}^n\right)}$, where $n \in \mathbb Z$, the comonoid  $(C,(c_{C,C})^n\Delta_C,\varepsilon_C)$.

If $B$ is a bimonoid in $\mathcal C$, then $B^{\mathrm{op}^n, \mathrm{cop}^{-n}}$ is a bimonoid too.
(By the induction argument, it is sufficient to check this only for $n=\pm 1$.)

Standard convolution techniques (see e.g.~\cite[Section 4.2]{DNR} or~\cite[Lemma 35, Proposition 36]{Porst1}) combined with a diagram chasing shows that $S \colon H \to H^\mathrm{op}$ is a monoid homomorphism and $S \colon H^\mathrm{cop} \to H$ is a comonoid homomorphism.
Moreover, every bimonoid homomorphism of Hopf monoids commutes with (or preserves) the antipode.
Finally, $H^{\mathrm{op}^n, \mathrm{cop}^{-n}}$ is a Hopf monoid with the same antipode $S$ for every $n\in\mathbb Z$.  

\section{Existence theorems for supports and universal coacting bi- and Hopf monoids}\label{SectionSupportCoactingExistence}

The aim of the next sections is to identify the sufficient conditions on the base category~$\mathcal C$ to ensure the existence
of supports and the universal coacting bi- and Hopf monoids.

\subsection{Monomorphisms and epimorphisms} In order to formulate the conditions, we first need to recall some definitions and results related to monomorphisms and epimorphisms. The details can be found e.g. in~\cite{AHS}.

A monomorphism $i \colon A \rightarrowtail B$  is called
\begin{itemize}
 \item \textit{regular} if $i$ is an equalizer of some morphisms $f_1, f_2 \colon B \to C$; 
\item \textit{strong} if for every commutative square
\begin{equation}\label{EqEpiMonoDiagonal}
\xymatrix{ P \ar@{->>}[r]^\pi \ar[d]_f & Q \ar[d]^g \ar@{-->}[ld]_t \\
	A\ \ar@{>->}[r]^i & B
}
\end{equation}
 where $\pi$ is an epimorphism, there exists a diagonal fill-in, i.e. a morphism $t \colon Q \to A$ such that $i t = g$
and $t \pi = f$ (obviously, such $t$ is unique);
\item \textit{extremal} if for every factorization $i=f\pi$, where $\pi$ is an epimorphism, $\pi$ is in fact an isomorphism.
\end{itemize}

Regular, strong and extremal epimorphisms are introduced in the dual way.

Every regular monomorphism is strong, every strong monomorphism is extremal. The identity isomorphism is regular, a composition of two strong monomorphisms is again strong, a limit (=intersection) of strong subobjects is again a strong subobject (see also Remark~\ref{RemarkIntersections} below). If in a pullback $$\xymatrix{ P\strut \ar[r]^{t} \ar@{>->}[d]_{h} \ar@{}[rd]|<{\pullback} & A\strut \ar@{>->}[d]^f \\
	C \ar[r]^g & B}
$$  $f$ is a strong monomorphism, then $h$ is a strong monomorphism too.

A category is called (Epi, ExtrMono)-\textit{structured} if for every morphism $f$ there exists an epimorphism $\pi$
and an extremal monomorphism $i$ such that $f=i\pi$ and every commutative square~\eqref{EqEpiMonoDiagonal}, where $i$ is  an extremal monomorphism and $\pi$ is an epimorphism, has a diagonal fill-in. In (Epi, ExtrMono)-structured categories all extremal monomorphisms are strong. (ExtrEpi, Mono)-structured categories are introduced in the dual way.

A category $\mathcal C$ is \textit{wellpowered} if for every object $M$ in $\mathcal C$ the set of equivalence classes of monomorphisms to $M$ is a small\footnote{Here we may either assume that we have a fixed universe (see e.g.~\cite[Chapter 1, Section 6]{MacLaneCatWork}) of sets that we call \textit{small} or
	that we work in naive set theory. In the latter case \textit{sets} are arbitrary classes and \textit{small sets} are classes that belong to other classes. If we fix a universe, then in the definition of a (co)wellpowered category we just require that for any object $M$ there exists a bijection between the set of equivalence classes of monomorphisms to (resp., epimorphisms from) $M$ and some member of the universe.} set. Dually, a category $\mathcal C$ is \textit{cowellpowered} if
for every object $M$ in $\mathcal C$ the set of equivalence classes of epimorphisms from $M$ is a small set.

\subsection{Conditions on the base category}\label{SubsectionSupportCoactingConditions}

In this section we list conditions on the base category that make the constructions of Section~\ref{SectionSupportCoactingExistence} possible. Namely, $\mathcal C$ will be a monoidal category satisfying some of the following properties:
\begin{enumerate}
	\item\label{PropertySmallLimits} there exist all small limits in $\mathcal C$;
	\item\label{PropertyFiniteAndCountableColimits} there exist finite and countable colimits in $\mathcal C$;
	\item\label{PropertyEpiExtrMonoFactorizations} $\mathcal C$ is (Epi, ExtrMono)-structured;
	\item\label{PropertySubObjectsSmallSet} $\mathcal C$ is wellpowered;	
	\myitem[4*]\label{PropertyFactorObjectsSmallSet} $\mathcal C$ is cowellpowered;
	\item\label{PropertyMonomorphism} for every monomorphism $f$ and every object $M$
	both
	$f \otimes \id_M$ and $\id_M \otimes f$ are monomorphisms too; 
	\myitem[5a]\label{PropertyExtrMonomorphism} for every extremal monomorphism $f$ the morphism
	$f \otimes f$ is an extremal monomorphism too; 
	\myitem[5*]\label{PropertyEpimorphism} for every epimorphism $f$ and every object $M$
	both
	$f \otimes \id_M$ and
	$\id_M \otimes f$ are epimorphisms too; 
	\item\label{PropertyLimitsOfSubobjectsArePreserved} for every object $M$ the functor $ M \otimes (-)$ preserves limits (= intersections) of extremal subobjects in $ \mathcal C$ (see Remark~\ref{RemarkIntersections} below);
	\item\label{PropertyTensorPullback}
	for every object~$M$ the functor
	$ M \otimes (-)$
	preserves preimages, i.e.
	for every
	pullback $$\xymatrix{ P\strut \ar[r]^{t} \ar@{>->}[d]_{h} \ar@{}[rd]|<{\pullback} & A\strut \ar@{>->}[d]^f \\
		C \ar[r]^g & B}
	$$
	where $f$ is an arbitrary monomorphism and $g$ is an arbitrary morphism having the same codomain $B$
	(recall that in this case $h$ is automatically a monomorphism too) the diagram below is a pullback too:
	$$
	\xymatrix{ M \otimes P
		 \ar[r]^{\id_M \otimes\, t} \ar@{->}[d]_{\id_M \otimes\,  h} \ar@{}[rd]|<{\pullback} & M \otimes A
		 \ar@{->}[d]^{\id_M \otimes\,  f} \\
		M \otimes C \ar[r]^{\id_M \otimes\,  g} & M \otimes B}
	$$
	\item\label{PropertySwitchProdTensorIsAMonomorphism}
	for any nonempty small set $\Lambda$ and any objects $M$ and $A_\alpha$, $\alpha \in \Lambda$, the morphism $$\xymatrix{M \otimes \prod\limits_{\alpha\in\Lambda}A_\alpha \ar[rr]^(0.45){({ \id_M} \otimes {\pi_\alpha})_{\alpha\in\Lambda}}
		& \qquad\quad& \prod\limits_{\alpha\in\Lambda} (M \otimes A_\alpha),}$$ where $\pi_\alpha$ is the projection from $\prod\limits_{\alpha\in\Lambda}A_\alpha$ to $A_\alpha$, $\alpha \in \Lambda$,
	is a monomorphism;	
	\item\label{PropertyEqualizers}
	for every object $M$ the functor $ M \otimes (-)$
	preserves all equalizers;
	\item\label{PropertyFreeMonoid} the forgetful functor $\mathsf{Mon}(\mathcal C) \to \mathcal C$
	has a left adjoint $\mathcal F \colon  \mathcal C \to \mathsf{Mon}(\mathcal C)$.
\end{enumerate}
\begin{remark}\label{RemarkIntersections}
		Properties~\ref{PropertySmallLimits} and~\ref{PropertySubObjectsSmallSet} imply that there exist limits of any families of subobjects,
		i.e. if $\varphi_\alpha \colon A_\alpha \rightarrowtail B$ are monomorphisms
		for some set $\Lambda$ and objects $B$, $A_\alpha$, where $\alpha \in \Lambda$,
		then there exists $\lim T$ where $T \colon \Lambda \cup \lbrace 0 \rbrace \to \mathcal C$, $\Lambda \cup \lbrace 0 \rbrace$
		is the category with the set of objects $\Lambda \cup \lbrace 0 \rbrace$ and
		either only the arrows $\alpha \to 0$ or, in addition, some of arrows $\alpha \to \beta$
		such that $\varphi_\beta = \varphi_{\alpha\beta}\varphi_\alpha$ for some morphism $\varphi_{\alpha\beta} \colon A_\alpha \to A_\beta$  (the resulting $\lim T$ will not depend on whether we include $\alpha \to \beta$ or not)
		  and the functor $T$ is defined as follows: $T\alpha = A_\alpha$, $T0=B$, $T(\alpha \to 0) = \varphi_\alpha$,
		  $T(\alpha \to \beta) =\varphi_{\alpha\beta}$ for all $\alpha,\beta \in \Lambda$.
	
\end{remark}	
\begin{proposition}\label{PropositionPropertyEquilizers}
\begin{enumerate}
\item\label{RemarkPropertyEpiExtrMonoFactorizations} Property~\ref{PropertyEpiExtrMonoFactorizations}
	follows from Properties~\ref{PropertySmallLimits} and~\ref{PropertySubObjectsSmallSet};
\item Property~\ref{PropertyEqualizers} follows from 	Properties~\ref{PropertySmallLimits}, \ref{PropertyTensorPullback}  and~\ref{PropertySwitchProdTensorIsAMonomorphism}.
\end{enumerate}
\end{proposition}
\begin{proof}
1) By \cite[Proposition 12.5]{AHS} any wellpowered small complete category is strongly complete. Now \cite[Corollary 14.21]{AHS} implies that such a category is both (Epi, ExtrMono)- and (ExtrEpi, Mono)-structured, as desired.

2) Recall that an equalizer $h$ of morphisms $f,g \colon A \to B$ can be calculated via the pullback
	\begin{equation*}
		\xymatrix{ P\strut\ \ar@{>->}[r]^{h} \ar@{>->}[d]_{h} \ar@{}[rd]|<{\pullback} & A\strut \ar@{>->}[d]^{(\id_A, f)} \\
			A\ \ar@{>->}[r]^(0.4){(\id_A,g)} & A\times B}\end{equation*}
	Property~\ref{PropertyTensorPullback} implies that
	$$	\xymatrix{M\otimes P\ar@{->}[rr]^{ \id_M \otimes\, h} \ar@{->}[d]_{ \id_M \otimes\, h} \ar@{}[rd]|<{\pullback} & & M\otimes A \ar@{->}[d]^{ \id_M \otimes\, (\id_A, f)} \\
		M\otimes A \ar@{->}[rr]^(0.4){ \id_M \otimes\, (\id_A,g)} & & M\otimes (A\times B)}$$
	is a pullback.
	
	Note that
	$$({ \id_M} \otimes  {\pi_A} )({ \id_M} \otimes (\id_A,f) )={ \id_M} \otimes {\id_A},$$
	$$({ \id_M} \otimes {\pi_B})({ \id_M} \otimes (\id_A,f))={ \id_M} \otimes f,$$
	$$({ \id_M} \otimes {\pi_A})({ \id_M} \otimes (\id_A,g))={ \id_M} \otimes {\id_A}$$
	and $$({ \id_M} \otimes {\pi_B})({ \id_M} \otimes (\id_A,g))={ \id_M} \otimes  g.$$
	Hence for every pair of morphisms
	$\alpha, \beta \colon Q \to M\otimes A$
	the equality $$({ \id_M} \otimes (\id_A,f))\alpha =({ \id_M} \otimes (\id_A,g))\beta$$ holds
	if and only if $$({ \id_M} \otimes f)\alpha  = ({ \id_M} \otimes g )\alpha \text{ and }\alpha = \beta.$$
	The ``if'' part is a consequence of the fact that $$({ \id_M} \otimes {\pi_A},
	{ \id_M} \otimes {\pi_B}) \colon M\otimes (A\times B) \rightarrowtail (M\otimes A)\times (M\otimes B)$$ is a monomorphism by Property~\ref{PropertySwitchProdTensorIsAMonomorphism}.
	
	Therefore, $ {\id_M} \otimes h$ is indeed an equalizer of
	${ \id_M} \otimes f$ and ${ \id_M} \otimes g $.
\end{proof}	

\begin{example}\label{ExampleCategoriesThatSatisfyProperties}
 The basic examples of categories $\mathcal C$ satisfying Properties~\ref{PropertySmallLimits}--\ref{PropertyFreeMonoid},
	\ref{PropertyFactorObjectsSmallSet}, \ref{PropertyEpimorphism} and \ref{PropertyExtrMonomorphism} we keep in mind are $\textbf{Sets}$ (with the Cartesian monoidal product), $\textbf{Sets}^\mathrm{op}$  (with the same monoidal product as in $\textbf{Sets}$, which becomes the co-Cartesian monoidal product with respect to $\textbf{Sets}^\mathrm{op}$), $\textbf{Vect}_\mathbbm{k}$ and $\textbf{Vect}_\mathbbm{k}^\mathrm{op}$ for a field $\mathbbm{k}$.
	In all these examples all monomorphisms and epimorphisms are extremal.
	 Property~\ref{PropertyFreeMonoid} is equivalent to the existence of the free monoid of a set
	for $\textbf{Sets}$, the free (or tensor) algebra of a vector space for $\textbf{Vect}_\mathbbm{k}$ and 
	the cofree coalgebra  of a vector space for $\textbf{Vect}_\mathbbm{k}^\mathrm{op}$. Property~\ref{PropertyFreeMonoid} holds trivially in $\textbf{Sets}^\mathrm{op}$, since all comonoids in $\textbf{Sets}$ are just sets endowed with diagonal maps.
	More examples of categories satisfying Properties~\ref{PropertySmallLimits}--\ref{PropertyFreeMonoid},
	\ref{PropertyFactorObjectsSmallSet}, \ref{PropertyEpimorphism} and \ref{PropertyExtrMonomorphism} will be provided in Section~\ref{ExamplesBaseCaategory}.    
\end{example}

The proof of Proposition~\ref{PropositionMonoidEpiMono} below is straightforward since if $\mathcal C$ satisfies Property~\ref{PropertyEpiExtrMonoFactorizations}, then every extremal monomorphism is strong:
\begin{proposition}\label{PropositionMonoidEpiMono}
	Suppose a monoidal category $\mathcal C$ satisfies Properties~\ref{PropertyEpiExtrMonoFactorizations}
	and~\ref{PropertyEpimorphism}.
	Let $f \colon M \to N$ be a monoid homomorphism for monoids $M$ and $N$ in $\mathcal C$
	and let $f = i \pi$ for an epimorphism $\pi \colon M \twoheadrightarrow L$, an extremal monomorphism $i \colon L \rightarrowtail N$	and an object $L$ in $\mathcal C$.
	Then $L$ admits a unique monoid structure making  $\pi$ and $i$ monoid homomorphisms.
\end{proposition}

\subsection{Limits and colimits in $\mathsf{Mon}(\mathcal C)$ and reflectivity of $\mathsf{Hopf}(\mathcal C)$ in $\mathsf{Bimon}(\mathcal C)$}
\label{Subsection(Co)limitsMonHopfMonReflectivity}

Let $T \colon J \to \mathsf{Mon}(\mathcal C)$ be a functor where $J$ is a category and $\mathcal C$ is a monoidal category.
Suppose that $M := \lim UT$ is the limit of $UT$ in $\mathcal C$ where $U\colon \mathsf{Mon}(\mathcal C) \to \mathcal C$
is the forgetful functor. Then the unique morphisms $\mu_M$
and $u_M$ making the diagrams below commutative for every object $j$ of $J$ ($\varphi_j$ is the limiting cone) turn $M$ into
a monoid and, therefore, into the limit of $T$ in $\mathsf{Mon}(\mathcal C)$:

$$\xymatrix{M\otimes M \ar@{-->}[d]^{\mu_M} \ar[rr]^{\varphi_j \otimes \varphi_j} & & Tj \otimes Tj \ar[d]^{\mu_{Tj}} \\
 M \ar[rr]^{\varphi_j} & & Tj } \qquad 
\xymatrix{ 	& \mathbbm{1} \ar@{-->}[ld]_{u_M} \ar[rd]^{u_{Tj}} &  \\
	 M \ar[rr]^{\varphi_j} & & Tj  
}
$$

In other words, the forgetful functor $U$ creates limits.

Now consider colimits in $\mathsf{Mon}(\mathcal C)$:

\begin{theorem}\label{TheoremColimitsMon} Let $\mathcal C$ be a braided monoidal category satisfying Properties~\ref{PropertySmallLimits},
 \ref{PropertyEpiExtrMonoFactorizations}, 	\ref{PropertyFactorObjectsSmallSet}, 
  \ref{PropertyEpimorphism}, \ref{PropertyFreeMonoid} of Section~\ref{SubsectionSupportCoactingConditions}.
	Let $T \colon J \to \mathsf{Mon}(\mathcal C)$ be a functor where $J$ is a category.
	Suppose there exists $N := \colim UT$ and $\varphi_j$ is the corresponding colimiting cocone. (The colimit is taken in $\mathcal C$.)
	Then there exists a monoid homomorphism $\pi \colon {\mathcal F}N \twoheadrightarrow P$, which is, in addition,
	an epimorphism in $\mathcal C$, such that the composition 
	$$\xymatrix{ Tj \ar[r]^{\varphi_j} & N \ar[r]^{\eta_N} & {\mathcal F}N \ar@{->>}[r]^{\pi} & P  }$$
	is a colimiting cocone of $T$ in $\mathsf{Mon}(\mathcal C)$.
	(Here $\eta$ is the unit of the adjunction ${\mathcal F} \dashv U$.)
\end{theorem}
\begin{proof}Consider the set $(\psi_\alpha)_{\alpha\in\Lambda}$
	of all monoid homomorphisms $\psi_\alpha \colon {\mathcal F}N \to P_\alpha$ that are epimorphisms in $\mathcal C$
	such that the composition $\psi_\alpha \eta_N \varphi_j$ is a monoid homomorphism for every $j$.
	By Property~\ref{PropertyFactorObjectsSmallSet} we may assume that the set $\Lambda$ is small.
	Introduce on $\Lambda$ a partial ordering by $\alpha \preccurlyeq \beta$
	if $\psi_\alpha =\psi_{\beta\alpha}\psi_\beta$ for some morphism $\psi_{\beta\alpha}$.
	Note that since $\psi_\beta$ is an epimorphism, every such morphism $\psi_{\beta\alpha}$
	is unique. Moreover, by Property~\ref{PropertyEpimorphism}, $\psi_\beta \otimes \psi_\beta$
	is an epimorphism too. Now the fact that both $\psi_\alpha$ and  $\psi_\beta$ are monoid homomorphisms
	implies that $\psi_{\beta\alpha}$ is a monoid homomorphism too.
	Denote by $P_0$ the limit of $P_\alpha$ in $\mathcal C$ and by $i$ and $\pi$, respectively, the extremal monomorphism and the epimorphism from the (Epi, ExtrMono)-factorization of the comparison morphism between $\psi_\alpha$ and the limiting cone:
	$$\xymatrix{ Tj \ar[r]^{\varphi_j} & N \ar[r]^{\eta_N} & {\mathcal F}N \ar@{->>}[r]^{\psi_\alpha} \ar@{->>}[d]^{\pi} & P_\alpha  \\
	                                   &                   &    P\               \ar@{>->}[r]^i             &  P_0 \ar[u] }$$
                                   
    By the remarks made before the theorem, $P_0$ is a limit in $\mathsf{Mon}(\mathcal C)$ too. By Proposition~\ref{PropositionMonoidEpiMono}  the object $P$ bears a unique structure of a monoid and $\pi$ and $i$ are monoid homomorphisms.  Hence all the squares and triangles in the diagrams below, except, possibly, the left ones, are commutative:
     $$\xymatrix{Tj \otimes Tj \ar[d] \ar[rr]^{{\eta_N\varphi_j} \otimes {\eta_N\varphi_j}} & & {\mathcal F}N \otimes {\mathcal F}N \ar[d] \ar@{->>}[r]^{\pi\otimes \pi} & P \otimes P\ar[d] \ar[r]^{i\otimes i} & P_0 \otimes P_0 \ar[d] \ar[r] & P_\alpha \otimes P_\alpha \ar[d] \\
    Tj \ar[rr]^{\eta_N\varphi_j} & & {\mathcal F}N \ar@{->>}[r]^{\pi}  & P\  \ar@{>->}[r]^{i}  & P_0  \ar[r]  & P_\alpha } $$

 $$\xymatrix{Tj \ar[rr]^{\eta_N\varphi_j} & & {\mathcal F}N \ar@{->>}[r]^{\pi}  & P\  \ar@{>->}[r]^{i}  & P_0  \ar[r]  & P_\alpha \\
 & & \mathbbm{1} \ar[llu] \ar[u] \ar[ru] \ar[rru] \ar[rrru] & & &} $$
Now the fact that $P_0$ is a limit and $i$ is a monomorphism implies that $\pi \eta_N \varphi_j$ is a monoid homomorphism
for every $j$.

In other words, $P$ belongs to the set $\lbrace P_\alpha \mid \alpha \in \Lambda \rbrace$
    and corresponds to the global maximum of~$\Lambda$. Hence $P$ is the limit of $P_\alpha$ in $\mathcal C$
    and $i$ is in fact an isomorphism.   

    Suppose now $\tau_j \colon Tj \to Q$ is a cocone in $\mathsf{Mon}(\mathcal C)$.
    Then there exists a unique monoid homomorphism $\tau \colon {\mathcal F}N \to Q$ making the left triangle of the diagram below commutative:
    	$$\xymatrix{ Tj \ar[r]^{\varphi_j} \ar[rd]_{\tau_j} & N \ar[r]^{\eta_N} & {\mathcal F}N \ar@{->>}[r]^{\pi}
    		\ar[ld]^{\tau}
    		 \ar@{->>}[d]^{\pi_0} & P \ar[ld]^s  \\
                                     &    Q              & \ar@{>->}[l]^{i_0} \ Q_0            &  }$$
                                 Consider the (Epi, ExtrMono)-factorization
                                 $\tau = i_0\pi_0$.
Again, by Proposition~\ref{PropositionMonoidEpiMono}, the corresponding object $Q_0$ bears a unique structure of a monoid and $\pi_0$ and $i_0$ are monoid homomorphisms. Note that $\tau_j = i_0 \pi_0 \eta_N \varphi_j$
and therefore $\pi_0 \eta_N \varphi_j$ are monoid homomorphims too for every $j$. Hence $Q_0$
belongs to the set $\lbrace P_\alpha \mid \alpha \in \Lambda \rbrace$. Since $P$ corresponds to the global maximum of~$\Lambda$,
there exists the comparison morphism $s \colon P \to Q_0$, which, being one of $\psi_{\beta\alpha}$ above, is a monoid homomorphism too. Hence the composition $i_0 s$ is a monoid homomorphism and $P$ is indeed the colimit of $T$ in $\mathsf{Mon}(\mathcal C)$. The uniqueness of $i_0 s$ follows from the fact that $\pi$ is an epimorphism.
\end{proof}	

Now we prove the following theorem inspired by~\cite{CH, Porst1,Takeuchi}:
\begin{theorem}\label{TheoremBimonHopfLeftAdjoint}
Let $\mathcal C$ be a braided monoidal category satisfying Properties~\ref{PropertySmallLimits}--\ref{PropertyEpiExtrMonoFactorizations}, 	\ref{PropertyFactorObjectsSmallSet},  \ref{PropertyEpimorphism}, \ref{PropertyFreeMonoid} of Section~\ref{SubsectionSupportCoactingConditions}.
		Then the forgetful functor $\mathbf{Hopf}(\mathcal C) \to \mathbf{Bimon}(\mathcal C)$
	admits a left adjoint functor $H_l \colon  \mathbf{Bimon}(\mathcal C) \to \mathbf{Hopf}(\mathcal C)$.
\end{theorem}
\begin{proof}
	For a given bimonoid $B$ consider the coproduct $\tilde B = \coprod\limits_{n=0}^\infty B^{\mathrm{op}^{-n},\mathrm{cop}^n}$
in $\mathsf{Mon}(\mathcal C)$, which exists by  Property~\ref{PropertyFiniteAndCountableColimits} and Theorem~\ref{TheoremColimitsMon}. Note that being a coproduct, the monoid $\tilde B$ bears a natural structure of a comonoid making $\tilde B$ the coproduct of $B^{\mathrm{op}^{-n},\mathrm{cop}^n}$ in $\mathsf{Bimon}(\mathcal C)$.

We regard $(-)^{\mathrm{op},\mathrm{cop}^{-1}}$ as an endofunctor on $\mathsf{Bimon}(\mathcal C)$ that twists both the multiplication and the comultiplication and changes neither morphisms nor objects. Being an automorphism of the category $\mathsf{Bimon}(\mathcal C)$, the functor $(-)^{\mathrm{op},\mathrm{cop}^{-1}}$ preserves all limits and colimits.
Hence for any bimonoids $B_\alpha$ we may identify~$\coprod\limits_\alpha \left( B_\alpha\right)^{\mathrm{op},\mathrm{cop}^{-1}}$ with~$\left(\coprod\limits_\alpha B_\alpha \right)^{\mathrm{op},\mathrm{cop}^{-1}}$.
	Define the morphism $S$ on $\tilde B$ as the unique bimonoid homomorphism making
	the diagram below commutative:
	$$\xymatrix{ \coprod\limits_{n=0}^\infty B^{\mathrm{op}^{-n},\mathrm{cop}^n}  \ar@{-->}[r]^(0.4)S & \left(\coprod\limits_{n=0}^\infty  B^{\mathrm{op}^{-n},\mathrm{cop}^n}\right)^{\mathrm{op},\mathrm{cop}^{-1}} \\
		\ar[u]_(0.45){i_n}   B^{\mathrm{op}^{-n},\mathrm{cop}^n} \ar@{=}[r] & \ar[u]_{{(i_{n+1})}^{\mathrm{op},\mathrm{cop}^{-1}}}  \left(B^{\mathrm{op}^{-n-1},\mathrm{cop}^{n+1}}\right)^{\mathrm{op},\mathrm{cop}^{-1}} \\
	}$$

	Here $i_n$ are the morphisms $B^{\mathrm{op}^{-n},\mathrm{cop}^n} \to \tilde B$ from the universal property of the coproduct. 
		By the definition, $S$ is a bimonoid homomorphism $\tilde B \to \tilde B^{\mathrm{op},\mathrm{cop}^{-1}}$.

	If $H$ is bimonoid endowed with a bimonoid homomorphism $S_H \colon H \to H^{\mathrm{op},\mathrm{cop}^{-1}}$
	and $\varphi_0 \colon B \to H$ is an arbitary bimonoid homomorphism, then the diagram
	below shows that there exists a unique bimonoid homomorphism $\varphi \colon \tilde B \to H$
	such that $\varphi S = S_H \varphi$ and $\varphi_0 = \varphi i_0$:
		$$\xymatrix{ H^{\mathrm{op}^{-n},\mathrm{cop}^n} \ar[r]^{\left(S_H\right)^n} & H \\
					\left(\coprod\limits_{m=0}^\infty B^{\mathrm{op}^{-m},\mathrm{cop}^m}\right)^{\mathrm{op}^{-n},\mathrm{cop}^n}  \ar@{->}[r]^(0.6){S^n}
			\ar[u]_(0.6){\varphi^{\mathrm{op}^{-n},\mathrm{cop}^n}}	 & \coprod\limits_{m=0}^\infty  B^{\mathrm{op}^{-m},\mathrm{cop}^m} \ar@{-->}[u]_(0.6){\varphi} \\
		\ar[u]_(0.4){(i_0)^{\mathrm{op}^{-n},\mathrm{cop}^n}}   B^{\mathrm{op}^{-n},\mathrm{cop}^n} \ar@{=}[r]
		\ar@/^80pt/[uu]^{\left(\varphi_0\right)^{\mathrm{op}^{-n},\mathrm{cop}^n}}
		 & \ar[u]_(0.4){i_n}  B^{\mathrm{op}^{-n},\mathrm{cop}^n} \\
	}$$
	
	Consider the coequalizer $\gamma \colon \tilde B \to E$ in $\mathcal C$
	of morphisms $S * \id_{\tilde B}$, $\id_{\tilde B} * S$ and $u \varepsilon$
	where $u$ and $\varepsilon$ are, respectively, the unit and the counit of $\tilde B$.
	(The coequalizer exists by Property~\ref{PropertyFiniteAndCountableColimits}.)
	Like in the proof of Theorem~\ref{TheoremColimitsMon}, consider the free monoid ${\mathcal F}E$
	on $E$ and the set $(\psi_\alpha)_{\alpha\in\Lambda}$
	of all monoid homomorphisms $\psi_\alpha \colon {\mathcal F}E \to P_\alpha$ that are epimorphisms in $\mathcal C$
	such that the composition $\psi_\alpha \eta_E \gamma$ is a monoid homomorphism for every $j$.
	By Property~\ref{PropertyFactorObjectsSmallSet} we may assume that the set $\Lambda$ is small.
	Introduce on $\Lambda$ a partial ordering by $\alpha \preccurlyeq \beta$
	if $\psi_\alpha =\psi_{\beta\alpha}\psi_\beta$ for some morphism $\psi_{\beta\alpha}$.
	Note that since $\psi_\beta$ is an epimorphism, every such morphism $\psi_{\beta\alpha}$
	is unique. Again, Property~\ref{PropertyEpimorphism} implies that $\psi_{\beta\alpha}$ is a monoid homomorphism too.
	Denote by $P_0$ the limit of $P_\alpha$ in $\mathcal C$ and by $i$ and $\pi$, respectively, the extremal monomorphism and the epimorphism from the (Epi, ExtrMono)-factorization of the comparison morphism between $\psi_\alpha$ and the limiting cone:
			$$\xymatrix{ \tilde B \ar@<-1ex>[r] \ar[r] \ar@<1ex>[r]  & \tilde B \ar[r]^{\gamma} \ar[rrd]^\theta & E \ar[r]^{\eta_E} & {\mathcal F}E \ar@{->>}[r]^{\psi_\alpha} \ar@{->>}[d]^{\pi} & P_\alpha  \\
		&     &              &    P\               \ar@{>->}[r]^i             &  P_0 \ar[u] }$$
	The same argument as in Theorem~\ref{TheoremColimitsMon} shows that
	$\theta := \pi\eta_E\gamma$ is the coequalizer of $S * \id_{\tilde B}$, $\id_{\tilde B} * S$ and $u \varepsilon$
	among monoid homomomorphisms from $\tilde B$. 
	
	Let $\theta = i_1 \pi_1$ be the (Epi, ExtrMono)-factorization of $\theta$.
	By Proposition~\ref{PropositionMonoidEpiMono}, both $i_1$ and $\pi_1$ are monoid homomorphisms.
	Since $i_1$ is a monomorphism, $$\pi_1 (S * \id_{\tilde B}) = \pi_1 (\id_{\tilde B} * S) = \pi_1 u \varepsilon.$$
	Hence $\pi_1 = \tau_1 \theta$ for some monoid homomorphism $\tau_1$. Then $\theta = i_1\tau_1\theta$ and
	$\pi_1 = \tau_1 i_1 \pi_1$. Recall that $\pi_1$ is an epimorphism and $\theta$ is a coequalizer among monoid homomomorphisms.
	Therefore $i_1 = \tau_1^{-1}$ is an isomorphism and $\theta$ is an epimorphism in $\mathcal C$. 
	
	Diagram chasing dual to that in the proof of Porst's Crucial Lemma~\cite[Lemma 38]{Porst1}
	shows that $$(\theta \otimes \theta) \Delta (S * \id_{\tilde B}) 
	= (\theta \otimes \theta) \Delta (\id_{\tilde B} * S)
	=  (\theta \otimes \theta) \Delta u \varepsilon,$$
	$$\varepsilon (S * \id_{\tilde B}) 
	= \varepsilon (\id_{\tilde B} * S)
	=  \varepsilon u \varepsilon$$
	 (here $\Delta$ is the comultiplication in $\tilde B$)
	 and therefore there exist unique monoid homomorphisms $\varepsilon_P \colon P \to \mathbbm{1}$
	 and $\Delta_P \colon P \to P\otimes P$ making the diagrams below commutative
	 and turning $P$ into a bimonoid:
	 	$$\xymatrix{ \tilde B \ar@<-1ex>[r] \ar[r] \ar@<1ex>[r] & \tilde B \ar@{->>}[rr]^{\theta}\ar[d]^{\Delta} & & P \ar@{-->}[d]^{\Delta_P}  \\
	 	                       &    \tilde B  \otimes \tilde B   \ar@{->>}[rr]^{\theta \otimes \theta}  & & P \otimes P}
 	                       \qquad 
 	                       \xymatrix{ \tilde B \ar@<-1ex>[r] \ar[r] \ar@<1ex>[r] & \tilde B \ar@{->>}[r]^{\theta}\ar[rd]_{\varepsilon} & P \ar@{-->}[d]^{\varepsilon_P}  \\
 	                       	&      &  \mathbbm{1}}$$
     Moreover, there exists a unique monoid homomorphism $S_P \colon P \to P^{\mathrm{op},\mathrm{cop}^{-1}}$ (which is a comonoid homomorphism too since $\theta$ is an epimorphism) such that $S_P \theta = \theta S$:               	
                    $$\xymatrix{ \tilde B \ar@<-1ex>[r] \ar[r] \ar@<1ex>[r] \ar[d]^{S} & \tilde B \ar@{->>}[rr]^{\theta}\ar[d]^{S} & & P \ar@{-->}[d]^{S_P}  \\
    B^{\mathrm{op},\mathrm{cop}^{-1}}    \ar@<-1ex>[r] \ar[r] \ar@<1ex>[r]              	&  B^{\mathrm{op},\mathrm{cop}^{-1}} \ar[rr]^{\theta^{\mathrm{op},\mathrm{cop}^{-1}}}    & & P^{\mathrm{op},\mathrm{cop}^{-1}}}$$    	
	 Since $\theta$ is an epimorphism, $P$ is a Hopf monoid. 
	 The universal property of $\tilde B$ implies that for every Hopf monoid $H$ and every bimonoid homomorphism
	 $\varphi_0 \colon B \to H$ there exists a unique bimonoid homomorphism $\varphi \colon \tilde B \to H$
	 such that $\varphi S = S_H \varphi$ and $\varphi_0 = \varphi i_0$. Hence there exists
	 a unique monoid homomorphism $\psi \colon P \to H$ such that $\varphi = \psi \theta$.
	 Note that $\psi$ is a comonoid homomorphism since $\theta$ is an epimorphism. Furthermore, $\psi$ is a Hopf monoid homomorphism since every bimonoid homomorphism between Hopf monoids preserves the antipode,
	 see Section~\ref{SubsectionHopfMonoids}.
	 Therefore, $\psi$ is the unique Hopf monoid homomorphism $P \to H$ such that
	 $\varphi_0 = \psi \theta i_0$, i.e. we can take $H_l B := P$.
\end{proof}		

\begin{remark} In Proposition~\ref{PropositionMonoidEpiMono} and Theorems~\ref{TheoremColimitsMon} and~\ref{TheoremBimonHopfLeftAdjoint} instead of Property~\ref{PropertyEpimorphism} it is sufficient to require just that for every epimorphism $\varphi$ in~$\mathcal C$ the morphisms $\varphi\otimes \varphi$ and $\varphi\otimes \varphi \otimes \varphi$
	are epimorphisms too.
\end{remark}

\begin{example}\label{ExampleHopfLeftAdjointSets} In the case $\mathcal{C} = \mathbf{Sets}$ the left adjoint functor $H_l$ assigns to each monoid $M$ its \textit{Grothendieck group}, i.e. the group with the same generators and relations as in $M$.
\end{example}

\subsection{Supports of morphisms $A \to B \otimes Q$}\label{SubsectionSupportMorTensAB}

Now we are ready to define supports of morphisms and prove their existence.

Let $\mathcal C$ be a monoidal category.
For given objects $A,B$ in $\mathcal C$ denote by $\mathbf{MorTens}(A,B)$
the comma category $(A\downarrow B\otimes(-))$, i.e.
the category where
\begin{itemize}
	\item the objects are all morphisms $\rho \colon A \to B \otimes Q$ for arbitrary objects $Q$;
	\item the morphisms between $\rho_1 \colon A \to B \otimes Q_1$ and $\rho_2 \colon   A \to  B \otimes Q_2$
	 are morphisms $\tau \colon Q_1 \to Q_2$
making the diagram below commutative:
$$ \xymatrix{	 A \ar[r]^(0.4){\rho_1} \ar[rd]_{\rho_2} &  B \otimes Q_1 \ar[d]^{\id_{ B} \otimes\, \tau} \\
	&  B \otimes Q_2}$$
\end{itemize}

\begin{remark}
	The category $\mathbf{MorTens}(A,B)$ defined above should not be confused with the category
	$\mathsf{MorTens}(Q)$ defined in Section~\ref{SubsectionComeasurings} in order to introduce comeasurings. Both contain
	$\rho \colon A \to B \otimes Q$ as objects, but in $\mathbf{MorTens}(A,B)$ the objects $A$ and $B$ are fixed and 
	in $\mathsf{MorTens}(Q)$ we fix $Q$, the objects $A$ and $B$ may be arbitrary.
\end{remark}	

\begin{definition}
	We say that a morphism $\rho \colon A \to B \otimes Q$ is a \textit{tensor epimorphism}
	if $\rho \in \LIO(\mathbf{MorTens}(A,B))$, i.e. if
	for every $f,g \colon Q \to R$, such that
 \begin{equation*}
		({\id_B} \otimes f)\rho = ({\id_B} \otimes g) \rho,\end{equation*}
	we have $f=g$.
\end{definition}

If there exists $|\rho| \colon  A \to B \otimes \tilde Q$ for some $\rho$, then we call the object $\supp \rho := \tilde Q$
the \textit{support} of $\rho$. From the definition of the absolute value it follows that $\supp \rho$ is defined up to an isomorphism compatible with $|\rho|$.

\begin{examples}\label{support}
\hspace{1cm}
\begin{enumerate}
	\item\label{ExampleVectMorTensSupp} Let $\mathcal C = \mathbf{Vect}_\mathbbm{k}$ and let $\rho \colon A \to B \otimes Q$ be a linear map where $A,B,Q$ are vector spaces over a field $\mathbbm{k}$.
	Choose a basis  $(a_\alpha)_\alpha$ in $A$ and a basis $(b_\beta)_\beta$ in $B$.
	Define $q_{\beta\alpha} \in Q$ by $\rho(a_\alpha)=\sum_{\beta} b_\beta \otimes q_{\beta\alpha}$.
	Applying the elements of the dual vector space $B^*$
	to the left component of $B \otimes Q$, we see that
	the $\mathbbm{k}$-linear span $Q_0$ of all $q_{\beta\alpha}$ is the minimal subspace $Q_0 \subseteq Q$
	such that $\rho(A) \subseteq B \otimes Q_0$. Hence $\rho$ is a tensor epimorphism if and only if $Q=Q_0$.
	Denote by $\rho_0 \colon A \to B \otimes Q_0$ the corestriction of $\rho$ to $B \otimes Q_0$.
	Suppose now that $\rho=(\id_B \otimes \tau)\rho'$ for some vector space $Q'$, tensor epimorphism $\rho' \colon A \to B \otimes Q'$ and a linear map $\tau \colon Q' \to Q$. Define $q'_{\beta\alpha} \in Q'$ by $\rho'(a_\alpha)=\sum_{\beta} b_\beta \otimes q'_{\beta\alpha}$. Then $\tau(q'_{\beta\alpha})=q_{\beta\alpha}$ for all $\alpha$ and $\beta$. Thus
	$\rho_0 = (\id_B \otimes \tau)\rho'$ and $\rho_0 \preccurlyeq \rho'$.
	Therefore, $|\rho| = \rho_0$ and $\supp \rho = Q_0$.
	
\item  Let $\Gamma \colon A=\bigoplus\limits_{g\in G} A^{(g)}$ be a grading
	on an algebra $A$ over a field $\mathbbm{k}$ by a group $G$.
	Then $A$ is a $\mathbbm{k}G$-comodule algebra, where
	$\mathbbm{k}G$ is the group Hopf algebra and the comodule structure $\rho \colon A \to A \otimes \mathbbm{k}G$
	is defined by $\rho(a):=a\otimes g$ for all $a\in A^{(g)}$ and $g \in G$. Let again $\mathcal C = \mathbf{Vect}_\mathbbm{k}$.
	Then $\supp \rho = \langle \supp \Gamma \rangle_\mathbbm{k}$,
	which justifies using the name \emph{support} for $\supp \rho$.

\item  In the case $\mathcal C = \mathbf{Sets}$, $\supp \rho$ is the projection of the image of $\rho$ on the second component
	of $B\times Q$ and $|\rho| \colon A \to B \times (\supp \rho)$ is the corestriction of $\rho$ to $B \times (\supp \rho)$.
\end{enumerate}
\end{examples}

Proposition~\ref{PropositionEpimorphismTensorEpimorphism} establishes a link between tensor epimorphisms and ordinary epimorphisms.

\begin{proposition}\label{PropositionEpimorphismTensorEpimorphism} 
	Let $\rho_1 \colon A \to B \otimes Q_1$ be a tensor epimorphism	in a monoidal category~$\mathcal C$.
	Then a morphism $\tau \colon Q_1 \to Q_2$ is an epimorphism if and only if
	$\rho_2 = ({\id_B}\otimes \tau) \rho_1$ is a tensor epimorphism.
\end{proposition}
\begin{proof} Suppose first that $\rho_2$ is a tensor epimorphism.
	 Let $f,g \colon Q_2 \to Q_3$ be morphisms such that $f\tau = g\tau$.
	 Then $$(\id_B \otimes f)\rho_2 = ({\id_B}\otimes f\tau) \rho_1 = ({\id_B}\otimes g\tau) \rho_1 = (\id_B \otimes g)\rho_2$$
	 and $f=g$:
	\begin{equation*}\xymatrix{
		A \ar[r]^(0.35){\rho_1} \ar[rd]_(0.45){\rho_2} & B \otimes Q_1 \ar[d]^(0.45){{\id_B}\otimes \tau} \\
		& B \otimes Q_2 \ar@<-3pt>[rr]_{{\id_B}\otimes g} \ar@<3pt>[rr]^{{\id_B}\otimes f} & & B \otimes Q_3
}\end{equation*}

Suppose now that $\tau$ is an epimorphism.
Let $f,g \colon Q_2 \to Q_3$ be morphisms such that $(\id_B \otimes f)\rho_2 = (\id_B \otimes g)\rho_2$.
Then 
$$({\id_B}\otimes f\tau) \rho_1 = (\id_B \otimes f)\rho_2 = (\id_B \otimes g)\rho_2 = ({\id_B}\otimes g\tau) \rho_1.$$
Since 	$\rho_1$ is a tensor epimorphism, we have $f\tau = g\tau$ and $f=g$.
\end{proof}

\begin{theorem}\label{TheoremAbsValueSupportExistence}
	Let $\mathcal C$ be a monoidal category satisfying Properties~\ref{PropertySmallLimits}, \ref{PropertySubObjectsSmallSet}--\ref{PropertyLimitsOfSubobjectsArePreserved} and \ref{PropertyEqualizers}. Then for every objects $A,B$ in $\mathcal C$ there exist absolute values of
	 all objects in the category $\mathbf{MorTens}(A,B)$.
	 As a consequence, there exist supports for all morphisms 
	$\rho \colon A \to B \otimes Q$ in~$\mathcal C$. 
\end{theorem}

In order to prove Theorem~\ref{TheoremAbsValueSupportExistence}, consider the following construction.

Let $\rho \colon A \to B \otimes Q$ be a morphism for some objects $A,B,Q$ in $\mathcal C$.
Consider the category~$\mathcal Q(\rho)$ where:
\begin{itemize}
	\item
	the objects
	are pairs $(\rho_1, i_1)$
	where $\rho_1 \colon A \to B \otimes Q_1$ is a morphism
	and $i_1 \colon Q_1 \rightarrowtail Q$ is an extremal monomorphism
	making the diagram below commutative:
	\begin{equation}\label{EqSupportDef}\xymatrix{
		A \ar[r]^(0.4){\rho_1} \ar[rd]_{\rho} & B \otimes Q_1 \ar[d]^{\id_{B} \otimes\, i_1} \\
		& B \otimes Q
	}\end{equation}
	
	\item the morphisms between $$\rho_1 \colon A \to B \otimes Q_1,
	\quad i_1 \colon Q_1 \rightarrowtail Q$$ and $$\rho_2 \colon   A \to  B \otimes Q_2,
	\quad i_2 \colon Q_2 \rightarrowtail Q$$ are morphisms $\tau \colon Q_1 \to Q_2$
	making the diagrams below commutative:
	$$ \xymatrix{	 A \ar[r]^(0.4){\rho_1} \ar[rd]_{\rho_2} &  B \otimes Q_1 \ar[d]^{\id_{ B} \otimes\, \tau} \\
		&  B \otimes Q_2} \qquad \xymatrix{	Q_1\ \ar@{>->}[r]^{i_1} \ar[d]_{\tau} &  Q \\
		Q_2 \ar@{>->}[ur]_{i_2}}$$
	
\end{itemize}

Consider the functor $T \colon \mathcal Q(\rho) \to \mathcal C$
that maps a pair $(A \to  B \otimes Q_1,\  Q_1 \rightarrowtail Q)$
to $Q_1$.  Define a candidate $\mathop{\mathrm{supp}_0} \rho$ for the support of $\rho$ by $\mathop{\mathrm{supp}_0} \rho := \mathop\mathrm{lim} T$. The latter exists by Properties~\ref{PropertySmallLimits} and \ref{PropertySubObjectsSmallSet}. 

\begin{remark}\label{RemarkSupp0Embedding}
	Note that $i_0 \colon \mathop{\mathrm{supp}_0} \rho \rightarrowtail Q$ is an extremal monomorphism, since $\mathop{\mathrm{supp}_0} \rho$ is a limit (=intersection) of extremal subobjects of $Q$. (Recall that by Property~\ref{PropertyEpiExtrMonoFactorizations}, which follows from Properties~\ref{PropertySmallLimits} and \ref{PropertySubObjectsSmallSet}, all extremal monomorphisms are strong.)
\end{remark}

By Property~\ref{PropertyLimitsOfSubobjectsArePreserved}, if $\mathop{\mathrm{supp}_0} \rho \to Q_1$
is the limiting cone for $T$, then $B \otimes (\mathop{\mathrm{supp}_0} \rho) \to B \otimes Q_1$ is the limiting cone
for $B \otimes T(-)$. However, $A\to B \otimes Q_1$ is a cone over $B \otimes T(-)$ too. Hence there exists a unique morphism
$|\rho|_0 \colon A \to B \otimes \mathop{\mathrm{supp}_0} \rho$ between the cones. In particular, the diagram below is commutative: 

\begin{equation*}\xymatrix{
	A \ar[rr]^(0.35){|\rho|_0} \ar[rd]^{\rho_1} \ar@/_2pc/[rdd]_\rho & & B \otimes \mathop{\mathrm{supp}_0} \rho \ar[ld] 
	\ar@/^2pc/[ldd]^{\id_B \otimes i_0} \\
	& B \otimes Q_1 \ar[d]^{\id_B \otimes i_1} & \\
             & B \otimes Q
}\end{equation*}

We claim that $\mathop{\mathrm{supp}_0} \rho = \supp \rho$ and $|\rho| = |\rho|_0$.

\begin{remark}\label{RemarkSupp0IsGlobalMinimum}
Note that  $i \colon \mathop{\mathrm{supp}_0} \rho \rightarrowtail Q$ is the global minimum in the preorder of extremal subobjects $i_1 \colon Q_1 \rightarrowtail Q$ corresponding to objects in~$\mathcal Q(\rho)$.
\end{remark}

\begin{proposition}\label{PropositionSuppIsATensorEpimorphism} Suppose a monoidal category~$\mathcal C$
	satisfies Properties~\ref{PropertySmallLimits}, \ref{PropertySubObjectsSmallSet}, \ref{PropertyLimitsOfSubobjectsArePreserved} and \ref{PropertyEqualizers}. Then for every  $\rho \colon A \to B \otimes Q$
	the morphism
	$|\rho|_0$ is a tensor epimorphism.
\end{proposition}
\begin{proof}
	Let $f,g \colon \mathop{\mathrm{supp}_0} \rho \to R$ be such morphisms
	that \begin{equation}\label{EquationMonoAfterTensor}
	({\id_B} \otimes f)|\rho|_0 = ({\id_B} \otimes g) |\rho|_0.\end{equation}
	We claim that $f=g$.
	
	Let $h \colon P \to \mathop{\mathrm{supp}_0} \rho$ be an equalizer of $f$ and $g$. Then by Property~\ref{PropertyEqualizers}
	the morphism
	${\id_B}\otimes h$ is an equalizer of ${\id_B}\otimes f$ and ${\id_B}\otimes g$. Hence~\eqref{EquationMonoAfterTensor} implies that there exists a morphism
	$q \colon A \to  B\otimes P$ making the diagram below commutative:
	$$\xymatrix{ 
		& & B \otimes Q                                  &   &          \\	
		A \ar[rr]^(0.45){|\rho|_0} \ar@{-->}[rrd]^q \ar[rru]^\rho &  & B \otimes (\mathop{\mathrm{supp}_0} \rho)\strut \ar@<0.2pc>[rr]^{{\id_B}\otimes f}
		\ar@<-0.2pc>[rr]_{{\id_B}\otimes g} \ar@{->}[u]_{\id_B\otimes i_0}   & & B \otimes R \\
		& & B \otimes P\strut  \ar[u]_{{\id_B}\otimes h} & &             }
	$$
	Recall that every equalizer, in particular $h$, is an extremal monomorphism.
	By Remark~\ref{RemarkSupp0Embedding}, $\mathop{\mathrm{supp}_0} \rho$ is an extremal subobject of $Q$, whence $P$ is an extremal subobject of $Q$. Moreover, $\rho$ factors through $q$. Since $\mathop{\mathrm{supp}_0} \rho$ is the corresponding limit,
	 $\mathop{\mathrm{supp}_0} \rho$ is an extremal subobject of $P$.
		  At the same time, $P$ is an extremal subobject of $\mathop{\mathrm{supp}_0} \rho$ via $h$. Hence $h$ is an isomorphism and $f=g$.
\end{proof}

Now Theorem~\ref{TheoremAbsValueSupportExistence} follows from Lemma~\ref{LemmaSupportZeroIsTheActualSupport} below:
\begin{lemma}\label{LemmaSupportZeroIsTheActualSupport} Let $\mathcal C$ be a monoidal category satisfying Properties~\ref{PropertySmallLimits}, \ref{PropertySubObjectsSmallSet}--\ref{PropertyLimitsOfSubobjectsArePreserved} and \ref{PropertyEqualizers}. We have $\mathop{\mathrm{supp}_0} \rho = \supp \rho$ and $|\rho| = |\rho|_0$
	for every morphism  $\rho \colon A \to B \otimes Q$.
\end{lemma}	
\begin{proof} By Proposition~\ref{PropositionSuppIsATensorEpimorphism}, $|\rho|_0$ is a tensor epimorphism, i.e. $|\rho|_0 \in \LIO(\mathbf{MorTens}(A,B))$. Hence it is sufficient to show that if  $\rho = (\id_B \otimes \tau)\rho_1$
	for some tensor epimorphism $\rho_1 \colon A \to B \otimes Q_1$ and a morphism $\tau \colon Q_1 \to Q$, then $|\rho|_0 \preccurlyeq \rho_1$.
	
	Let $\tau = i \pi$ be the (Epi, ExtrMono)-factorization of $\tau$
where $\pi \colon Q_1 \twoheadrightarrow Q_2$ is an epimorphism and $i \colon Q_2 \rightarrowtail Q$
is an extremal monomorphism.

Recall that by $i_0 \colon \mathop{\mathrm{supp}_0} \rho \rightarrowtail Q$
we denote the extremal monomorphism such that $\rho = (\id_B \otimes i_0)|\rho|_0$.

Consider the following diagram:
$$\xymatrix{ & B \otimes Q_1 \ar[d]^{\id_B \otimes \pi} \\
	 A \ar[ru]^{\rho_1} \ar[d]_{|\rho|_0}  &  B \otimes Q_2 \strut \ar@{>->}[d]^{\id_B \otimes i}  \\ 
     B \otimes (\mathop{\mathrm{supp}_0} \rho)\, \ar@{>->}[r]^(0.6){\id_B \otimes i_0} \ar@{>-->}[ru]^{\id_B \otimes h}   &  B \otimes Q }$$

By Remark~\ref{RemarkSupp0IsGlobalMinimum} there exists an extremal monomorphism $h \colon \mathop{\mathrm{supp}_0} \rho  \rightarrowtail Q_2 $ such that $i_0 = i h$. By Proposition~\ref{PropositionEpimorphismTensorEpimorphism},
$(\id_B \otimes \pi)\rho_1$ is a tensor epimorphism, whence, by the same proposition, 
$h$ is an epimorphism. Now the extremality of $h$ implies that $h$ is an isomorphism.
Therefore, $|\rho|_0=(\id_B \otimes h^{-1}\pi)\rho_1$ and $|\rho|_0 \preccurlyeq \rho_1$.
\end{proof}		

We conclude Section~\ref{SubsectionSupportMorTensAB} with two propositions that we will use to calculate supports:

\begin{proposition}\label{PropositionRhoMonoComposition}
	Let $\mathcal C$ be a monoidal category satisfying Properties~\ref{PropertySmallLimits}, \ref{PropertySubObjectsSmallSet}--\ref{PropertyLimitsOfSubobjectsArePreserved} and \ref{PropertyEqualizers} and let $\rho \colon A \to B \otimes Q$ be a morphism for some objects $A,B,Q$ in $\mathcal C$. If $i \colon Q \rightarrowtail \tilde Q$ is an extremal monomorphism, then
	$|(\id_B \otimes i) \rho| = |\rho|$ and $\supp \bigl((\id_B \otimes i) \rho \bigr) = \supp \rho$.
\end{proposition}
\begin{proof}  Let $\tilde \rho := (\id_B \otimes i) \rho$.	
	 By Corollary~\ref{CorollaryLIOCoarserFinerSufficientConditions}, we have $|\tilde \rho| \preccurlyeq |\rho|$. 
	 In particular, there exists a morphism $\tau \colon \supp \rho \to \supp \tilde \rho$
	 such that $(\id_B \otimes \tau)|\rho|=|\tilde \rho|$:
	 	 $$\xymatrix{ A \ar[d]_{|\rho|} \ar[rd]^{|\tilde\rho|} \ar@/^3pc/[rrdd]^{\tilde\rho}\ar@/_3.5pc/[dd]_\rho &  \\
	 	B \otimes (\supp \rho)\strut_{\strut} \ar@{>->}[d] \ar@{-->}[r]_{\id_B \otimes \tau} &  B \otimes (\supp \tilde \rho)\strut_{\strut} \ar@{>->}[rd] \\
	 	B \otimes Q\ \ar@{>->}[rr]_{\id_B \otimes i}  & & B \otimes \tilde Q \\
	 }$$
 Recall that, by Remark~\ref{RemarkSupp0Embedding} and Lemma~\ref{LemmaSupportZeroIsTheActualSupport},
 the monomorphisms $\supp \rho \rightarrowtail \tilde Q$
 and $\supp \tilde\rho \rightarrowtail \tilde Q$ are extremal. Moreover,
 $\tau$ is compatible with $\supp \rho \rightarrowtail \tilde Q$
 and $\supp \tilde\rho \rightarrowtail \tilde Q$, since $|\rho|$ is a tensor epimorphism. On the other hand, $\tilde \rho$ factors through $B \otimes (\supp \rho)$. By Remark~\ref{RemarkSupp0IsGlobalMinimum} and Lemma~\ref{LemmaSupportZeroIsTheActualSupport}, $\supp \tilde \rho$ is the global minimum among all such extremal subobjects $Q_1$ of $\tilde Q$ that $\tilde \rho$ factors through $B \otimes Q_1$. Hence $\supp \tilde \rho$ is a subobject of $\supp \rho$. At the same time, $\supp \rho$ is a subobject of $\supp  \tilde \rho$ via $\tau$. Thus $\tau$ is an isomorphism and we may identify
	$|\tilde \rho|$ with $|\rho|$ and $\supp \tilde\rho$ with $\supp \rho$.
\end{proof}	

\begin{proposition}\label{PropositionATensorEpimorphismIsSupp}
	Let $\mathcal C$ be a monoidal category satisfying Properties~\ref{PropertySmallLimits}, \ref{PropertySubObjectsSmallSet}--\ref{PropertyLimitsOfSubobjectsArePreserved} and \ref{PropertyEqualizers}.
	 For every tensor epimorphism $\rho \colon A \to B \otimes Q$
	and every extremal monomorphism $i \colon Q \rightarrowtail \tilde Q$
	we have $|(\id_B \otimes i)\rho|=\rho$ and $\supp \bigl((\id_B \otimes i)\rho\bigr) = Q$.
\end{proposition}
\begin{proof} By Proposition~\ref{PropositionRhoMonoComposition} it is sufficient to consider the case $i = \id_Q$, $\tilde Q = Q$.
	However, as $\rho \in \LIO(\mathbf{MorTens}(A,B))$, we have $|\rho|=\rho$.
\end{proof}

\subsection{Supports of comodule structures}\label{SubsectionSupportComodule}

We first show that under certain conditions on the base category the support of a comodule structure
is always a subcomonoid. 

\begin{theorem}\label{TheoremSupportComodule}
	Suppose $\mathcal C$ is a monoidal category
	satisfying Properties~\ref{PropertySmallLimits}, \ref{PropertySubObjectsSmallSet}--\ref{PropertyTensorPullback}, \ref{PropertyEqualizers} and~\ref{PropertyExtrMonomorphism} of Section~\ref{SubsectionSupportCoactingConditions}.
	Let $M$ be a comodule over a comonoid $(C,\Delta,\varepsilon)$
	and let $\rho \colon M \to M \otimes C$ be the corresponding morphism.
	Then in $\mathcal C$ there exist unique morphisms $$\Delta_0 \colon
	\supp \rho \to \supp \rho \otimes \supp \rho \text{\quad and \quad}\varepsilon_0 \colon \supp \rho \to \mathbbm{1}$$ making the diagrams below commutative:
	\begin{equation}\label{EquationSuppDelta}\xymatrix{ \supp \rho \ar@{-->}[d]^{\Delta_0}\ \ar@{>->}[r] &  C \ar[d]^{\Delta} \\
		(\supp \rho) \otimes (\supp \rho) \ar@{>->}[r]\  & C \otimes C \\  }\end{equation}
	\begin{equation}\label{EquationSuppEpsilon}
	\xymatrix{ \supp \rho\phantom{\Bigl|} \ar@{-->}[r]^(0.6){\varepsilon_0}\ar@{>->}[d]  &   \mathbbm{1}  \\
		C \ar[ru]^\varepsilon &       }
	\end{equation}
	
	Moreover, $(\supp \rho,\Delta_0,\varepsilon_0)$ is a comonoid and the diagrams below are commutative too:
	
	\begin{equation}\label{EquationComoduleDelta} \xymatrix{	 M \ar[rr]^(0.4){|\rho|} \ar[d]_{|\rho|} & &  M \otimes \supp \rho \ar[d]_{|\rho| \otimes \id_{\supp \rho}} \\
		M \otimes \supp \rho \ar@{->}[rr]^(0.4){{\id_M} \otimes \Delta_0} & &  M \otimes \supp \rho \otimes \supp \rho  }
	\end{equation}
	\begin{equation}\label{EquationComoduleEpsilon}\xymatrix{ M \ar[r]^(0.32){|\rho|} \ar@{=}[d] & M \otimes \supp \rho \ar@{->}[d]^{{\id_M} \otimes \varepsilon_0}  \\
		M   &  M \otimes \mathbbm{1}\ar[l]_(0.55){\sim}  } \end{equation}
	In other words, $M$ is a  $(\supp \rho)$-comodule and the monomorphism $\supp \rho \rightarrowtail C$ is a comonoid homomorphism.
\end{theorem}
\begin{proof} By Property~\ref{PropertyMonomorphism} the morphism $(\supp \rho) \otimes (\supp \rho) \to C \otimes C$
	is a monomorphism, which implies the uniqueness of $\Delta_{0}$.
	In addition, from~\eqref{EquationSuppEpsilon} above it is clear that the only way to define $\varepsilon_{0}$ is to restrict
	$\varepsilon$ on $\supp \rho$. Define $\varepsilon_{0}$ in this way, i.e. as the composition of
	$\supp \rho \to C$ and $\varepsilon$. Now the commutativity of~\eqref{EquationComoduleEpsilon} follows from the definition
	of a comodule applied to the  $C$-comodule $M$.
	
	Consider the pullback
	\begin{equation}\label{EqTheoremSupportComodulePullback}
\xymatrix{P\ \ar[d]_(0.45){\xi} \ar@{>->}[r] \ar@{}[rd]|<{\pullbackc} &  C \ar[d]^(0.45){\Delta} \\
		(\supp \rho) \otimes (\supp \rho) \ar@{>->}[r]\  & C \otimes C \\  }	
	\end{equation}
	By Property~\ref{PropertyExtrMonomorphism} the monomorphism $P \rightarrowtail C$ is extremal.
	(Recall that by Property~\ref{PropertyEpiExtrMonoFactorizations}, which follows from Properties~\ref{PropertySmallLimits} and \ref{PropertySubObjectsSmallSet}, all extremal monomorphisms are strong and therefore pullback stable.)
	
	Tensor the diagram~\eqref{EqTheoremSupportComodulePullback} above by $M$ (recall that by Property~\ref{PropertyTensorPullback} it will still be a pullback) and incorporate into a larger one (the outer square is commutative since $M$ is a $C$-comodule):
	$$\xymatrix{
		M   \ar[rrr]^{|\rho|} \ar[dd]^{|\rho|}\ar@{-->}[rrd]^{\rho_1}      & &           &   M \otimes (\supp \rho)\strut_{\strut}
		\ar@{-->}[ld]_{{\id_M}\otimes \tau}        \ar@{>->}[d] \\	
		& &	M \otimes P\ \,\ar[d]_{{\id_M} \otimes \xi} \ar@{>->}[r] \ar@{}[rd]|<{\pullbackd}    &   M \otimes C                                        \ar[d]^{\Delta}      \\
		M \otimes (\supp \rho)   
		\ar[rr]^(0.4){|\rho| \otimes {\id_{\supp \rho}}}    & & M \otimes ((\supp \rho) \otimes (\supp \rho))\ \,\ar@{>->}[r]  &  M \otimes( C \otimes C)    \\  }$$
	By the universal property of a pullback there exists $\rho_1 \colon M \to M \otimes P$ 
	and by the definition of $\supp \rho$ there exists $\tau \colon \supp \rho \rightarrowtail P$
	making the diagram above commutative.
	Now it is sufficient to define $\Delta_{0} := \xi\tau$.
	The commutativity of~\eqref{EqTheoremSupportComodulePullback} and the definition of $\tau$
	imply the commutativity of~\eqref{EquationSuppDelta}.
	 The commutativity of~\eqref{EquationComoduleDelta}
	follows from the fact that $M \otimes (\supp \rho) \otimes (\supp \rho) \to M \otimes C \otimes C$ is a monomorphism.
	The axioms of a comonoid for $(\supp \rho,\Delta_0,\varepsilon_0)$ are verified analogously:
	$$ \xymatrix{  \supp \rho \ar[rr]^(0.4){\Delta_0}\ar[d]^(0.45){\Delta_0} & & (\supp \rho) \otimes (\supp \rho)\ \ar@{>->}[r] \ar[d]^(0.45){\Delta_0\otimes{\id_{\supp \rho}}} & C \otimes C \ar[dd]^{\Delta\otimes {\id_C}} \\
		(\supp \rho) \otimes (\supp \rho)\strut \ar@{>->}[d] \ar[rr]^(0.4){{\id_{\supp \rho}}\otimes\Delta_0} & &(\supp \rho) \otimes (\supp \rho) \otimes (\supp \rho)
		\ar@{>->}[rd]  & \\
		C \otimes C \ar[rrr]^{ {\id_C}\otimes\Delta} & & &  C \otimes C \otimes C}$$
	
	$$ \xymatrix{ 
		\supp \rho\   \ar@{>->}[r]  \ar[d]^{\Delta_0} \ar@/_65pt/[dd]_{\sim}      &  C  \ar[d]^\Delta \ar@/^65pt/[dd]^{\sim}         \\
		(\supp \rho) \otimes (\supp \rho)\    \ar@{>->}[r]
		\ar[d]^{{\id_{\supp \rho}}\otimes \varepsilon_{0}} &  C \otimes C
		\ar[d]^{{\id_{\supp \rho}}\otimes \varepsilon}  \\
		(\supp \rho) \otimes \mathbbm{1}\   \ar@{>->}[r]    &  C \otimes \mathbbm{1} \\
	}$$
(The diagram for the left counit axiom is constructed similarly.)
\end{proof}	

Now we show that a morphism between comodule structures is always a comonoid homomorphism if its domain is a tensor epimorphism:

\begin{proposition}\label{PropositionFromTensorEpiComonoidMorphism}
	Let $\mathcal C$ be a monoidal category and let $\rho_i \colon M \to M \otimes C_i$ define on an object $M$ structures of $C_i$-comodules
	for comonoids $(C_i, \Delta_i, \varepsilon_i)$ for $i=1,2$.
	Suppose that $\rho_1$ is a tensor epimorphism and the diagram below is commutative for some morphism $\tau \colon C_1 \to C_2$:
	\begin{equation*}\xymatrix{
		M \ar[r]^(0.35){\rho_1} \ar[rd]_(0.45){\rho_{2}} & M \otimes C_1 \ar[d]^{{\id_M}\otimes \tau} \\
		& M \otimes C_2
	}\end{equation*}
	Then $\tau$ is a comonoid homomorphism.
\end{proposition}
\begin{proof} Consider the diagrams
	\begin{equation}\label{EquationComoduleDeltaTauComon} \xymatrix{	 M \ar[rr]^(0.4){\rho_1} \ar[d]_{\rho_1} & &  M \otimes C_1 \ar[d]_{{\id_M} \otimes \Delta_1}
		\ar[rr]^{\id_M \otimes \tau}
		& &  M \otimes C_2 \ar[d]_{{\id_M} \otimes \Delta_2}\\
		M \otimes C_1 \ar@{->}[rr]^(0.4){\rho_1 \otimes \id_{C_1}} & &  M \otimes C_1 \otimes C_1 \ar[rr]^{{\id_M} \otimes {\tau\otimes \tau}} & & M \otimes C_2\otimes C_2  }
	\end{equation}
	\begin{equation}\label{EquationComoduleEpsilonTauComon}\xymatrix{ M \ar[r]^(0.35){\rho_1} \ar@{=}[d] & M \otimes C_1 \ar@{->}[d]_{{\id_M} \otimes \varepsilon_1} \ar[r]^{{\id_M} \otimes \tau}  & M \otimes C_2 \ar[ld]^{{\id_M} \otimes \varepsilon_{2}} \\
		M   &  M \otimes \mathbbm{1}\ar[l]_(0.6){\sim}  } \end{equation}
	where the left squares and the outer, respectively, square and pentagon are commutative by the definition of a comodule. Then use the definition of a tensor epimorphism.
\end{proof}	

\begin{remark} Suppose that $\mathcal C$ satisfies Properties~\ref{PropertySmallLimits}, \ref{PropertySubObjectsSmallSet}--\ref{PropertyTensorPullback}, \ref{PropertyEqualizers} and~\ref{PropertyExtrMonomorphism} of Section~\ref{SubsectionSupportCoactingConditions}.
Theorem~\ref{TheoremSupportComodule} and Proposition~\ref{PropositionFromTensorEpiComonoidMorphism} imply that in the construction of $\supp \rho$
via the limit (see $\mathop{\mathrm{supp}_0} \rho$ in Section~\ref{SubsectionSupportMorTensAB})
in the case of a comodule structure $\rho \colon M \to M \otimes C$ over a comonoid $C$ it is sufficient to consider
only comodule structures $\rho_1 \colon M \to M \otimes Q_1$ for comonoids $Q_1$ and comonoid homomorphisms between $Q_1$.
(The limit $\lim T$ is still taken in $\mathcal C$.)
\end{remark}

\begin{example}
	In the case $\mathcal C = \mathbf{Vect}_\mathbbm{k}$, where $\mathbbm{k}$ is a field, $\supp \rho$ corresponds to taking the intersection of all such subcoalgebras that the map can be factored through them. Therefore $\supp \rho$ coincides with the support defined in~\cite{AGV1}. (See also Examples~\ref{support} (\ref{ExampleVectMorTensSupp}).)
\end{example}

Now we are ready to interpret Theorem~\ref{TheoremSupportComodule} and Proposition~\ref{PropositionFromTensorEpiComonoidMorphism} in terms of the Lifting Problem.

Let $\mathcal C$ be a monoidal category.
For a given object $M$ in $\mathcal C$ denote by $\mathbf{ComodStr}(M)$
the category where
\begin{itemize}
	\item the objects are morphisms $\rho \colon M \to M \otimes C$
	defining on $M$ a comodule structure for some comonoid $C$;
	\item the morphisms between $\rho_1 \colon M \to M \otimes C_1$ and $\rho_2 \colon   M \to  M \otimes C_2$
	are comonoid homomorphisms $\tau \colon C_1 \to C_2$
	making the diagram below commutative:
	$$ \xymatrix{	 M \ar[r]^(0.4){\rho_1} \ar[rd]_{\rho_2} &  M \otimes C_1 \ar[d]^{\id_{ B} \otimes\, \tau} \\
		&  M \otimes C_2}$$
\end{itemize}

Denote by $G$ 
the forgetful functor $\mathbf{ComodStr}(M) \to \mathbf{MorTens}(M, M)$.

Suppose that $\mathcal C$ satisfies Properties~\ref{PropertySmallLimits}, \ref{PropertySubObjectsSmallSet}--\ref{PropertyTensorPullback}, \ref{PropertyEqualizers} and~\ref{PropertyExtrMonomorphism} of Section~\ref{SubsectionSupportCoactingConditions}. Fix
a comodule structure $\rho \colon M \to M \otimes C$  over a comonoid $C$.
Theorem~\ref{TheoremSupportComodule} implies that $|\rho|_G=G\rho_0$
for a unique object $\rho_0$ in $\mathbf{ComodStr}(M)$.  By Proposition~\ref{PropositionFromTensorEpiComonoidMorphism},
 $\rho_0$  satisfies the assumptions of Proposition~\ref{PropositionLIOLiftingAbsValue}
and, therefore, $\rho_0$ is an initial object in $\mathbf{ComodStr}(M)_G(|\rho|_G)$.

\subsection{Universal comeasuring monoids}\label{SubsectionUnivComeasExistence}

Fix $\Omega$-magmas $A$ and $B$ in a braided monoidal category $\mathcal C$. Consider the category $\mathbf{Comeas}(A,B)$ where
\begin{itemize}
	\item the objects are all comeasurings $\rho \colon A \to B \otimes Q$ for arbitrary monoids $Q$;
	\item the morphisms from $\rho_1 \colon A \to B \otimes Q_1$
to $\rho_2 \colon A \to B \otimes Q_2$  are monoid homomorphisms $\varphi \colon Q_1 \to Q_2$
making the diagram below commutative:

$$\xymatrix{ A \ar[r]^(0.4){\rho_1} 
	\ar[rd]_{\rho_2}
	& B \otimes Q_1 \ar[d]^{\id_B \otimes \varphi} \\
	& B \otimes Q_2} $$
\end{itemize}

Denote by $G_1$ the forgetful functor $\mathbf{Comeas}(A,B)\to \mathbf{MorTens}(A,B)$. Given a tensor epimorphism $\rho_U \colon A \to B \otimes U$ for some object $U$ in $\mathcal C$, let us call the monoid $\mathcal{A}^\square(\rho_U)$ corresponding to the initial object $\rho_U^\mathbf{Comeas} \colon A \to B \otimes \mathcal{A}^\square(\rho_U)$ in $\mathbf{Comeas}(A,B)_{G_1}(\rho_U)$ (if it exists) the \textit{$U$-universal comeasuring monoid} from $A$ to $B$.

\begin{theorem}\label{TheoremMonUnivComeasExistence}
	  Suppose that a braided monoidal category $\mathcal C$
	satisfies  Properties~\ref{PropertySmallLimits}, \ref{PropertyEpiExtrMonoFactorizations}, \ref{PropertyFactorObjectsSmallSet}, \ref{PropertyMonomorphism},  \ref{PropertyEpimorphism},  \ref{PropertySwitchProdTensorIsAMonomorphism}, \ref{PropertyFreeMonoid} of Section~\ref{SubsectionSupportCoactingConditions}.
	Then there exists an initial object in $\mathbf{Comeas}(A,B)_{G_1}(\rho_U)$ if $\mathbf{Comeas}(A,B)_{G_1}(\rho_U)$ is not empty.
	In other words, the Lifting Problem for the forgetful functor $$G_1 \colon \mathbf{Comeas}(A,B)\to \mathbf{MorTens}(A,B)$$
	has a solution for all such $\rho_U \in \LIO(\mathbf{MorTens}(A,B))$.
\end{theorem}

We first describe the construction and then prove in several lemmas that this construction indeed provides
the initial object in $\mathbf{Comeas}(A,B)_{G_1}(\rho_U)$.

Recall that by $\mathcal F \colon  \mathcal C \to \mathsf{Mon}(\mathcal C)$ we denote the left adjoint to the forgetful functor $\mathsf{Mon}(\mathcal C) \to \mathcal C$ and let $\eta_M \colon M \to \mathcal FM$ be the unit of this adjunction.
	Consider the category~$\Lambda$ where objects $\alpha$ correspond to all monoid homomorphisms $\tau_\alpha \colon {\mathcal F}U \twoheadrightarrow  Q_\alpha$ such that
	they are epimorphisms in $\mathcal C$ and the composition
	$({\id_B} \otimes {\tau_\alpha\eta_U})\rho_U$ is a comeasuring. The arrow $\alpha \to \beta$ exists if and only if
	$\tau_\beta = \tau_{\alpha\beta}\tau_\alpha$ for some morphism $\tau_{\alpha\beta} \colon Q_\alpha \to Q_\beta$.
	We say that $\tau_\alpha$ and $\tau_\beta$ are equivalent if $ \tau_{\alpha\beta}$ is an isomorphism.
	Define the functor $T$ as follows: $T\alpha = Q_\alpha$,
	$T(\alpha \to \beta) =\tau_{\alpha\beta}$ for every objects $\alpha,\beta$ in $\Lambda$.
	By Property~\ref{PropertyFactorObjectsSmallSet}, we may assume that the set of equivalence classes is small. Since in $ \lim T$ it is sufficient to take only a single representative for each equivalence class,
	by Property~\ref{PropertySmallLimits}, the limit
	 $Q_0 := \lim T$ exists. Denote by $\varphi_\alpha \colon Q_0 \to Q_\alpha$ the limiting cone.
	 Since epimorphisms  $\tau_\alpha \colon {\mathcal F}U \twoheadrightarrow  Q_\alpha$ form themselves a cone,
	 there exists a unique morphism $\theta$ between the cones:
		$$	\xymatrix{  U \ar[r]^(0.4){\eta_U} &  {\mathcal F}U  \ar@{-->}[r]^{\theta} \ar@{->>}[rd]^{\tau_\alpha} &  Q_0 \ar@{->}[d]^{\varphi_\alpha} \\
		&  &  \ Q_\alpha 
	}$$
As in the proof of Theorem~\ref{TheoremColimitsMon}, Property~\ref{PropertyEpimorphism} implies that morphisms $\tau_{\alpha\beta}$ between different epimorphisms $\tau_\alpha$ are monoid homomorphisms too.
Hence, by the remarks made at the beginning of Section~\ref{Subsection(Co)limitsMonHopfMonReflectivity},
 there exists a unique monoid structure on $Q_0$ making $\theta$ a monoid homomorphism.

Below we show that $({\id_B} \otimes {\theta\eta_U})\rho_U$ is the initial object in $\mathbf{Comeas}(A,B)_{G_1}(\rho_U)$.

\begin{lemma}
$({\id_B} \otimes {\theta\eta_U})\rho_U$ is a comeasuring.
\end{lemma}
\begin{proof}
By the definition, a morphism is a comeasuring if for every $\omega\in \Omega$ a certain diagram is commutative.
Fix $\omega\in \Omega$ and denote by $\sigma_1$ and $\sigma_2$ the two boundary compositions of the corresponding diagram for $({\id_B} \otimes {\eta_U})\rho_U$. It is possible that $\sigma_1\ne \sigma_2$, but \begin{equation}\label{EquationSigma12TauAlpha}
\left({\id_B^{t(\omega)}}\otimes\tau_\alpha\right) \sigma_1 = \left({\id_B^{t(\omega)}}\otimes\tau_\alpha\right) \sigma_2
\end{equation}
for every $\alpha$ in $\Lambda$. Now it is sufficient to show that 
\begin{equation}\label{EquationSigma12Theta}\left({\id_B^{t(\omega)}}\otimes\theta\right) \sigma_1 = \left({\id_B^{t(\omega)}}\otimes\theta\right) \sigma_2.\end{equation}
Denote by $\Lambda_0$ a small set of equivalence class representatives in $\Lambda$.
Let $\varphi$ be the unique morphism making the diagram below commutative ($\pi_\alpha$ are the corresponding projections, $\alpha\in\Lambda_0$):

$$\xymatrix{{\mathcal F}U \ar@{->>}[d]^{\tau_\alpha} \ar[r]^\theta  &  Q_0 \ar@{-->}[d]^{\varphi}
	                                                             \ar[ld]^{\varphi_\alpha}  \\
            Q_\alpha          & \prod\limits_{\alpha\in\Lambda_0} Q_\alpha \ar@{->}[l]^{\pi_\alpha}  }$$
Note that since $Q_0$ is a limit (as we have mentioned above, it is sufficient to take the limit
just on $\Lambda_0$), $\varphi$ is a monomorphism.

Define the morphisms $\varkappa$ and $\tau$ as the unique morphisms making the diagram below commutative (now $\tilde\pi_\alpha$ are the corresponding projections; in addition, by Properties~\ref{PropertyMonomorphism} and~\ref{PropertyEpimorphism}, ${\id_B^{t(\omega)}}\otimes {\,\varphi}$ is a monomorphism and ${\id_B^{t(\omega)}}\otimes\tau_\alpha$ is an epimorphism for every $\alpha \in \Lambda$):
$$\xymatrix{ A^{s(\omega)} \ar@<0.5ex>[r]^(0.4){\sigma_1} \ar@<-0.5ex>[r]_(0.4){\sigma_2} & 
		B^{t(\omega)} \otimes {\mathcal F}U \ar@{->>}[d]^{{\id_B^{t(\omega)}}\otimes\tau_\alpha} \ar@/_5pc/@{-->}[dd]_\tau \ar[r]^{{\id_B^{t(\omega)}}\otimes{\,\theta}}  &  	B^{t(\omega)} \otimes Q_0\strut \ar@{>->}[d]^{{\id_B^{t(\omega)}}\otimes {\,\varphi}} 
		   \\ 
&	B^{t(\omega)} \otimes Q_\alpha  &  \ar@{->}[l]^{{\id_B^{t(\omega)}}\otimes{\,\pi_\alpha}}
	 B^{t(\omega)} \otimes \prod\limits_{\alpha\in\Lambda_0} Q_\alpha  \ar@/^1pc/@{>-->}[ld]^(0.4)\varkappa	 \\
	 	& \prod\limits_{\alpha\in\Lambda_0} B^{t(\omega)} \otimes Q_\alpha \ar[u]^{\tilde\pi_\alpha}  \\
}$$
By Property~\ref{PropertySwitchProdTensorIsAMonomorphism} the morphism $\varkappa$ is a monomorphism.
Moreover, \eqref{EquationSigma12TauAlpha} implies
\begin{equation*}
	\tau \sigma_1 = \tau \sigma_2.\end{equation*}
Hence 
\begin{equation*}
	\varkappa\left({\id_B^{t(\omega)}}\otimes{\varphi\theta}\right) \sigma_1 = 	\varkappa\left({\id_B^{t(\omega)}}\otimes{\varphi\theta}\right) \sigma_2.\end{equation*}
Since both $\varkappa$ and ${\id_B^{t(\omega)}}\otimes{\varphi}$
are monomorphisms, we get \eqref{EquationSigma12Theta}. As a consequence, $({\id_B} \otimes {\theta\eta_U})\rho_U$ is a comeasuring too.
\end{proof}

\begin{lemma}\label{LemmaEpiAgainAComeasuring} If for some monoid homomorphism $f \colon {\mathcal F}U \to Q$ the morphism
		$$({\id_B} \otimes {f\eta_U})\rho_U$$ is a comeasuring
		and $f = i\pi$ is the (Epi, ExtrMono)-factorization of $f$,
		then  $$({\id_B} \otimes {\pi\eta_U})\rho_U$$ is a comeasuring too.
\end{lemma}
\begin{proof} Recall that by Proposition~\ref{PropositionMonoidEpiMono}
	both $i$ and $\pi$ are monoid homomorphisms.
	
	Again we have to show that for every $\omega\in \Omega$ a certain diagram is commutative.
Fix $\omega\in \Omega$ and denote by $\sigma_1$ and $\sigma_2$ the two boundary compositions of the corresponding diagram for $({\id_B} \otimes {\eta_U})\rho_U$. We know that
$$
\left({\id_B^{t(\omega)}}\otimes f\right) \sigma_1 = \left({\id_B^{t(\omega)}}\otimes f\right) \sigma_2.
$$
 Hence
$$
\left({\id_B^{t(\omega)}}\otimes i\pi\right) \sigma_1 = \left({\id_B^{t(\omega)}}\otimes i\pi\right) \sigma_2
$$
and
$$
\left({\id_B^{t(\omega)}}\otimes \pi\right) \sigma_1 = \left({\id_B^{t(\omega)}}\otimes \pi\right) \sigma_2
$$
since by Property~\ref{PropertyMonomorphism} the morphism ${\id_B^{t(\omega)}}\otimes i$ is a monomorphism.
Therefore, $({\id_B} \otimes {\pi\eta_U})\rho_U$ is a comeasuring too.
\end{proof}	

\begin{lemma}\label{LemmaThetaEpi}
$\theta$ is an epimorphism in $\mathcal C$.
\end{lemma}
\begin{proof}
Consider the (Epi, ExtrMono)-factorization  $\theta = i\pi$, which exists by Property~\ref{PropertyEpiExtrMonoFactorizations}.
By Lemma~\ref{LemmaEpiAgainAComeasuring} the morphism $({\id_B} \otimes {\pi\eta_U})\rho_U$ is a comeasuring too.
Therefore $\pi = \tau_\beta$ for some $\beta$ from $\Lambda$ and
for every $\alpha$ from $\Lambda$ we have $\tau_\alpha = (\varphi_\alpha i)\tau_\beta$:

		$$	\xymatrix{  U \ar[r]^(0.4){\eta_U} &  {\mathcal F}U  \ar@{->>}[r]^{\tau_\beta} \ar@{->>}[rd]^{\tau_\alpha} &  Q_\beta 
			\ar@{->}[r]^i
			& Q_0 \ar@{->}[ld]^{\varphi_\alpha}\\
	&  &  \ Q_\alpha 
}$$
	Hence $Q_\beta$ is a limit of the functor $T$ too, $i$ is an isomorphism and $\theta$ is an epimorphism.
\end{proof}
\begin{proof}[Proof of Theorem~\ref{TheoremMonUnivComeasExistence}]
	Let $\rho \colon A \to B \otimes Q$ be a comeasuring such that $\rho = ({\id_B}\otimes \tau)\rho_U$
	for some morphism $\tau \colon U \to Q$. There exists a unique monoid homomorphism
	$\sigma \colon {\mathcal F}U \to Q$ such that $\tau = \sigma \eta_U$.
	Now consider the (Epi, ExtrMono)-factorization $\sigma = i\pi$. By Lemma~\ref{LemmaEpiAgainAComeasuring} the morphism $({\id_B} \otimes {\pi\eta_U})\rho_U$ is a comeasuring too. Therefore $\pi = \tau_\alpha$ for some $\alpha$ from $\Lambda$
	and $\rho = ({\id_B} \otimes {\,i\varphi_\alpha})({\id_B} \otimes {\,\theta\eta_U})\rho_U$:
	
	$$	\xymatrix{ B\otimes U \ar[rr]^{{\id_B}\otimes {\,\eta_U}} \ar[rrd]^{{\id_B}\otimes \tau} & & B \otimes {\mathcal F}U \ar@{->}[d]^(0.6){{\id_B}\otimes {\,\sigma}} \ar@{->>}[rr]^{{\id_B}\otimes {\,\theta}} \ar@{->>}[rrd]^{\,{\id_B}\otimes\tau_\alpha} & & B \otimes Q_0 \ar@{->>}[d]^{{\id_B}\otimes {\,\varphi_\alpha}} \\
		A \ar[u]^{{\id_B}\otimes {\,\rho_U}} \ar[rr]^(0.4)\rho & & B\otimes Q & & \ B\otimes Q_\alpha \ar@{>->}[ll]^{{\id_B}\otimes {\,i}}
	}$$
	
	Suppose now that there exists another monoid morphism $g \colon Q_0 \to Q$ 
	such that $$\rho = ({\id_B} \otimes {g})({\id_B} \otimes {\theta\eta_U})\rho_U.$$
	Recall that $\rho_U$ is a tensor epimorphism. Hence $$g\,\theta\eta_U = i\varphi_\alpha\, \theta\eta_U.$$
	By the universal property of $\eta_U$ we have
	$$g\,\theta = i\varphi_\alpha\, \theta.$$
    Lemma~\ref{LemmaThetaEpi} implies that $g= i\varphi_\alpha$.
    Hence the morphism is unique and $(\id_B \otimes {\theta\eta_U})\rho_U$
    is indeed an initial object in $\mathbf{Comeas}(A,B)_{G_1}(\rho_U)$.
\end{proof}	

\begin{example}\label{ExampleUnivComeasSets} In the case $\mathcal C = \mathbf{Sets}$
	there exist maps $\varphi_U \colon A \to B$ and $\psi_U \colon A \to U$
	such that $\rho_U(a) = (\varphi_U(a), \psi_U(a))$
	for every $a\in A$. The map $\rho_U$
	is a tensor epimorphism if and only if $\psi_U$ is surjective. By Example~\ref{ExampleComeasSets}, $\mathbf{Comeas}(A,B)_{G_1}(\rho_U)$ is not empty only if $\varphi_U$
	is an $\Omega$-magma homomorphism. (Below in Remark~\ref{RemarkTerminalComeas} we show that the converse is true too.) Consider the case when $\Omega$ is as in Examples~\ref{comod} (\ref{ExampleUnitalAlgebraOmega}) and $A$ and $B$ are ordinary monoids. Suppose that indeed  $\varphi_U$ is a monoid homomorphism. Then 
	$\mathcal{A}^\square(\rho_U)$ is isomorphic to $A$ factored by the congruence generated by the kernel equivalence relation of $\psi_U$. If we take $\psi_U = \id_A$, then $\mathcal{A}^\square(\rho_U) \cong A$ and  $\rho_U^\mathbf{Comeas}=(\varphi_U, \id_A)$
	is universal among all comeasurings $\rho \colon A \to B \times Q$ such that $\rho = (\varphi_U, \psi)$
	for some map $\psi \colon A \to Q$. (Recall that by Example~\ref{ExampleComeasSets} such $\rho$ is a comeasuring if and only if $\psi$ is a monoid homomorphism.)
\end{example}

\begin{remark} \label{RemarkTerminalComeas}
	Suppose that there exists a terminal object $T$ in $\mathcal{C}$ and the object $B \otimes T$ is again terminal for every object $B$. Note that the unique morphisms
	$\mathbbm{1}\to T$ and $T\otimes T \to T$ define on $T$ the structure of a monoid as all the corresponding diagrams are trivially commutative since $T$ is a terminal object. Denote by $\rho_T$  the unique comeasuring $A \to B \otimes T$. Then $\rho_T$ is a terminal object in $\mathbf{Comeas}(A,B)_{G_1}(\rho_U)$. 
	
	Note that the above conditions hold for $\mathcal C = \mathbf{Vect}_\mathbbm{k}$,
	but do not hold	for $\mathcal C = \mathbf{Sets}$, since $\lbrace * \rbrace$ is a terminal object in $\mathbf{Sets}$,
	but $B \times \lbrace * \rbrace$ is isomorphic to  $\lbrace * \rbrace$ if and only if $B$ consists of a single element.
	However, a terminal object in $\mathbf{Comeas}(A,B)_{G_1}(\rho_U)$  still exists if $\mathbf{Comeas}(A,B)_{G_1}(\rho_U)$ is not empty.
	As we have already mentioned in Example~\ref{ExampleUnivComeasSets}, $\mathbf{Comeas}(A,B)_{G_1}(\rho_U)$ is not empty only if $\rho_U = (\varphi_U, \psi_U)$
	for some map $\psi_U \colon A \to U$ and an $\Omega$-magma homomorphism $\varphi_U \colon A \to B$.
	Assume that this is the case. Then $\rho_T \colon A \to B \times \lbrace * \rbrace$, where $\rho_T(a):=(\varphi_U(a),*)$,
	is a comeasuring, which is a terminal object in $\mathbf{Comeas}(A,B)_{G_1}(\rho_U)$ 
	 for $\mathcal C = \mathbf{Sets}$.
\end{remark}	

\subsection{Universal coacting bimonoids}\label{SubsectionUnivBiCoactingExistence}
Fix an $\Omega$-magma $A$ in a braided monoidal category~$\mathcal C$.
Define the category $\mathbf{Coact}(A)$ where \begin{itemize}
	\item the objects are all coactions $\rho \colon A \to A \otimes B$ for arbitrary bimonoids $B$;
	\item  the morphisms from $\rho_1 \colon A \to A \otimes B_1$
	to $\rho_2 \colon A \to A \otimes B_2$  are bimonoid homomorphisms $\varphi \colon B_1 \to B_2$
	making the diagram below commutative:
	$$\xymatrix{ A \ar[r]^(0.4){\rho_1} 
		\ar[rd]_{\rho_2}
		& A \otimes B_1 \ar[d]^{{\id_A} \otimes \varphi} \\
		& A \otimes B_2} $$
\end{itemize}

Consider the following commutative diagram consisting of forgetful functors:
$$\xymatrix{
	\mathbf{Coact}(A) \ar[r]^{G_3} \ar[d]_{G_2} & \mathbf{Comeas}(A,A) \ar[d]^{G_1} \\
	\mathbf{ComodStr}(A) \ar[r]^G    & \mathbf{MorTens}(A,A) 
}$$	

Let $U$ be a comonoid and let a tensor epimorphism $\rho_U \colon A \to A \otimes U$
define on $A$ a structure of a $U$-comodule. 
Let us call the bimonoid  corresponding to the initial object
in $\mathbf{Coact}(A)_{G_2}(\rho_U)$ (if it exists) the \textit{$U$-universal coacting bimonoid} on~$A$.

\begin{remark}\label{RemarkCoactAG2RhoUEquivDef}
	By Proposition~\ref{PropositionFromTensorEpiComonoidMorphism}, 
	the category $\mathbf{Coact}(A)_{G_2}(\rho_U)$  coincides with the category
	$\mathbf{Coact}(A)_{GG_2}(G\rho_U) = \mathbf{Coact}(A)_{G_1G_3}(G \rho_U)$.
\end{remark}

It was noticed by M.~Sweedler~\cite{Sweedler} that the universal measuring coalgebra
for $A=B$ is in fact a universal acting bialgebra. Here we apply this idea to lift initial objects to subcategories
containing certain comeasurings.

We say that a subcategory $\mathcal D$ of $\mathbf{Comeas}(A,A)$
is \textit{closed under coarsenings} if for every object $\rho \colon A\to A \otimes Q_1$
in $\mathcal D$ and every monoid homomorphism $\tau \colon Q_1 \to Q_2$, where $Q_2$ is a monoid in $\mathcal C$,
the comeasuring $(\id_A \otimes \tau)\rho \colon A\to A \otimes Q_2$ is again an object in $\mathcal D$.

\begin{theorem}\label{TheoremBimonUnivCoactingExistenceForSubCat}  Let
	$A$ be an $\Omega$-magma in a braided monoidal category $\mathcal C$.
	Let $\mathcal D$ be a full subcategory of $\mathbf{Comeas}(A,A)$ such that
	there exists an initial object $\rho_0 \colon A \to A \otimes B_0$ in $\mathcal D$.
	Denote by  $\mathbf{Coact}(A)(\mathcal D)$ the full subcategory of $\mathbf{Coact}(A)$
	consisting of all the objects whose images under the forgetful functor $G_3 \colon \mathbf{Coact}(A)\to \mathbf{Comeas}(A,A)$
	belong to $\mathcal D$. 
	Suppose that $\mathcal D$ is closed under coarsenings and $A \mathrel{\widetilde\to} A \otimes \mathbbm{1}$
	and $(\rho_0 \otimes \id_{B_0})\rho_0$ are objects in $\mathcal D$ too.
	Then the monoid $B_0$ admits a unique comonoid structure turning $\rho_0$
	into a coaction, which is the initial object in $\mathbf{Coact}(A)(\mathcal D)$.
\end{theorem}
\begin{proof}
	By our assumptions, $(\rho_0 \otimes \id_{B_0})\rho_0$ is an object in $\mathcal D$.
	Hence there exists a unique monoid homomorphism 
	$\Delta_0 \colon B_0 \to B_0 \otimes B_0$
	making the diagram below commutative:
	$$\xymatrix{ A \ar[rrr]^{\rho_0} \ar[d]_{\rho_0} &&& A\otimes B_0
		\ar@{-->}[d]^{\id_A\otimes \Delta_0} \\
		A\otimes B_0 \ar[rrr]_(0.4){\rho_0\otimes \id_{B_0}} &&&  A\otimes B_0 \otimes B_0}$$
	Now we define the comultiplication in $B_0$ to be equal to $\Delta_0$.
	
	In order to define the counit, we consider the morphism $$A
	\mathrel{\widetilde\to} A  \otimes \mathbbm{1},$$ which is an object in $\mathcal D$ by the assumptions of the theorem too.
	There exists a unique monoid homomorphism ${\varepsilon_0} \colon   B_0 \to \mathbbm{1}$
	making the diagram below commutative:
	$$\xymatrix{ A \ar[r]^(0.3){\rho_0} \ar[rd]_{\sim} & A \otimes B_0
		\ar[d]^{\id_A \otimes {\varepsilon_0}}\\
		& A \otimes \mathbbm{1}}$$
	We define the counit in $B_0$ to be equal to ${\varepsilon_0}$.
	
	Now we have to prove that the comultiplication $\Delta_0$ is coassociative and that ${\varepsilon_0}$
	and $\Delta_0$ indeed satisfy the counit axioms.
	
	Consider the diagram
	$$\xymatrix{ A \ar[rr]^{\rho_0} \ar[dd]_{\rho_0} \ar[rd]^{\rho_0} && A \otimes B_0 \ar[rd]^{\qquad \rho_0 \otimes \id_{B_0}} \ar'[d][dd]_{\rho_0 \otimes \id_{B_0}} & \\
		& A \otimes B_0 \ar[rr]^(0.3){\id_A \otimes \Delta_0} \ar[dd]_(0.3){\rho_0 \otimes \id_{B_0}} && A \otimes B_0 \otimes B_0\ar[dd]_(0.3){\rho_0 \otimes \id_{B_0 \otimes B_0}} \\
		A \otimes B_0  \ar'[r][rr]_(0.3){\id_A \otimes \Delta_0}
		\ar[rd]^{\id_A \otimes \Delta_0} && A \otimes B_0 \otimes B_0 \ar[rd]_(0.35){\id_A \otimes \Delta_0 \otimes \id_{B_0}\qquad} &\\
		& A \otimes B_0 \otimes B_0 
		\ar[rr]_(0.4){\id_A  \otimes \id_{B_0}\otimes \Delta_0}  && A \otimes B_0 \otimes B_0 \otimes B_0 } $$

	The right face is commutative since
	it is the definition of $\Delta_0$ tensored by $\id_{B_0}$.
	The commutativity of the left, the upper and the rear face follows from
	the definition of $\Delta_0$ too.
	The front face is commutative since both compositions are equal to $\rho_0 \otimes \Delta_0$.
	Therefore the compositions of the lower face being composed with $\rho_0$
	are equal too.
	Now the universal property of $B_0$ implies the commutativity of the diagram
	$$\xymatrix{ B_0 \ar[rr]_{\Delta_0}\ar[d]^{\Delta_0} && B_0 \otimes B_0 \ar[d]^{{\Delta_0}\otimes \id_{B_0}}\\ 
		B_0 \otimes B_0 \ar[rr]_(0.42){ \id_{B_0}\otimes{\Delta_0}} && B_0 \otimes B_0 \otimes B_0 }$$
	and the coassociativity follows.
	
	Consider the diagram
	$$\xymatrix{ A \ar[rr]^{\rho_0} \ar[rd]^{\rho_0} \ar[dd]_{\rho_0} && A \otimes B_0 \ar@{=}[ld]\ar[dd]^{\rho_0 \otimes \id_{B_0}} \\
		& A \otimes B_0 \ar[dd]^(0.3){\sim} & \\
		A \otimes B_0\ar[rd]^{\sim} \ar'[r][rr]^(0.3){\id_A\otimes {\Delta_0}} && A \otimes B_0 \otimes B_0 \ar[ld]^{\qquad\id_A\otimes{\varepsilon_0}\otimes\id_{B_0}} \\
		& A \otimes \mathbbm{1} \otimes B_0 & } $$
	
	The commutativity of the upper face and the left face is obvious. The commutativity of the right face
	follows from the definition of ${\varepsilon_0}$ and the commutativity of the rear face follows from the definition
	of ${\Delta_0}$. Therefore, the monoid homomorphisms forming the lower face become equal after their
	composition with $\rho_0$. 
	Now universal property of $B_0$ implies the commutativity of the diagram
	$$\xymatrix{ B_0 \ar[rr]^{{\Delta_0}} \ar[rd]^{\sim} 
		& & B_0\otimes B_0 \ar[ld]^{\qquad{\varepsilon_0}\otimes \id_{B_0}}\\
		& \mathbbm{1} \otimes B_0  & }$$
	Hence ${\varepsilon_0}$ satisfies the left counit axiom.
	
	Analogously, if we consider the diagram
	$$\xymatrix{ A \ar[rr]^{\rho_0}\ar[dd]_{\rho_0} \ar[rd]^{\sim} && A\otimes B_0
		\ar[ld]^{\qquad\id_A\otimes {\varepsilon_0}} \ar[dd]^{\rho_0 \otimes \id_{B_0}} \\
		& A\otimes \mathbbm{1} \ar[dd]^(0.3){\rho_0 \otimes \id_\mathbbm{1}} & \\
		A \otimes B_0\ar[rd]^{\sim} \ar'[r][rr]_(0.3){\id_A\otimes {\Delta_0}} && A \otimes B_0 \otimes B_0 \ar[ld]^{\qquad\id_A\otimes\id_{B_0}\otimes{\varepsilon_0}} \\
		& A \otimes B_0 \otimes \mathbbm{1} & } $$
	we get the commutativity of the diagram
	$$\xymatrix{ B_0 \ar[rr]^{{\Delta_0}} \ar[rd]^{\sim} 
		& & B_0\otimes B_0 \ar[ld]^{\qquad \id_{B_0}\otimes{\varepsilon_0}}\\
		& B_0 \otimes  \mathbbm{1} & }$$
	Hence ${\varepsilon_0}$ satisfies the right counit axiom and $B_0$ is indeed a bimonoid.
	
	Suppose $B$ is another bimonoid and $\rho \colon  A \to A\otimes B$ is a coaction
	that is an object in $\mathbf{Coact}(A)(\mathcal D)$.
	Denote by $\varphi \colon B_0 \to B$ the unique monoid homomorphism 
	making the diagram
	\begin{equation*}
		\xymatrix{ A \ar[r]^(0.3){\rho_0} \ar[rd]_\rho & A \otimes B_0 \ar@{-->}[d]^{\id_A\otimes
				\varphi} \\
			& A\otimes B} \end{equation*}
	commutative.
	We claim that $\varphi$ is a bimonoid homomorphism.
	
	Consider the diagram
	$$\xymatrix{  A \ar@{=}[rr]\ar[dd]_{\rho_0}\ar[rd]^{\rho_0} && A \ar'[d][dd]_{\rho}  \ar[rd]^{\rho}& \\
		& A \otimes B_0 \ar[rr]^(0.35){\id_A\otimes\varphi}\ar[dd]_(0.3){\rho_0 \otimes \id_{B_0}} && A\otimes B \ar[dd]^{\rho \otimes \id_{B}} \\
		A \otimes B_0 \ar'[r][rr]^(0.3){\id_A\otimes\varphi}\ar[rd]^{\id_A \otimes {\Delta_0}} && A\otimes B \ar[rd]^{\id_A \otimes \Delta}& \\
		& A \otimes B_0\otimes B_0 \ar[rr]^{\id_A\otimes\varphi\otimes\varphi} & & A\otimes B\otimes B } $$
	where $\Delta \colon B \to B\otimes B$ is the comultiplication in $B$.
	The upper face and the rear face are commutative by the definition of $\varphi$. The front face is commutative since it coincides with the upper face (and the rear face) tensored  by $\varphi$ from the right.
	The left and the right faces are commutative since both $\rho$ and $\rho_0$ define on $A$ the structure
	of a right comodule. 
	
	Hence, after the composition with $\rho_0$, the lower face becomes commutative too
	and the universal property of $B_0$ implies the commutativity of the diagram
	$$\xymatrix{B_0\ar[d]_{{\Delta_0}}  \ar[r]^{\varphi} & B \ar[d]^{\Delta}\\
		B_0 \otimes B_0  \ar[r]_(0.52){\varphi\otimes\varphi} & B \otimes B   \\ 
	}$$ Therefore $\varphi$ preserves the comultiplication.
	
	Consider the diagram $$\xymatrix{ & A \ar[lddd]_{\rho} 
		\ar[dd]^(0.7){\rho_0}  \ar[rddd]^{\sim} & \\ & & \\
		&  A  \otimes B_0 
		\ar[rd]_{\id_A\otimes{\varepsilon_0}\ } \ar[ld]^{\ \id_A \otimes \varphi}& \\
		A \otimes B \ar[rr]_{\id_A\otimes\varepsilon} & & A\otimes \mathbbm{1} }$$ where $\varepsilon$
	is the counit in $B$.
	
	The large triangle and the right triangle are commutative by the counitality of the comodule structures on $A$.
	The left triangle is commutative by the definition of $\varphi$. Hence the lower triangle becomes commutative
	after the composition with $\rho_0$. Thus the universal property of $B_0$ implies the commutativity of the diagram
	$$\xymatrix{ B_0\ar[rr]^{\varphi}\ar[rd]_{{\varepsilon_0}} & & B \ar[ld]^{\varepsilon}  \\
		& \mathbbm{1}   &}$$
	Therefore $\varphi$ preserves the counit and is indeed a bimonoid homomorphism.
\end{proof}	

Now we prove sufficient conditions for $\mathbf{Comeas}(A,A)_{G_1}(G\rho_U)$ to satisfy the assumptions of Theorem~\ref{TheoremBimonUnivCoactingExistenceForSubCat}.

\begin{lemma}\label{LemmaCompositionAndMonoidalMorTens}
Let $\rho_{ij} \colon A_i \to A_i \otimes Q_j$
 be morphisms for some objects $A_i$ and monoids $Q_j$ in a braided monoidal category $\mathcal C$
 for $i,j\in\lbrace 1,2 \rbrace$.
 Then $$\bigl((\rho_{11}\otimes \id_{Q_2}) \rho_{12}\bigr) \mathbin{\widetilde\otimes} \bigl((\rho_{21}\otimes \id_{Q_2}) \rho_{22}\bigr) = \bigl((\rho_{11} \mathbin{\widetilde\otimes} \rho_{21} )\otimes \id_{Q_2}\bigr) (\rho_{12} \mathbin{\widetilde\otimes} \rho_{22})$$
 where $\widetilde\otimes$ are the monoidal products in the categories $\mathsf{MorTens}(Q_1\otimes Q_2)$, $\mathsf{MorTens}(Q_1)$
 and $\mathsf{MorTens}(Q_2)$, respectively. (See Section~\ref{SubsectionComeasurings}.)
\end{lemma}	
\begin{proof}
	Use the properties of the braiding.
\end{proof}	

\begin{lemma}\label{LemmaCompositionComeasAndTrivial}
	If $\rho_1 \colon A \to A \otimes Q_1$ and $\rho_2 \colon A \to A \otimes Q_2$
	are comeasurings for some $\Omega$-magma $A$ and monoids $Q_1$ and $Q_2$ in a braided monoidal category $\mathcal C$, then 
	$(\rho_1 \otimes {\id_{Q_2}} )\rho_2$ and $A	\mathrel{\widetilde\to} A  \otimes \mathbbm{1}$
	are comeasurings too.
\end{lemma}	
\begin{proof}
Applying Lemma~\ref{LemmaCompositionAndMonoidalMorTens} inductively, we 
get $$\bigl((\rho_1 \otimes {\id_{Q_2}} )\rho_2\bigr)^{{}\mathbin{\widetilde\otimes} m}=(\rho_1^{{}\mathbin{\widetilde\otimes} m} \otimes {\id_{Q_2}} )\rho_2^{{}\mathbin{\widetilde\otimes} m}$$
for every $m\in\mathbb Z_+$.
Therefore,  $(\rho_1 \otimes {\id_{Q_2}} )\rho_2$ is a comeasuring.

The coherence theorem for braided monoidal categories
implies that $A	\mathrel{\widetilde\to} A  \otimes \mathbbm{1}$ is a comeasuring too.
\end{proof}	

\begin{lemma}\label{LemmaComeasAAU} 
	Let $A$ be an $\Omega$-magma in a braided monoidal category $\mathcal C$
	and let a tensor epimorphism $\rho_U \colon A \to A \otimes U$
	define on $A$ a structure of a $U$-comodule for some comonoid~$U$. 
	Then the morphism $A	\mathrel{\widetilde\to} A  \otimes \mathbbm{1}$ is an object in $\mathbf{Comeas}(A,A)_{G_1}(G\rho_U)$.
	Moreover, if $\rho_1 \colon A \to A \otimes Q_1$ and $\rho_2 \colon A \to A \otimes Q_2$
	are objects in $\mathbf{Comeas}(A,A)_{G_1}(G\rho_U)$, then $(\rho_1 \otimes {\id_{Q_2}} )\rho_2$ is an object in
	$\mathbf{Comeas}(A,A)_{G_1}(G\rho_U)$ too.
\end{lemma}
\begin{proof} 	 By Lemma~\ref{LemmaCompositionComeasAndTrivial}, $A	\mathrel{\widetilde\to} A  \otimes \mathbbm{1}$ is a comeasuring.
	By the definition of a comodule the diagram below is commutative:
	 $$\xymatrix{  & A \otimes U \ar[d]^{{\id_A}\otimes {\varepsilon_U}}\\
		 A \ar[r]^(0.4){\sim} \ar[ru]^{\rho_U} & A \otimes \mathbbm{1} }$$
	 (Here $\varepsilon_U$ is the counit in $U$.)
	 Hence $A	\mathrel{\widetilde\to} A  \otimes \mathbbm{1}$ is indeed an object in $\mathbf{Comeas}(A,A)_{G_1}(G\rho_U)$.
	
	 By Lemma~\ref{LemmaCompositionComeasAndTrivial}, the morphism $(\rho_1 \otimes {\id_{Q_2}} )\rho_2$
	is a comeasuring.
	Recall that from the definition of $\mathbf{Comeas}(A,A)_{G_1}(G\rho_U)$ it follows that there exist  such morphisms $\tau_i$ that $\rho_i = ({\id_A}\otimes \tau_i)\rho_U$ where $i=1,2$.
	Denote the comultiplication in $U$ by $\Delta_U$.
Then from the commutativity of the diagram 
	$$
	\xymatrix{ & A \otimes U \ar[rrrd]^{{\id_A}\otimes \Delta_U} & & & \\
		A \ar[ru]^{\rho_U} \ar[r]^{\rho_U} \ar[rd]^{\rho_2}  & A \otimes U \ar[rrr]^{\rho_U \otimes{\id_U}}
		\ar[d]^{{\id_A}\otimes {\tau_2}}
		& & & A \otimes U \otimes U  \ar[d]^{{\id_A}\otimes {\tau_1} \otimes {\tau_2}} \\
		& A \otimes {Q_2} \ar[rrr]^(0.4){\rho_1 \otimes {\id_{Q_2}}}  & & & A \otimes {Q_1} \otimes {Q_2} \\
	}	
	$$
it follows that $$(\rho_1 \otimes {\id_{Q_2}} )\rho_2 = ({\id_A} \otimes {({\tau_1} \otimes {\tau_2})\Delta_U})\rho_U$$ and
 $(\rho_1 \otimes {\id_{Q_2}} )\rho_2$  is indeed an object in $\mathbf{Comeas}(A,A)_{G_1}(G\rho_U)$.
	\end{proof}	

\begin{theorem}\label{TheoremBimonUnivCoactingExistence}  Let $\mathcal C$ be a braided monoidal category 
	and let $\rho_U \colon A \to A \otimes U$
	be a tensor epimorphism	defining on an $\Omega$-magma~$A$ a structure of a $U$-comodule for a comonoid $U$
	such that there exists $\mathcal{A}^\square(\rho_U)$.
	Then $ \mathcal{B}^\square(\rho_U):=\mathcal{A}^\square(\rho_U)$ admits a unique comonoid structure turning
	$\rho_U^\mathbf{Coact}:=\rho_U^\mathbf{Comeas}$ into a coaction, which is the initial object in $\mathbf{Coact}(A)_{G_2}(\rho_U)$.
\end{theorem}
\begin{proof} By Lemma~\ref{LemmaComeasAAU} the morphisms $(\rho_U^\mathbf{Coact}\otimes \id_{\mathcal{B}^\square(\rho_U)})\rho_U^\mathbf{Coact}$ and $A
	\mathrel{\widetilde\to} A  \otimes \mathbbm{1}$ are objects in $\mathbf{Comeas}(A,A)_{G_1}(G\rho_U)$.
	Now we apply Theorem~\ref{TheoremBimonUnivCoactingExistenceForSubCat} to the case $$\mathcal D  = \mathbf{Comeas}(A,A)_{G_1}(G\rho_U)$$ and then use Remark~\ref{RemarkCoactAG2RhoUEquivDef}.
\end{proof}	
\begin{remark}
Theorem~\ref{TheoremBimonUnivCoactingExistence} asserts that
if for some object $\rho_U$ in 
$\mathbf{ComodStr}(A)$ we have $$G \rho_U \in \LIO(\mathbf{MorTens}(A,A))$$
and there exists an initial object $\rho_U^\mathbf{Comeas}$ in $\mathbf{Comeas}(A,A)_{G_1}(G \rho_U)$,
then there exists an initial object $\rho_U^\mathbf{Coact}$ in 
$$\mathbf{Coact}(A)_{G_2}(\rho_U) = \mathbf{Coact}_{G_1 G_3}(A)(G \rho_U)$$
such that $G_3 \rho_U^\mathbf{Coact} = \rho_U^\mathbf{Comeas}$.
\end{remark}

\begin{corollary}\label{CorollaryBimonUnivCoactingExistence}  Suppose that a braided monoidal category $\mathcal C$
	satisfies  Properties~\ref{PropertySmallLimits}, \ref{PropertyEpiExtrMonoFactorizations},  \ref{PropertyFactorObjectsSmallSet}, \ref{PropertyMonomorphism},
	\ref{PropertyEpimorphism}, \ref{PropertySwitchProdTensorIsAMonomorphism}, \ref{PropertyFreeMonoid} of Section~\ref{SubsectionSupportCoactingConditions}. Let $\rho_U \colon A \to A \otimes U$
	be a tensor epimorphism	defining on an $\Omega$-magma~$A$ a structure of a $U$-comodule for a comonoid $U$.
	Then $ \mathcal{B}^\square(\rho_U):=\mathcal{A}^\square(\rho_U)$ admits a unique comonoid structure turning
	$\rho_U^\mathbf{Coact}:=\rho_U^\mathbf{Comeas}$ into a coaction, which is the initial object in $\mathbf{Coact}(A)_{G_2}(\rho_U)$.
\end{corollary}
\begin{proof} Apply Theorems~\ref{TheoremMonUnivComeasExistence} and~\ref{TheoremBimonUnivCoactingExistence}.
\end{proof}	
\begin{remark} Note that $A \mathrel{\widetilde\to} A \otimes \mathbbm{1}$ is a terminal object in $\mathbf{Coact}(A)_{G_2}(\rho_U)$,
	since for every bimonoid $B$ there exists the only bimonoid homomorphism $B \to \mathbbm{1}$, namely, the counit $\varepsilon$.
\end{remark}

\begin{examples}\label{Tambara-Manin}
\hspace{1cm}
\begin{enumerate}
\item\label{ExampleTambaraBialgebra}
	Let $A$ and $B$ be finite dimensional $\Omega$-algebras over a field $\mathbbm{k}$
	 with bases $a_1, \ldots, a_m$
	and $b_1, \ldots, b_n$, respectively.
	
	Denote by $U$ the vector space with the formal basis $(u_{ij})_{\substack{1\leqslant i \leqslant n, \\ 1\leqslant j \leqslant m}}$. Define the linear map $\rho_U \colon A \to B \otimes U$
	by $$\rho_{U}(a_j):=\sum\limits_{i=1}^n b_i \otimes u_{ij}\text{ for all }1\leqslant j \leqslant m.$$
	By Example~\ref{support} (\ref{ExampleVectMorTensSupp}), $\rho_{U}$ is a tensor epimorphism. 
	
	For every linear map $\rho \colon A \to B \otimes Q$ we have $\rho = ({\id_A} \otimes \tau)\rho_U$ where the linear map $\tau \colon U \to Q$ and elements $q_{ij} \in Q$ are defined by $\tau(u_{ij}):=q_{ij}$
	and  $$\rho(a_j)=\sum\limits_{i=1}^n b_i \otimes q_{ij}\text{ for all }1\leqslant j \leqslant m.$$
	By Theorem~\ref{TheoremMonUnivComeasExistence} there exists the universal comeasuring algebra $\mathcal A^{\square}(\rho_U)$,
	which is universal among all comeasurings. Here we recover Tambara's algebra $\alpha(B,A)$~\cite{Tambara}.
	If $A=B$, we can define on $U$ the structure of a coalgebra: $\Delta(u_{ij}):= \sum\limits_{k=1}^m u_{ik} \otimes u_{kj}$,
	$\varepsilon(u_{ij})=\delta_{ij}$ where $\delta_{ij}$ is Kronecker's delta. Then $A$ becomes a $U$-comodule and
	$\mathcal B^{\square}(\rho_U)=\mathcal A^{\square}(\rho_U)$ is the universal coacting bialgebra on $A$.
\item\label{ExampleManinBialgebra}
	Let $A=\bigoplus_{k\in \mathbb Z} A^{(k)}$ be an associative $\mathbb Z$-graded 
	unital algebra over some field $\mathbbm{k}$ such that $\dim A^{(k)} < +\infty$ for all $k\in\mathbb Z$.
	Choose bases $a^{(k)}_i$, where $1\leqslant i \leqslant n_k$, in each component $A^{(k)}$.
	Let $U$ be a vector space with basis $u^{(k)}_{ij}$ where $1\leqslant i,j \leqslant n_k$.
	Define the linear map $\rho_U \colon A \to B \otimes U$
	by $$\rho_{U}\left(a^{(k)}_j\right):=\sum\limits_{i=1}^n a^{(k)}_i \otimes u^{(k)}_{ij}\text{ for all }1\leqslant j \leqslant m.$$
		By Example~\ref{support} (\ref{ExampleVectMorTensSupp}), $\rho_U$ is a tensor epimorphism and for every linear map $\rho \colon A \to A \otimes Q$
	such that $$\rho\left(A^{(k)}\right)\subseteq A^{(k)} \otimes Q \text{ for every } k\in\mathbb Z$$ there exists
	a linear map $\tau \colon U \to Q$ such that $\rho = ({\id_A} \otimes \tau)\rho_U$.
	Indeed, we define the elements  $q^{(k)}_{ij} \in Q$ by 
	$$\rho\left(a^{(k)}_j\right)=\sum\limits_{i=1}^n b_i \otimes q^{(k)}_{ij}\text{ for all }1\leqslant j \leqslant n_k,\ k\in \mathbb Z.$$
	Then we put $\tau\left(u^{(k)}_{ij}\right):=q^{(k)}_{ij}$.
	We can again define on $U$ the structure of a coalgebra: $\Delta\left(u^{(k)}_{ij}\right):= \sum\limits_{\ell=1}^m u^{(k)}_{i\ell} \otimes u^{(k)}_{\ell j}$,
	$\varepsilon\left(u^{(k)}_{ij}\right)=\delta_{ij}$, and $A$ becomes a $U$-comodule.
	By Theorem~\ref{TheoremMonUnivComeasExistence}
	there exists the bialgebra $\mathcal{B}^\square(\rho_U)$ which is universal among all bialgebras that coact on $A$ preserving the grading. The bialgebra $\mathcal{B}^\square(\rho_U)$ is Yu.\,I.~Manin's universal coacting bialgebra $\underline{\mathrm{end}}(A)$~\cite{Manin}.
\item\label{ExampleUnivCoactionSets} Let $\mathcal{C} = \mathbf{Sets}$
	and let $A$ be an ordinary monoid. A map $\rho_U \colon A \to A \times U$
	endows $A$ with a structure of a $U$-comodule for a comonoid $U$
	(recall that comonoids in $U$ are just sets endowed with the diagonal maps)
	if and only if $\rho_U = (\id_A, \psi_U)$ for some map $\psi_U \colon A \to U$.
	Hence Example~\ref{ExampleCoactionSets} implies that the bimonoid $\mathcal{B}^\square(\rho_U)$ is just the monoid $A$ factored by the congruence generated by the kernel equivalence relation of $\psi_U$. If $\psi_U = \id_A$,
	then $\mathcal{B}^\square(\rho_U) \cong A$ and $\rho_U^\mathbf{Coact}$ is just a diagonal map $A \to A \times A$,
	which is universal among all comodule structures on $A$.
\end{enumerate}
\end{examples}

\subsection{Universal coacting Hopf monoids}\label{SubsectionUnivHopfCoactingExistence}
Again fix an $\Omega$-magma $A$ in a braided monoidal category~$\mathcal C$.
Define the category $\mathbf{HCoact}(A)$ where \begin{itemize}
	\item  the objects are all coactions $\rho \colon A \to A \otimes H$ for arbitrary Hopf monoids $H$;
	\item  the morphisms from $\rho_1 \colon A \to A \otimes H_1$
to $\rho_2 \colon A \to A \otimes H_2$  are Hopf monoid homomorphisms $\varphi \colon H_1 \to H_2$
making the diagram below commutative:
$$\xymatrix{ A \ar[r]^(0.4){\rho_1} 
	\ar[rd]_{\rho_2}
	& A \otimes H_1 \ar[d]^{\id_A \otimes \varphi} \\
	& A \otimes H_2} $$
\end{itemize}

Consider the following commutative diagram consisting of forgetful functors:
$$\xymatrix{
	\mathbf{HCoact}(A)  \ar[d]^{G_4}  \\
	\mathbf{Coact}(A) \ar[r]^{G_3} \ar[d]^{G_2} & \mathbf{Comeas}(A,A) \ar[d]^{G_1} \\
	\mathbf{ComodStr}(A) \ar[r]^G    & \mathbf{MorTens}(A,A) 
}$$	

Let $U$ be a comonoid and let a tensor epimorphism $\rho_U \colon A \to A \otimes U$
define on $A$ a structure of a $U$-comodule. We call the Hopf monoid 
corresponding to the initial object
in $\mathbf{HCoact}(A)_{G_2G_4}(\rho_U)$ (if it exists) the \textit{$U$-universal coacting Hopf monoid} on $A$.

\begin{remark}\label{RemarkHCoactAG2G4RhoUEquivDef}
	By Proposition~\ref{PropositionFromTensorEpiComonoidMorphism}, 
	the category $\mathbf{HCoact}(A)_{G_2G_4}(\rho_U)$  coincides with the category
	$\mathbf{HCoact}(A)_{GG_2G_4}(G\rho_U) = \mathbf{HCoact}(A)_{G_1G_3G_4}(G \rho_U)$.
\end{remark}

It turns out that in the case when there exist free Hopf monoids, one can lift initial objects
with respect to the forgetful functor $G_4\colon \mathbf{HCoact}(A)\to \mathbf{Coact}(A)$:

\begin{theorem}\label{TheoremHopfMonUnivCoactingExistenceForSubCat}
	Let	$A$ be an $\Omega$-magma in a braided monoidal category $\mathcal C$.
	Let $\mathcal D$ be a full subcategory of $\mathbf{Coact}(A)$ such that
	there exists an initial object in $\mathcal D$.
	Denote by  $\mathbf{HCoact}(A)(\mathcal D)$ the full subcategory of $\mathbf{HCoact}(A)$
	consisting of all the objects whose images under the forgetful functor $G_4\colon \mathbf{HCoact}(A)\to \mathbf{Coact}(A)$
	belong to $\mathcal D$. 	
	Suppose that the forgetful functor $\mathbf{Hopf}(\mathcal C) \to \mathbf{Bimon}(\mathcal C)$
	admits a left adjoint functor $H_l \colon  \mathbf{Bimon}(\mathcal C) \to \mathbf{Hopf}(\mathcal C)$.
	Then there exists an initial object in $\mathbf{HCoact}(A)(\mathcal D)$.	
\end{theorem}
\begin{proof}
	Let  $\rho_0 \colon A \to A \otimes B_0$ be an initial object in $\mathcal D$.
	Denote $H_0 := H_lB_0$ and define the coaction $\tilde\rho_0 \colon A \to A \otimes H_0$
	by $\tilde\rho_0 := ({\id_A}\otimes \eta_{B_0})\rho_0$
	where $\eta_B \colon B \to H_l B$ is the unit of the adjunction.
	Then for any Hopf monoid $H$ and any coaction $\rho \colon A \to A \otimes H$ from $\mathbf{HCoact}(A)(\mathcal D)$ there exists a unique Hopf monoid homomorphism $\varphi$
	making the diagram below commutative:
	$$\xymatrix{ & & A \otimes B_0 \ar[d]^{{\id_A}\otimes \eta_{B_0}} \\
		A \ar[rrd]_\rho \ar[rr]^(0.3){\tilde\rho_0} \ar@/^3ex/[rru]^(0.3){\rho_0} & & A \otimes H_0
		\ar@{-->}[d]^{\id_A \otimes \varphi} \\
		& & A \otimes H }$$
	Hence $\tilde\rho_0$ is the initial object in $\mathbf{HCoact}(A)(\mathcal D)$.
\end{proof}		

\begin{theorem}\label{TheoremHopfMonUnivCoactingExistence} Let $\mathcal C$ be a braided monoidal category 
	and let $\rho_U \colon A \to A \otimes U$
	be a tensor epimorphism	defining on an $\Omega$-magma~$A$ a structure of a $U$-comodule for a comonoid~$U$.
	 Suppose that the forgetful functor $\mathbf{Hopf}(\mathcal C) \to \mathbf{Bimon}(\mathcal C)$
	 admits a left adjoint functor $H_l \colon  \mathbf{Bimon}(\mathcal C) \to \mathbf{Hopf}(\mathcal C)$
	 and there exists $\mathcal{B}^\square(\rho_U)$.
	Then the initial object in $\mathbf{HCoact}(A)_{G_2G_4}(\rho_U)$ indeed exists.	
\end{theorem}
\begin{proof}
	Apply Theorem~\ref{TheoremHopfMonUnivCoactingExistenceForSubCat} for $\mathcal D = \mathbf{Coact}(A)_{G_2}(\rho_U)$.
\end{proof}

\begin{corollary}\label{CorollaryHopfMonUnivCoactingExistence}  Suppose that a braided monoidal category $\mathcal C$
	satisfies  Properties~\ref{PropertySmallLimits}--\ref{PropertyEpiExtrMonoFactorizations}, 	\ref{PropertyFactorObjectsSmallSet}, \ref{PropertyMonomorphism}, \ref{PropertyEpimorphism},  \ref{PropertySwitchProdTensorIsAMonomorphism}, \ref{PropertyFreeMonoid} of Section~\ref{SubsectionSupportCoactingConditions}.  Let $\rho_U \colon A \to A \otimes U$
	be a tensor epimorphism	defining on an $\Omega$-magma~$A$ a structure of a $U$-comodule for a comonoid $U$ in~$\mathcal C$.
	Then the initial object in $\mathbf{HCoact}(A)_{G_2G_4}(\rho_U)$ indeed exists.
\end{corollary}
\begin{proof}
	Apply Theorems~\ref{TheoremBimonHopfLeftAdjoint}, \ref{TheoremHopfMonUnivCoactingExistence} and Corollary~\ref{CorollaryBimonUnivCoactingExistence}.
\end{proof}	

\begin{remark} Again, $A \mathrel{\widetilde\to} A \otimes \mathbbm{1}$ is a terminal object in $\mathbf{HCoact}(A)_{G_2G_4}(\rho_U)$,
	since for every Hopf monoid $H$ there exists the only Hopf monoid homomorphism $H \to \mathbbm{1}$, namely, the counit~$\varepsilon$.
\end{remark}

\begin{examples}\label{Manin-Hopf}
\hspace{1cm}
\begin{enumerate}
\item If $\mathcal C = \mathbf{Vect}_\mathbbm{k}$, $A$ is an $\Omega$-algebra over a field $\mathbbm{k}$ and $U$ and $\rho_U$ are, respectively, the coalgebra
	and the linear map from Example~\ref{Tambara-Manin} (\ref{ExampleTambaraBialgebra}),
	then the Hopf algebra $\mathcal{H}^\square(\rho_U)$ is the universal coacting Hopf algebra on $A$.
\item	 Let $A=\bigoplus_{n\in \mathbb Z} A^{(n)}$ be an associative $\mathbb Z$-graded 
	unital algebra such that $\dim A^{(n)} < +\infty$ for all $n\in\mathbb Z$.
	If $U$ and $\rho_U$ are, respectively, the coalgebra
	and the linear map from Example~\ref{Tambara-Manin} (\ref{ExampleManinBialgebra}), then $\mathcal{H}^\square(\rho_U)$
	is exactly Yu.\,I.~Manin's universal coacting Hopf algebra $\underline{\mathrm{aut}}(A)$~\cite{Manin}.
\item	 Let $\rho \colon  A \to A \otimes H$ be some Hopf monoid coaction on an $\Omega$-magma $A$
	with an absolute value $|\rho|_{G_2 G_4} \colon  A \to A \otimes (\supp \rho)$
	in a braided monoidal category
	$\mathcal C$ satisfying Properties~\ref{PropertySmallLimits}--\ref{PropertyFreeMonoid},
	\ref{PropertyFactorObjectsSmallSet}, \ref{PropertyEpimorphism} and \ref{PropertyExtrMonomorphism} of Section~\ref{SubsectionSupportCoactingConditions}.
	Then by Propositions~\ref{PropositionLIOCoarserFinerEquivalent} and~\ref{PropositionLIOAbsValueLiftCriterion} the $\mathcal{H}^\square(|\rho|_{G_2 G_4})$-coaction $|\rho|^{\mathbf{HCoact}}_{G_2 G_4}$ on $A$ is equivalent to $\rho$
	and is universal among all coactions equivalent to or coarser than $\rho$.
	We call $\mathcal{H}^\square(|\rho|_{G_2 G_4})$
	\textit{the universal Hopf monoid} of $\rho$ and in the case $\mathcal C = \mathbf{Vect}_\mathbbm{k}$ for a field $\mathbbm{k}$
	this Hopf monoid is exactly the universal Hopf algebra of $\rho$ introduced in~\cite{AGV1}.
\item\label{ExampleUnivHopfCoactionSets} 
	Let $\mathcal{C} = \mathbf{Sets}$. Recall that $\mathsf{Bimon}(\mathbf{Sets})= \mathbf{Mon}$, the category of ordinary monoids, and
	$\mathsf{Hopf}(\mathbf{Sets})= \mathbf{Grp}$, the category of groups. (The comultiplication in the objects is the diagonal map.)
	Recall that the left adjoint functor $H_l$ assigns to each monoid its Grothendieck group (see Example~\ref{ExampleHopfLeftAdjointSets}). As before, denote by $\eta$  the unit of this adjunction.
	Let $A$ be an ordinary monoid. By Theorem~\ref{TheoremHopfMonUnivCoactingExistence}, $\mathcal{H}^\square(\rho_U)=H_l\left( \mathcal{B}^\square(\rho_U)\right)$. Recall that 
	by Example~\ref{Tambara-Manin} (\ref{ExampleUnivCoactionSets}) the monoid $\mathcal{B}^\square(\rho_U)$ is
	just the monoid $A$ factored by the congruence generated by the kernel equivalence relation of $\psi_U$ where $\psi_U \colon A \to U$
	is the map defined by $\rho_U = (\id_A, \psi_U)$. If $\psi_U = \id_A$,
	then $\mathcal{H}^\square(\rho_U)$ is isomorphic to the Grothendieck group $H_l A$ of $A$ itself and $\rho_U^\mathbf{HCoact} = (\id_A, \eta_A)$. This action $\rho_U^\mathbf{HCoact} \colon A \to A \times H_l A$
	is universal among all group comodule structures on $A$.
\end{enumerate}
\end{examples}

\begin{remark}
	Theorem~\ref{TheoremHopfMonUnivCoactingExistence} asserts that
	if for some object $\rho_U$ in 
	$\mathbf{ComodStr}(A)$ we have $G \rho_U \in \LIO(\mathbf{MorTens}(A,A))$
	and there exists an initial object $\rho_U^\mathbf{Coact}$ in $\mathbf{Coact}(A)_{G_2}(\rho_U) = \mathbf{Coact}(A)_{GG_2}(G \rho_U)$,
	then there exists an initial object $\rho_U^\mathbf{HCoact}$ in 
	$\mathbf{HCoact}(A)_{G_2G_4}(\rho_U) = \mathbf{HCoact}_{G G_2G_4}(A)(G \rho_U)$. In other words, the theorem makes it possible to lift such locally initial objects from $\mathbf{Coact}(A)$ to $\mathbf{HCoact}(A)$.
\end{remark}

\section{Existence theorems for cosupports and universal acting bi- and Hopf monoids}\label{SectionDUALCosupportActingExistence}

In this section we introduce concepts and list results dual to those of Section~\ref{SectionSupportCoactingExistence}.

\subsection{Conditions on the base category}\label{SubsectionDUALCosupportActingConditions}

Consider a braided monoidal category $\mathcal C$ with a monoidal product $\otimes$, a braiding $c_{M,N} \colon
M \otimes N \mathrel{\widetilde\to} N \otimes M$ and natural isomorphisms $a_{L,M,N} \colon
(L \otimes M) \otimes N \mathrel{\widetilde\to} L \otimes (M \otimes N)$, $l_M \colon \mathbbm{1} \otimes M
\mathrel{\widetilde\to} M$ and $r_M \colon M  \otimes \mathbbm{1} \mathrel{\widetilde\to} M$.

Denote by $\mathcal C^{\mathrm{op},\mathrm{rev}}$ (here ``op'' means that we consider the opposite category and ``rev'' means that we consider the reverse monoidal product) the braided monoidal category that is isomorphic to the opposite category $\mathcal C^{\mathrm{op}}$ as an ordinary category, with the monoidal product $A\, \mathbin{\otimes^\mathrm{rev}}B := B \otimes A$,
the braiding $c_{M,N}^{\mathrm{op},\mathrm{rev}} := c_{M,N}$ and the natural isomorphisms
$a_{L,M,N}^{\mathrm{op},\mathrm{rev}} := a_{N,M,L}$, $l_M^{\mathrm{op},\mathrm{rev}}:= r_M^{-1}$ and 
$r_M^{\mathrm{op},\mathrm{rev}}:= l_M^{-1}$.

Note that $\mathcal C^{\mathrm{op},\mathrm{rev}}$ satisfies Properties~\ref{PropertySmallLimits}--\ref{PropertyFreeMonoid},
 \ref{PropertyFactorObjectsSmallSet}, \ref{PropertyEpimorphism} and~\ref{PropertyExtrMonomorphism}
of Section~\ref{SubsectionSupportCoactingConditions} if and only if  the original category~$\mathcal C$ satisfies
the following properties:

\begin{enumerate}
	\myitem[1*]\label{PropertyDUALSmallColimits} there exist all small colimits in $\mathcal C$;
	\myitem[2*]\label{PropertyDUALFiniteAndCountableLimits} there exist finite and countable limits in $\mathcal C$;
	\myitem[3*]\label{PropertyDUALExtrEpiMonoFactorizations} $\mathcal C$ is (ExtrEpi, Mono)-structured;
 	\myitem[4]\label{PropertyDUALSubObjectsSmallSet}
 $\mathcal C$ is wellpowered;	
	\myitem[4*]\label{PropertyDUALFactorObjectsSmallSet}
 $\mathcal C$ is cowellpowered;	
	\myitem[5]\label{PropertyDUALMonomorphism} for every monomorphism $f$ and every object $M$
both
$f \otimes \id_M$ and $\id_M \otimes f$ are monomorphisms too; 
	\myitem[5*]\label{PropertyDUALEpimorphism} for every epimorphism $f$ and every object $M$
both
$f \otimes \id_M$ and
$\id_M \otimes f$ are epimorphisms too; 
\myitem[5a*]\label{PropertyDUALExtrEpimorphism} for every extremal epimorphism $f$ the morphism
$f \otimes f$ is an extremal epimorphism too; 
	\myitem[6*]\label{PropertyDUALColimitsOfFactorObjectsArePreserved} for every object $M$ the functor $ (-) \otimes M$ preserves colimits of extremal quotient objects in $ \mathcal C$ (see Remark~\ref{RemarkDUALFactorObjects} below);
	\myitem[7*]\label{PropertyDUALTensorPushout}
	for every pushout
	 $$\xymatrix{ A \ar[r]^{g} \ar@{->>}[d]_{f}  & C \ar@{->>}[d]^h \\
		B \ar[r]^t & P \ar@{}[lu]|<{\pushout}}
	$$ where $f$ is an arbitrary epimorphism and $g$ is an arbitrary morphism having the same domain $A$ 
	(recall that in this case $h$ is automatically an epimorphism too) the diagram below is a pushout too:
	$$\xymatrix{ A\otimes M \ar[r]^{g\otimes{\id_M}} \ar@{->}[d]_{f\otimes{\id_M}}  & C\otimes M \ar@{->}[d]^{h\otimes{\id_M}} \\
		B\otimes M \ar[r]^{t\otimes{\id_M}} & P\otimes M \ar@{}[lu]|<{\pushout}}
	$$
	\myitem[8*]\label{PropertyDUALSwitchCoprodTensorIsAnEpimorphism} 
		for any nonempty small set $\Lambda$ and any objects $M$ and $A_\alpha$, where $\alpha \in \Lambda$, the unique morphism~$\varphi$ making the diagram below commutative, is an epimorphism: $$\xymatrix{\coprod\limits_{\alpha\in\Lambda}( A_\alpha \otimes M ) \ar[rr]^\varphi
	& \qquad\quad& \left( \coprod\limits_{\alpha\in\Lambda} A_\alpha \right) \otimes M \\
& A_\alpha \otimes M \ar[lu]^{\tilde \imath_\alpha} \ar[ru]_{ {i_\alpha} \otimes { \id_M}} &
}$$ (here $i_\alpha \colon A_\alpha
\to \coprod\limits_{\alpha\in\Lambda}A_\alpha$ and $\tilde \imath_\alpha \colon A_\alpha \otimes M
\to \coprod\limits_{\alpha\in\Lambda} A_\alpha \otimes M$ are the morphisms from the definition of the coproduct, $\alpha \in \Lambda$);	
	\myitem[9*]\label{PropertyDUALCoequalizers}
	for every object $M$ the functor $(-) \otimes M$
	preserves all coequalizers;
	\myitem[10*]\label{PropertyDUALCofreeComonoid} the forgetful functor $\mathsf{Comon}(\mathcal C) \to \mathcal C$
	has a right adjoint $\mathcal{G} \colon  \mathcal C \to \mathsf{Comon}(\mathcal C)$.
\end{enumerate}

\begin{remark}\label{RemarkDUALFactorObjects}
	Properties~\ref{PropertyDUALSmallColimits} and~\ref{PropertyDUALFactorObjectsSmallSet} imply that there exist colimits of any families of quotient objects,
	i.e. if $\varphi_\alpha \colon A \twoheadrightarrow B_\alpha$ are epimorphisms
	for some set $\Lambda$ and objects $A$, $B_\alpha$, where $\alpha \in \Lambda$,
	then there exists $\colim T$ where $T \colon \Lambda \cup \lbrace 0 \rbrace \to \mathcal C$, $\Lambda \cup \lbrace 0 \rbrace$
	is the category with the set of objects $\Lambda \cup \lbrace 0 \rbrace$ and
	either only the arrows $0 \to  \alpha $ or, in addition, some arrows $\alpha \to \beta$
	such that $\varphi_\beta = \varphi_{\alpha\beta}\varphi_\alpha$ for some morphism $\varphi_{\alpha\beta} \colon B_\alpha \to B_\beta$  (the resulting $\colim T$ will not depend on whether we include $\alpha \to \beta$ or not)
	and the functor $T$ is defined as follows: $T\alpha = B_\alpha$, $T0=A$, $T( 0 \to\alpha) = \varphi_\alpha$,
	$T(\alpha \to \beta) =\varphi_{\alpha\beta}$ for all $\alpha,\beta \in \Lambda$.
	
\end{remark}	

In the rest of Section~\ref{SectionDUALCosupportActingExistence} we assume that $\mathcal C$ satisfies some of the Properties~\ref{PropertyDUALSmallColimits}--\ref{PropertyDUALCofreeComonoid}, \ref{PropertyDUALSubObjectsSmallSet}, \ref{PropertyDUALMonomorphism} and~\ref{PropertyDUALExtrEpimorphism} and we will dualize
the results of Section~\ref{SectionSupportCoactingExistence} by
applying them to the category $\mathcal C^{\mathrm{op},\mathrm{rev}}$. Below we just list the definitions and the propositions obtained in this way.

\begin{proposition}\label{PropositionDUALPropertyCoequilizers} 
\begin{enumerate}
\item Property~\ref{PropertyDUALExtrEpiMonoFactorizations}
	follows from Properties~\ref{PropertyDUALSmallColimits} and~\ref{PropertyDUALFactorObjectsSmallSet};
\item Property~\ref{PropertyDUALCoequalizers} follows from 	Properties~\ref{PropertyDUALSmallColimits}, \ref{PropertyDUALTensorPushout}  and~\ref{PropertyDUALSwitchCoprodTensorIsAnEpimorphism}.
\end{enumerate}
\end{proposition}
\begin{proof}
Proposition~\ref{PropositionDUALPropertyCoequilizers} is dual to Proposition~\ref{PropositionPropertyEquilizers}. 	
\end{proof}	

\subsection{Examples}\label{ExamplesBaseCaategory}

As in Section~\ref{SubsectionSupportCoactingConditions}, the basic examples of categories $\mathcal C$ satisfying Properties~\ref{PropertyDUALSmallColimits}--\ref{PropertyDUALCofreeComonoid}, \ref{PropertyDUALSubObjectsSmallSet}, \ref{PropertyDUALMonomorphism} and~\ref{PropertyDUALExtrEpimorphism} are $\textbf{Sets}$ (with the Cartesian monoidal product), $\textbf{Sets}^\mathrm{op}$  (with the same monoidal product as in $\textbf{Sets}$), $\textbf{Vect}_\mathbbm{k}$ and $\textbf{Vect}_\mathbbm{k}^\mathrm{op}$ for a field $\mathbbm{k}$. We will now introduce further examples: (co)modules over bialgebras, $G$-sets and graded sets. 

\subsubsection{Modules over Hopf algebras} Let $B$ be a bialgebra over a field $\mathbbm k$. Then the forgetful functor ${}_B \mathsf{Mod} \to \mathbf{Vect}_{\mathbbm k}$, where ${}_B \mathsf{Mod}$ is the category of left $B$-modules and the monoidal product coincides with the tensor product $\otimes$ over $\mathbbm k$, is a functor that creates small limits and colimits as well as limits of subobjects and colimits of quotient objects. Moreover, the free (tensor) algebra $T(M)$ of a $B$-module $M$ inherits the structure of a $B$-module, which makes $T(M)$ a $B$-module algebra. In addition, all monomorphisms and epimorphisms in ${}_B \mathsf{Mod}$ are extremal.
Finally, by~\cite[Proposition 4.1]{AbdulIovanov} the forgetful functor $\mathsf{Comon}({}_B \mathsf{Mod}) \to {}_B \mathsf{Mod} $
admits a right adjoint.
 Therefore, the category ${}_B \mathsf{Mod}$
satisfies Properties~\ref{PropertySmallLimits}--\ref{PropertyFreeMonoid}, \ref{PropertyDUALSmallColimits}--\ref{PropertyDUALCofreeComonoid}, \ref{PropertyExtrMonomorphism} and \ref{PropertyDUALExtrEpimorphism} 
of Sections~\ref{SubsectionSupportCoactingConditions} and~\ref{SubsectionDUALCosupportActingConditions}. 
When ${}_B \mathsf{Mod}$ is braided, e.g. $B$ is a quasitriangular Hopf algebra,
we can apply to $\mathcal C={}_B \mathsf{Mod}$ all the results of the present paper.

\subsubsection{Comodules over Hopf algebras}\label{SubsectionApplicationsComodules}
Throughout we use Sweedler's notation, namely $\rho(m) = m_{(0)}\otimes m_{(1)}$, $m\in M$ will stand for the linear map 
$\rho \colon M \to M \otimes C$ defining a structure of a right $C$-comodule on a $\mathbbm k$-vector space $M$. 

Consider now the forgetful functor $\mathsf{Comod}^B \to \mathbf{Vect}_{\mathbbm k}$
where $\mathsf{Comod}^B$ is the category of right $B$-comodules, $B$ is a bialgebra over a field $\mathbbm k$ and the monoidal product again coincides with the tensor product $\otimes$ over $\mathbbm k$. This functor creates finite limits and small colimits as well as limits of subobjects and colimits of quotient objects.
If $M_\alpha$, $\alpha \in \Lambda$,
are right $B$-comodules, then their product in $\mathsf{Comod}^B$
is the subspace of their Cartesian product $\prod\limits_{\alpha \in \Lambda} M_\alpha$
consisting of all tuples $(m_\alpha)_{\alpha \in \Lambda}$, $m_\alpha \in M_\alpha$,
such that for each tuple there exists a single finite dimensional subcoalgebra $C\subseteq B$
where $\rho(m_\alpha) \in M_\alpha \otimes C$ for all $\alpha \in \Lambda$. Again, the free algebra $T(M)$ of a $B$-comodule $M$ inherits the structure of a $B$-comodule, which makes $T(M)$ a $B$-comodule algebra. In addition, all monomorphisms and epimorphisms in $\mathsf{Comod}^B$ are extremal. 
Finally, by~\cite[Proposition 4.1]{AbdulIovanov} the forgetful functor $\mathsf{Comon}(\mathsf{Comod}^B) \to \mathsf{Comod}^B$
admits a right adjoint.
We can conclude that $\mathsf{Comod}^B$ satisfies Properties~\ref{PropertySmallLimits}--\ref{PropertyFreeMonoid}, \ref{PropertyDUALSmallColimits}--\ref{PropertyDUALCofreeComonoid}, \ref{PropertyExtrMonomorphism} and \ref{PropertyDUALExtrEpimorphism} 
	of Sections~\ref{SubsectionSupportCoactingConditions} and~\ref{SubsectionDUALCosupportActingConditions}.
When $\mathsf{Comod}^B$ is braided, e.g. $B$ is a coquasitriangular Hopf algebra,
we can apply to $\mathcal C=\mathsf{Comod}^B$ all the results of the present paper.

A notable example of this type is the category of differential graded vector spaces. Let $\mathbbm{k}$ be a field. Denote by $\mathbf{dgVect}_\mathbbm{k}$ the category of differential $\mathbb{Z}$-graded vector spaces (or dg-vector spaces, for short) or, in another terminology, chain complexes in $\mathbf{Vect}_\mathbbm{k}$. Objects in $\mathbf{dgVect}_\mathbbm{k}$
are families $(V_n)_{n\in\mathbb Z}$ of vector spaces $V_n$ equipped with linear maps $d \colon V_n \to V_{n-1}$, $n\in\mathbb Z$,
such that $d^2 = 0$. The maps $d$ are called \textit{differentials}. Every family $(V_n)_{n\in\mathbb Z}$ can be identified with its $\mathbb{Z}$-graded total space $V=\bigoplus\limits_{n\in\mathbb Z} V_n$. Moreover $d$ extends to a graded linear map $V \to V$ of degree $(-1)$ such that $d^2 = 0$. Morphisms in $\mathbf{dgVect}_\mathbbm{k}$ are grading preserving (= graded of degree $0$) linear maps commuting with $d$.
Note that $\mathbf{dgVect}_\mathbbm{k}$ is an abelian category where limits and colimits are computed componentwise.

Let $U=\bigoplus\limits_{k\in\mathbb Z} U_k$ and $V=\bigoplus\limits_{m\in\mathbb Z} V_m$ be two dg-vector spaces. Then the monoidal product $W=U\otimes V$ in $\mathbf{dgVect}_\mathbbm{k}$ is defined by $W:=\bigoplus\limits_{n\in\mathbb Z} W_n$ where $$W_n := \bigoplus\limits_{k\in\mathbb Z} U_k \otimes V_{n-k}.$$ The differentials $d \colon W_n \to W_{n-1}$
are defined by $$d(u\otimes v) := du \otimes v + (-1)^{k} u\otimes dv\text{ for }u\in U_k\text{ and }v\in V_m,\ k,m\in\mathbb Z.$$
The monoidal unit in $\mathbf{dgVect}_\mathbbm{k}$ is $\mathbbm{k}$ regarded as a chain complex concentrated in degree $0$ with zero differential.
The category $\mathbf{dgVect}_\mathbbm{k}$ is symmetric where the swap $c \colon U_k \otimes V_m \mathbin{\widetilde\to} V_m \otimes U_k$ is defined by $c(u\otimes v) := (-1)^{km} v \otimes u$ for all $u\in U_k$, $v\in V_m$, $k,m\in\mathbb Z$.

Monoids in $\mathbf{dgVect}_\mathbbm{k}$ are just unital associative differential graded algebras (or dg-algebras for short).

\begin{theorem}\label{TheoremDgVectSatisfiesProperties}
	Let $\mathbbm k$ be a field. 
	Then $\mathbf{dgVect}_\mathbbm{k}$ is a symmetric monoidal category satisfying Properties~\ref{PropertySmallLimits}--\ref{PropertyFreeMonoid}, \ref{PropertyDUALSmallColimits}--\ref{PropertyDUALCofreeComonoid},
	\ref{PropertyExtrMonomorphism} and \ref{PropertyDUALExtrEpimorphism} 
	of Sections~\ref{SubsectionSupportCoactingConditions} and~\ref{SubsectionDUALCosupportActingConditions}.
	Moreover, all monomorphisms and epimorphisms in $\mathbf{dgVect}_\mathbbm{k}$ are extremal.
	\end{theorem}
\begin{proof}
	Consider the Hopf algebra $H$ over $\mathbbm k$ with basis $c^k v^\ell$, where $k\in \mathbb Z$, $\ell = 0,1$,
	$vc = -cv$, $v^2=0$. $\Delta v = c\otimes v + v \otimes 1$, $\Delta c = c\otimes c$, $Sc= c^{-1}$, $Sv=-c^{-1} v$.
	Then $\mathbf{dgVect}_\mathbbm{k}$ can be identified with $\mathsf{Comod}^H$
	where for every dg-vector space $(V_m)_{m\in\mathbb Z}$ 
	the structure of a right $H$-comodule on $\bigoplus_{m\in \mathbb Z} V_m$
	is given by
	$\rho(a) := a\otimes c^{-m} + da \otimes vc^{-m}$ for $a \in V_m$, $m\in\mathbb Z$,
	and if $V$ is a right $H$-comodule, then $V_m := \lbrace a \in V \mid \lambda_m(a_{(1)}) a_{(0)}=a \rbrace$,
	$da := \mu(a_{(1)}) a_{(0)}$. Here $\lambda_m, \mu \in H^*$ are defined by
	$\mu(c^k v^\ell) := \delta_{\ell 1}$, $\lambda_m(c^k v^\ell) := \delta_{k,-m}\delta_{\ell 0}$.
	(Note that $\mu^2 = 0$ and $\lambda_{m-1} \mu = \mu \lambda_m$.)
\end{proof}	
	
\subsubsection{$G$-sets} Let $G$ be a group. Then the category ${}_G \mathsf{Mod}$
of left $G$-modules in $\mathbf{Sets}$, i.e. sets $M$ with a fixed homomorphism $G \to S(M)$
where $S(M)$ is the symmetric group on $M$, is a symmetric monoidal category with the Cartesian monoidal product
and the ordinary swap $c_{M,N} \colon M\times N \to N\times M$
where $c_{M,N}(m,n):= (n,m)$ for all $m \in M$ and $n\in N$.
The objects in ${}_G \mathsf{Mod}$ are often called \textit{$G$-sets}.

\begin{lemma}\label{LemmaGSetsMonEpi} Monomorphisms in  ${}_G \mathsf{Mod}$ are injective $G$-module homomorphisms.
Epimorphisms in  ${}_G \mathsf{Mod}$ are surjective $G$-module homomorphisms. 
\end{lemma}
\begin{proof}
	Let $f \colon M \to N$ be $G$-module homomorphism.
	
	Suppose $f(m_1)=f(m_2)$ for some $m_1\ne m_2$, $m_1,m_2\in M$.
	Then $ff_1 = ff_2$ where $f_i \colon G \to M$ is defined by $f_i(g)=g m_i$ for all $g\in G$ and $G$ is endowed with the $G$-action by left shifts, however $f_1 \ne f_2$. Hence $f$ is not a monomorphism.
	
	Note that the image $f(M)$ of $f$ is a $G$-submodule, i.e. the union of some orbits. Suppose $f(M)\ne N$. Then there exists an orbit $O \subseteq N$ such that $O \cap f(M) = \varnothing$. Let $X := (N \backslash O) \sqcup \lbrace x_1,x_2\rbrace$,
	where $Gx_i=x_i$ for $i=1,2$. Define $f_i \colon N \to X$ by $f_i(n)=n$ for all $n\in N \backslash O$
	and $f_i(O)=x_i$, $i=1,2$. Then $f_1 f = f_2 f$, however $f_1 \ne f_2$. Hence $f$ is not an epimorphism.
\end{proof}		
 
The forgetful functor ${}_G \mathsf{Mod} \to \mathbf{Sets}$ creates small limits and colimits as well as limits of subobjects and colimits of quotient objects.  Again, the free monoid $\langle X \rangle$ for a $G$-module $X$ inherits the structure of an $G$-module, which makes $\langle X \rangle$ an $G$-module monoid. Recall that all comonoids in $\mathbf{Sets}$ are just sets
$X$ with the diagonal comultiplication $\Delta \colon X \to X \times X$ and the trivial counit $\varepsilon \colon X \to \lbrace * \rbrace$. Hence each set is a cofree comonoid on itself and the same is true in the category ${}_G \mathsf{Mod}$ too.

In other words, the following theorem holds:
\begin{theorem}\label{TheoremGSetsSatisfiesProperties}
	Let $G$ be a group. Then ${}_G \mathsf{Mod}$ is a symmetric monoidal category 
	satisfying Properties~\ref{PropertySmallLimits}--\ref{PropertyFreeMonoid}, \ref{PropertyDUALSmallColimits}--\ref{PropertyDUALCofreeComonoid},
	\ref{PropertyExtrMonomorphism} and \ref{PropertyDUALExtrEpimorphism} 
	of Sections~\ref{SubsectionSupportCoactingConditions} and~\ref{SubsectionDUALCosupportActingConditions}.
		Moreover, all monomorphisms and epimorphisms in ${}_G \mathsf{Mod}$ are extremal.
\end{theorem}

\subsubsection{Graded sets} Let $M$ be a monoid (in $\mathbf{Sets}$). Then the category $\mathsf{Comod}^M$ of right $M$-comodules in $\mathbf{Sets}$ is just the category of maps $X \to M$ (which we will denote by the same symbol $\deg$) for arbitrary sets $X$ endowed with the Cartesian (with respect to $\mathbf{Sets}$) monoidal product, where $\deg \colon X \times Y \to M$
is defined by $\deg(x,y):=\deg(x)\deg(y)$ for all $x\in X$ and $y\in Y$. The category $\mathsf{Comod}^M$ can be interpreted as the category of $M$-graded sets, i.e. sets $X$ decomposed into disjoint unions $X=\bigsqcup\limits_{m\in M} X^{(m)}$ of subsets $X^{(m)}$ marked by elements of $M$, i.e. $\deg x = m$ for all $x\in X^{(m)}$, $m\in M$.

The forgetful functor $\mathsf{Comod}^M \to \mathbf{Sets}$ is a strict monoidal functor creating equalizers and
small colimits as well as limits of subobjects and colimits of quotient objects. Moreover, since morphisms in $\mathsf{Comod}^M$ are just grading preserving maps, all monomorphisms are again injective and all epimorphisms are surjective. 

If $X_\alpha$, $\alpha \in \Lambda$, are $M$-graded sets, then their product in $\mathsf{Comod}^M$
is the subset of their Cartesian product $\prod\limits_{\alpha \in \Lambda} X_\alpha$
consisting of all tuples $(x_\alpha)_{\alpha \in \Lambda}$ where $x_\alpha \in X_\alpha$
and $\deg x_\alpha = \deg x_\beta$ for all $\alpha,\beta \in \Lambda$.
The terminal object in $\mathsf{Comod}^M$ is $M$ itself with the standard grading $M = \bigsqcup\limits_{m\in M} M^{(m)}$
where $M^{(m)} := \lbrace m \rbrace$.

Again, the free monoid $\langle X \rangle$ for an $M$-graded set $X$ inherits the $M$-grading, which makes $\langle X \rangle$ an $M$-graded monoid.

For a graded set $X$ a graded map $X\to\lbrace * \rbrace$ exists if and only if $X$ is trivially graded, i.e. coincides with its homogeneous component $X^{(e)}$ where $e\in M$ is the identity element. Hence comonoids in $\mathsf{Comod}^M$
are just trivially graded sets $X$ with the diagonal comultiplication $\Delta \colon X \to X \times X$ and the trivial counit $\varepsilon \colon X \to \lbrace * \rbrace$. Hence for an arbitrary graded set $Y$ its cofree comonoid $\mathcal G Y$ in $\mathsf{Comod}^M$ is its neutral component $Y^{(e)}$.

 Therefore we get

\begin{theorem}\label{TheoremGradedSetsSatisfiesProperties}
	Let $M$ be a monoid (in $\mathbf{Sets}$). Then $\mathsf{Comod}^M$
	is a monoidal category
	satisfying Properties~\ref{PropertySmallLimits}--\ref{PropertyFreeMonoid}, \ref{PropertyDUALSmallColimits}--\ref{PropertyDUALCofreeComonoid},
	\ref{PropertyExtrMonomorphism} and \ref{PropertyDUALExtrEpimorphism}
	of Sections~\ref{SubsectionSupportCoactingConditions} and~\ref{SubsectionDUALCosupportActingConditions}.
	If $M$ is commutative, then $\mathsf{Comod}^M$ is symmetric with
	the ordinary swap. Moreover, all monomorphisms and epimorphisms in $\mathsf{Comod}^M$ are extremal.
\end{theorem}

\subsubsection{Set-theoretic Yetter~--- Drinfel'd modules}
Let $G$ be a group. Denote by ${}^G_G \mathcal{YD}$ the category of \textit{set-theoretic Yetter~--- Drinfel'd modules} (or
 \textit{${}^G_G \mathcal{YD}$-modules} for short), i.e. $G$-graded sets $X=\bigsqcup\limits_{g\in G} X^{(g)}$ that are $G$-modules and 
 $g X^{(t)} = X^{\left(gtg^{-1}\right)}$ for all $g,t\in G$. Consider the Cartesian (with respect to $\mathbf{Sets}$) monoidal product in ${}^G_G \mathcal{YD}$.
 
 The forgetful functor ${}^G_G \mathcal{YD} \to \mathsf{Comod}^G$ is a strict monoidal functor creating small limits and colimits as well as limits of subobjects and colimits of quotient objects. The terminal object in ${}^G_G \mathcal{YD}$ is $G$ with the standard grading and the action on itself by conjugation. 
 
 \begin{lemma}\label{LemmaYDSetsMonEpi} Monomorphisms in  ${}^G_G \mathcal{YD}$ are injective ${}^G_G \mathcal{YD}$-module homomorphisms. Epimorphisms in  ${}^G_G \mathcal{YD}$ are surjective ${}^G_G \mathcal{YD}$-module homomorphisms. 
 \end{lemma}
 \begin{proof}
 	Let $f \colon M \to N$ be a ${}^G_G \mathcal{YD}$-module homomorphism.
 	
 	Suppose $f(m_1)=f(m_2)$ for some $m_1\ne m_2$, $m_1,m_2\in M$.
 	Let $$g_0 := \deg m_1 = \deg f(m_1) = \deg f(m_2) = \deg m_2.$$
 	
 	Denote by $G_0$ the ${}^G_G \mathcal{YD}$-module that coincides with $G$ as a set endowed with the $G$-action by left shifts
 	and $\deg \colon G_0 \to G$ is defined by $\deg g := g g_0 g^{-1}$ for all $g\in G_0$.
 	Let $f_i \colon G_0 \to M$ be the ${}^G_G \mathcal{YD}$-module homomorphisms 
 	defined by $f_i(g) := gm_i$ for $i=1,2$. Then $ff_1 = ff_2$, however $f_1(e) \ne f_2(e)$ where $e\in G$ is the identity element. Hence $f$ is not a monomorphism.
 	
 	Again, the image $f(M)$ of $f$ is a ${}^G_G \mathcal{YD}$-submodule, i.e. the union of some orbits. Suppose $f(M)\ne N$. Then there exists an orbit $O \subseteq N$ such that $O \cap f(M) = \varnothing$.
 	Let $O_1$ and $O_2$ be two copies of $O$. Denote by $\tau_i\colon O \mathrel{\widetilde{\to}} O_i$ the 
 	corresponding isomorphisms, $i=1,2$.
 	 Let $X := (N \backslash O) \sqcup  O_1 \sqcup O_2$. Define $f_i \colon N \to X$ by $f_i(n)=n$ for all $n\in N \backslash O$ and $f_i(x)=\tau_i(x)$ for all $x\in O$, $i=1,2$. Then $f_1 f = f_2 f$, however $f_1 \ne f_2$. Hence $f$ is not an         
	 epimorphism.
 \end{proof}	

Again, the free monoid $\langle X \rangle$ of a ${}^G_G \mathcal{YD}$-module $X$ inherits the structure of
a ${}^G_G \mathcal{YD}$-module, which makes $\langle X \rangle$ a ${}^G_G \mathcal{YD}$-module algebra.
For an arbitrary ${}^G_G \mathcal{YD}$-module $X$ its cofree comonoid $\mathcal G X$ in ${}^G_G \mathcal{YD}$ is its neutral component $X^{(e)}$.	

The category ${}^G_G \mathcal{YD}$ is braided where the braiding $c_{X,Y} \colon X \times Y \to Y \times X$
is defined by $c_{X,Y}(x,y) := \bigl((\deg x)y, x\bigr)$.
Moreover, ${}^G_G \mathcal{YD}$ is closed monoidal where for $X\ne \varnothing$ the ${}^G_G \mathcal{YD}$-module $[X,Y]$ consists of all maps $f \colon X \to Y$
such that $\deg f(x) = (\deg f)(\deg x)$ for all $x\in X$ and some element $\deg f \in G$ that does not depend on $x$.
The $G$-module structure on $[X,Y]$ is defined by $(gf)(x):= gf(g^{-1} x)$ for $x\in X$. At the same time,
$[\varnothing, Y]=G$, which is the terminal object in ${}^G_G \mathcal{YD}$.

 Therefore we get
\begin{theorem}\label{TheoremYDSetsSatisfiesProperties}
	Let $G$ be a group. Then ${}^G_G \mathcal{YD}$
	is a braided closed monoidal category
	satisfying Properties~\ref{PropertySmallLimits}--\ref{PropertyFreeMonoid}, \ref{PropertyDUALSmallColimits}--\ref{PropertyDUALCofreeComonoid},
	\ref{PropertyExtrMonomorphism} and \ref{PropertyDUALExtrEpimorphism} 
	of Sections~\ref{SubsectionSupportCoactingConditions} and~\ref{SubsectionDUALCosupportActingConditions}.
		Moreover, all monomorphisms and epimorphisms in ${}^G_G \mathcal{YD}$ are extremal.
\end{theorem}

Hopf monoids in ${}^G_G \mathcal{YD}$ are trivially graded groups $H$ endowed with a homomorphism $\varphi \colon G \to \Aut(H)$.

An $H$-module $\Omega$-magma (recall that the general definition in an arbitrary category was given in Section~\ref{SubsectionMeasurings}) in ${}^G_G \mathcal{YD}$ is just an $\Omega$-magma
$A$ that is an ${}^G_G \mathcal{YD}$-module where the $G$-action is extended to an $H \mathrel{\leftthreetimes_\varphi} G$-action by automorphisms, all operations in $A$ are graded and $h A^{(g)}= A^{(g)}$ for all $g\in G$ and $h\in H$.

An $H$-comodule $\Omega$-magma in ${}^G_G \mathcal{YD}$ is an $G\times H$-graded $\Omega$-magma $A=\bigsqcup\limits_{\substack{g\in G,\\ h\in H}} A^{(g,h)}$ endowed with a $G$-action by automorphisms such that $gA^{(t,h)} = A^{(gtg^{-1}, \varphi(g)h)}$ for all $g,t\in G$ and $h\in H$.

If one considered the inverse braiding on ${}^G_G \mathcal{YD}$, then the operations on $A$ and the $H$-(co)action would be related in a more complicated way.

\subsection{Limits and colimits in $\mathsf{Comon}(\mathcal C)$ and coreflectivity of $\mathsf{Hopf}(\mathcal C)$ in $\mathsf{Bimon}(\mathcal C)$}\label{SubsectionDUAL(Co)limitsComonHopfMonCoreflectivity}

Let $T \colon J \to \mathsf{Comon}(\mathcal C)$ be a functor where $J$ is a category and $\mathcal C$ is a monoidal category.
Suppose that $C := \colim UT$ is the colimit of $UT$ in $\mathcal C$ where $U\colon \mathsf{Comon}(\mathcal C) \to \mathcal C$
is the forgetful functor. Then the unique morphisms $\Delta_C$
and $\varepsilon_C$ making the diagrams below commutative for every object $j$ of $J$ ($\varphi_j$ is the colimiting cocone) turn $C$ into a comonoid and, therefore, into the colimit of $T$ in $\mathsf{Comon}(\mathcal C)$:

$$\xymatrix{
	 Tj \ar[d]^{\Delta_{Tj}} \ar[rr]^{\varphi_j} & & C   \ar@{-->}[d]^{\Delta_C} \\
Tj \otimes Tj  \ar[rr]^{\varphi_j \otimes \varphi_j} & & 	C\otimes C   } \qquad 
\xymatrix{Tj \ar[rd]_{\varepsilon_{Tj}} \ar[rr]^{\varphi_j}  & &  C \ar@{-->}[ld]^{\varepsilon_C} \\
	& \mathbbm{1} & }
$$

In other words, the forgetful functor $U$ creates colimits.

Now consider limits in $\mathsf{Comon}(\mathcal C)$:

\begin{theorem}\label{TheoremDUALLimitsComon} Let $\mathcal C$ be a braided monoidal category satisfying Properties~\ref{PropertyDUALSmallColimits},
	\ref{PropertyDUALExtrEpiMonoFactorizations}, \ref{PropertyDUALSubObjectsSmallSet}, \ref{PropertyDUALMonomorphism} and \ref{PropertyDUALCofreeComonoid} of Section~\ref{SubsectionDUALCosupportActingConditions}.
	Let $T \colon J \to \mathsf{Comon}(\mathcal C)$ be a functor where $J$ is a category.
	Suppose there exists $N := \lim UT$ and $\varphi_j$ is the corresponding limiting cone. (The limit is taken in $\mathcal C$.)
	Then there exists a comonoid homomorphism $i \colon P \rightarrowtail {\mathcal G} N$, which is, in addition,
	a monomorphism in $\mathcal C$, such that the composition 
	$$\xymatrix{ P\ \ar@{>->}[r]^{i} & {\mathcal G}N \ar[r]^{\gamma_N} & N \ar[r]^{\varphi_j} &  Tj }$$
	is a limiting cone of $T$ in $\mathsf{Comon}(\mathcal C)$.
	(Here $\gamma$ is the counit of the adjunction $U \dashv {\mathcal G}$.)
\end{theorem}
\begin{proof}
	Theorem~\ref{TheoremDUALLimitsComon} is dual to Theorem~\ref{TheoremColimitsMon}.
\end{proof}

\begin{theorem}\label{TheoremDUALBimonHopfRightAdjoint}
	Let $\mathcal C$ be a braided monoidal category satisfying Properties~\ref{PropertyDUALSmallColimits}--\ref{PropertyDUALExtrEpiMonoFactorizations}, \ref{PropertyDUALSubObjectsSmallSet}, \ref{PropertyDUALMonomorphism} and \ref{PropertyDUALCofreeComonoid} of Section~\ref{SubsectionDUALCosupportActingConditions}.
	Then the forgetful functor $\mathbf{Hopf}(\mathcal C) \to \mathbf{Bimon}(\mathcal C)$
	admits a right adjoint functor $H_r \colon \mathbf{Bimon}(\mathcal C) \to \mathbf{Hopf}(\mathcal C)$.
\end{theorem}
\begin{proof}
	Theorem~\ref{TheoremDUALBimonHopfRightAdjoint} is dual to Theorem~\ref{TheoremBimonHopfLeftAdjoint}.
\end{proof}	

\begin{remark} Again, in Theorems~\ref{TheoremDUALLimitsComon} and~\ref{TheoremDUALBimonHopfRightAdjoint} instead of Property~\ref{PropertyDUALMonomorphism} it is sufficient to require just that for every monomorphism $\varphi$ in~$\mathcal C$ the morphisms $\varphi\otimes \varphi$ and $\varphi\otimes \varphi \otimes \varphi$
	are monomorphisms too.
\end{remark}

\begin{example}\label{ExampleDUALHopfRightAdjointSets} In the case $\mathcal{C} = \mathbf{Sets}$ the right adjoint functor $H_r$ assigns to each monoid $M$ its group $\mathcal U(M)$ of invertible elements.
\end{example}

\subsection{Cosupports of morphisms $P \otimes A \to B$}

Here we define cosupports of morphisms and prove their existence.

Let $\mathcal C$ be a monoidal category.
For given objects $A,B$ in $\mathcal C$ denote by $\mathbf{TensMor}(A,B)$
the comma category $( (-)\otimes A \downarrow B)$, i.e.
the category where
\begin{itemize}
	\item the objects are all morphisms $\psi \colon  P \otimes A \to B$ for arbitrary objects $P$;
	\item the morphisms between $\psi_1 \colon P_1 \otimes A \to B$ and $\psi_2 \colon P_2 \otimes  A \to B$ are morphisms $\tau \colon P_1 \to P_2$
	making the diagram below commutative:
	$$ \xymatrix{	P_1 \otimes A \ar[d]_{\tau\, \otimes\, \id_{A}} \ar[r]^(0.6){\psi_1}  &  B   \\
		P_2 \otimes A \ar[ru]_{\psi_2} &  }$$
\end{itemize}

For morphisms $\psi \colon P \otimes A \to B$ we are going to use the terminology and the notation from Section~\ref{SectionLiftingProblem} with respect to $X=\mathbf{TensMor}(A,B)^\mathrm{op}$. 

If there exists $|\psi| \colon  \tilde P \otimes A \to B$ for some $\psi$, then we call the object $\cosupp \psi := \tilde P$
the \textit{cosupport} of $\psi$. From the definition of the absolute value it follows that $\cosupp \psi$ is defined up to an isomorphism compatible with $|\rho|$.

\begin{definition}
	We say that a morphism $\psi \colon P \otimes A \to B$ is a \textit{tensor monomorphism}
	if $\rho \in \LIO(\mathbf{TensMor}(A,B)^\mathrm{op})$, i.e. if for every $f,g \colon R \to P$ such that
  \begin{equation*}
		\psi (f \otimes {\id_A}) = \psi (g \otimes {\id_A}) \end{equation*}
	we have $f=g$.
\end{definition}

\begin{example}
	In the case $\mathcal C = \mathbf{Vect}_\mathbbm{k}$, where $\mathbbm{k}$ is a field,
	$\psi$ is a tensor monomorphism if and only if the corresponding linear map $P \to \Hom_\mathbbm{k}(A,B)$
	is injective. In particular,	
	 $$\cosupp \psi = \psi(P\otimes(-))\subseteq \Hom_\mathbbm{k}(A,B).$$ Therefore $\cosupp \psi$ coincides with the cosupport defined in~\cite{AGV2}.
\end{example}

Unfolding the definition from Section~\ref{SubsectionLIOAbsValue} for $X=\mathbf{TensMor}(A,B)^\mathrm{op}$, we see that $\psi_1 \preccurlyeq \psi_2$ for morphisms $\psi_1 \colon P_1 \otimes A \to B$
and $\psi_2 \colon P_2 \otimes A \to B$
 if and only if there exists a morphism $\tau \colon \cosupp \psi_1 \to \cosupp \psi_2$
such that $|\psi_1| = |\psi_2|(\tau\otimes {\id_A}) $:
\begin{equation*}\xymatrix{
		{(\cosupp \psi_1)} \otimes A \ar[r]^(0.7){|\psi_1|}  \ar[d]_{\tau\otimes{\id_A}} & B  \\
		{(\cosupp \psi_2)} \otimes A \ar[ru]_(0.6){|\psi_2|} &
}\end{equation*}

\begin{theorem}\label{TheoremDUALAbsValueCosupportExistence}
	Let $\mathcal C$ be a monoidal category satisfying Properties~\ref{PropertyDUALSmallColimits}, \ref{PropertyDUALFactorObjectsSmallSet}--\ref{PropertyDUALColimitsOfFactorObjectsArePreserved} and \ref{PropertyDUALCoequalizers}. Then for every objects $A,B$ in $\mathcal C$
	there exist absolute values of
	all objects in the category $\mathbf{TensMor}(A,B)^\mathrm{op}$.
	As a consequence, there exist cosupports for all morphisms 
	 $\psi \colon P \otimes A \to B$ in $\mathcal C$. 
\end{theorem}
\begin{proof} Theorem~\ref{TheoremDUALAbsValueCosupportExistence} is dual to Theorem~\ref{TheoremAbsValueSupportExistence}.
\end{proof}	

\begin{remark} We will see in \cite[Remarks 4.3 (1)]{AGV3} that if $\mathcal C$ is monoidal closed,
	then it is sufficient to require just Property~\ref{PropertyDUALExtrEpiMonoFactorizations}
	for cosupports and absolute values to exist.
\end{remark}
	
\subsection{Cosupports of module structures}\label{SubsectionDUALCosupportModule}

Again, it turns out that under certain conditions on the base category the cosupport of a module structure
is always a submonoid.

\begin{theorem}\label{TheoremDUALCosupportModule}
	Suppose $\mathcal C$ is a monoidal category
	satisfying Properties~\ref{PropertyDUALSmallColimits},  \ref{PropertyDUALFactorObjectsSmallSet}--\ref{PropertyDUALTensorPushout}, \ref{PropertyDUALCoequalizers}
	and~\ref{PropertyDUALExtrEpimorphism} of Section~\ref{SubsectionDUALCosupportActingConditions}.
	Let $M$ be a module over a monoid $(A,\mu,u)$ 
	and let $\psi \colon A\otimes M \to M$ be the corresponding morphism.
	Then in $\mathcal C$ there exist unique morphisms $$\mu_0 \colon
	\cosupp \psi \otimes \cosupp \psi \to \cosupp \psi \text{\quad and \quad}u_0 \colon \mathbbm{1} \to \cosupp \psi $$ making the diagrams below commutative:
	\begin{equation*}\xymatrix{
			A \otimes A	\ar[d]^{\mu} \ar@{->>}[r] & 	(\cosupp \psi) \otimes (\cosupp \psi) \ar@{-->}[d]^{\mu_0}   \\  
			A  \ar@{->>}[r] & 	 \cosupp \psi    \\
	}\end{equation*}
	\begin{equation*}
		\xymatrix{
			\mathbbm{1} \ar[d]^u \ar@{-->}[rd]^(0.6){u_0} &       \\
			A \ar@{->>}[r] & \cosupp \psi   \\
		}
	\end{equation*}
	
	Moreover, $(\cosupp \psi,\mu_0,u_0)$ is a monoid and the diagrams below are commutative too:
	\begin{equation*} \xymatrix{
			{(\cosupp \psi)} \otimes {(\cosupp \psi)} \otimes M    \ar[rr]^(0.6){ \mu_0 \otimes {\id_M}}	\ar[d]_{{\id_{\cosupp \psi}} \otimes |\psi|}  	& &	  {(\cosupp \psi)} \otimes M  \ar[d]_{|\psi|} \\	
			{(\cosupp \psi)} \otimes M  \ar[rr]^(0.6){|\psi|}                       & &  M			
		}
	\end{equation*}
	\begin{equation*}
		\xymatrix{ M \ar@{=}[d] \ar[r]_(0.4){\sim} & \mathbbm{1} \otimes M  \ar[d]^{u_0 \otimes {\id_M}} \\
			M & \ar[l]^(0.7){|\psi|}{(\cosupp \psi)} \otimes M} \end{equation*}
	In other words, $M$ is a  $(\cosupp \psi)$-module and the epimorphism $\cosupp \psi \twoheadrightarrow A$ is a monoid homomorphism.
\end{theorem}

\begin{proof} Theorem~\ref{TheoremDUALCosupportModule} is dual to Theorem~\ref{TheoremSupportComodule}. 
\end{proof}	

Again, the morphism $\tau$ from the definition of the coarser/finer relation is always a monoid homomorphism
in the case of module structures.

\begin{proposition}\label{PropositionDUALFromTensorMonoMonoidMorphism}
	Let $\mathcal C$ be a monoidal category and let $\psi_i \colon  {A_i} \otimes M \to M $ define on an object~$M$ structures of $A_i$-modules
	for monoids $(A_i, \mu_i, u_i)$ for $i=1,2$.
	Suppose that $\psi_2$ is a tensor monomorphism and the diagram below is commutative for some morphism $\tau \colon A_1 \to A_2$:
	\begin{equation*}\xymatrix{
			A_1	\otimes M    \ar[d]_{ \tau \otimes {\id_M}} \ar[r]^(0.65){\psi_{1}}  &     M\\
			{A_2} \otimes M \ar[ru]_(0.6){\psi_2} &  \\
	}\end{equation*}
	Then $\tau$ is a monoid homomorphism.
\end{proposition}
\begin{proof} Proposition~\ref{PropositionDUALFromTensorMonoMonoidMorphism} is dual to Proposition~\ref{PropositionFromTensorEpiComonoidMorphism}.
\end{proof}	

\subsection{Universal measuring comonoids}\label{SubsectionDUALUnivMeasExistence}

Fix $\Omega$-magmas $A$ and $B$ in a braided monoidal category $\mathcal C$. Consider the category $\mathbf{Meas}(A,B)$ where
\begin{itemize}
	\item the objects are all measurings $\psi \colon P\otimes A \to B$ for arbitrary comonoids $P$;
	\item the morphisms from $\psi_1 \colon P_1\otimes A \to B$
	to $\psi_2 \colon P_2\otimes A \to B$  are monoid homomorphisms $\varphi \colon P_1 \to P_2$
	making the diagram below commutative:
\begin{equation*}\xymatrix{
		{P_1} \otimes A \ar[r]^(0.6){\psi_1}  \ar[d]_{\varphi\otimes{\id_A}} & B  \\
		{P_2} \otimes A \ar[ru]_(0.6){\psi_2} &
}\end{equation*}	
\end{itemize}

Denote by $G_1'$ the forgetful functor $\mathbf{Meas}(A,B)\to \mathbf{TensMor}(A,B)$. 
Let $\psi_V \colon V \otimes A \to B$ be a tensor monomorphism  for some object $V$.
We call the comonoid ${}_\square\mathcal{C}(\psi_V)$ corresponding to the initial object $$\psi_V^\mathbf{Meas} \colon {}_\square\mathcal{C}(\psi_V) \otimes A \to B$$ in $\mathbf{Meas}(A,B)^{\mathrm{op}}_{G_1'}(\psi_V)$ (if it exists) the \textit{$V$-universal measuring comonoid} from $A$ to $B$.

In other words, ${}_\square\mathcal{C}(\psi_V)$ is a $V$-universal measuring comonoid if  for every measuring $\psi \colon P \otimes A \to B$
such that $\psi =  \psi_V (\tau \otimes \id_A)$ for some morphism $\tau \colon P \to V$ there exists a unique comonoid homomorphism
$\varphi \colon P \to {}_\square\mathcal{C}(\psi_V)$ making the diagram below commutative:

\begin{equation*}\xymatrix{
		{P} \otimes A \ar[r]^(0.6){\psi}  \ar[d]_{\varphi\otimes{\id_A}} & B  \\
		{{}_\square\mathcal{C}(\psi_V)} \otimes A \ar[ru]_(0.6){\psi_V^\mathbf{Meas}} &
}\end{equation*}

\begin{theorem}\label{TheoremDUALComonUnivMeasExistence}  Suppose that a braided monoidal category $\mathcal C$
	satisfies  Properties~\ref{PropertyDUALSmallColimits}, \ref{PropertyDUALExtrEpiMonoFactorizations}, \ref{PropertyDUALSubObjectsSmallSet}, \ref{PropertyDUALMonomorphism},  \ref{PropertyDUALEpimorphism},  \ref{PropertyDUALSwitchCoprodTensorIsAnEpimorphism}, \ref{PropertyDUALCofreeComonoid} of Section~\ref{SubsectionDUALCosupportActingConditions}.
	Then there exists an initial object in $\mathbf{Meas}(A,B)^{\mathrm{op}}_{G_1'}(\psi_V)$ if $\mathbf{Meas}(A,B)^{\mathrm{op}}_{G_1'}(\psi_V)$ is not empty.
\end{theorem}
\begin{proof} Theorem~\ref{TheoremDUALComonUnivMeasExistence} is dual to Theorem~\ref{TheoremMonUnivComeasExistence}.
\end{proof}	

\begin{examples}\label{SweedlerUnivMeasuring}
\hspace{1cm}
\begin{enumerate}
\item In the case $\mathcal C = \mathbf{Vect}_\mathbbm{k}$ for a field $\mathbbm{k}$ we get the $V$-universal measuring coalgebra ${}_\square\mathcal{C}(\psi_V)$ (see~\cite[Theorem 3.10]{AGV2}). If we take $V = \Hom_\mathbbm{k}(A,B)$ and define $\psi_V \colon \Hom_\mathbbm{k}(A,B) \otimes A \to B$ by $\psi_V(\theta \otimes a) := \theta(a)$,
	then $\psi \preccurlyeq \psi_V$ for every measuring $\psi \colon P \otimes A \to B$.
	In this case ${}_\square\mathcal{C}(\psi_V)$ is the Sweedler universal measuring coalgebra (see \cite[Chapter
	VII]{Sweedler});
\item\label{ExampleDUALUnivMeasSets} In the case $\mathcal C = \mathbf{Sets}$
	a map $\psi_V \colon V \times A \to B$ is a tensor monomorphism if and only if $\psi_V(v_1,-)\ne \psi_V(v_2,-)$
	for $v_1 \ne v_2$. Therefore, $V$ may be identified with a subset in the set $\mathbf{Sets}(A,B)$
	of all maps $A\to B$. We have $\psi \preccurlyeq \psi_V$ for a given $\psi \colon P \times A \to B $ if and only if $\psi(P,-)\subseteq V$. Recall that $\psi$ is a measuring if and only if $\psi(p,-)$ is an $\Omega$-magma homomorphism
	for every $p\in P$. Hence in this case  ${}_\square\mathcal{C}(\psi_V) = V \cap \Hom_{\Omega}(A,B)$ where $\Hom_{\Omega}(A,B)$
	is the set of all $\Omega$-magma homomorphisms $A\to B$.
\end{enumerate}
\end{examples}

\begin{remark} \label{RemarkDUALInitialMeas}
	Suppose that there exists an initial object $I$ in $\mathcal{C}$ and the object $I \otimes A$ is again initial for every object $A$. (This holds, say, for $\mathcal C = \mathbf{Vect}_\mathbbm{k}$ and $\mathcal C = \mathbf{Sets}$.) Note that the unique morphisms $I \to \mathbbm{1}$ and $I \to I \otimes I$ define on $I$ the structure of a comonoid. (All the corresponding diagrams are commutative, since $I$ is an initial object.) Denote by $\psi_I$  the unique measuring $I \otimes A \to B$. Then $\psi_I$, being the initial object in $\mathbf{Meas}(A,B)$,
	is the terminal object in $\mathbf{Meas}(A,B)^{\mathrm{op}}_{G_1'}(\psi_V)$. 
\end{remark}	

\subsection{Universal acting bimonoids}\label{SubsectionDUALUnivBiActingExistence}

Again fix an $\Omega$-magma $A$ in a braided monoidal category $\mathcal C$.
Define the category $\mathbf{Act}(A)$ where \begin{itemize}
	\item the objects are all actions $\psi \colon B\otimes A \to A$ for arbitrary bimonoids $B$;
	\item the morphisms from $\psi_1 \colon B_1\otimes A \to A$
	to $\psi_2 \colon B_2\otimes A \to A$  are bimonoid homomorphisms $\varphi \colon B_1 \to B_2$
	making the diagram below commutative:
	\begin{equation*}\xymatrix{
			{B_1} \otimes A \ar[r]^(0.6){\psi_1}  \ar[d]_{\varphi\otimes{\id_A}} & A  \\
			{B_2} \otimes A \ar[ru]_(0.6){\psi_2} &
	}\end{equation*}	
\end{itemize}

Consider the following commutative diagram consisting of forgetful functors:
$$\xymatrix{
	\mathbf{Act}(A)^{\mathrm{op}} \ar[d]^{G'_2} \ar[r]^-{G'_3} & \mathbf{Meas}(A,A)^{\mathrm{op}} \ar[d]^{G'_1}   \\
	\mathbf{ModStr}(A)^{\mathrm{op}}  
	\ar[r]^-{G'} & \mathbf{TensMor}(A,A)^{\mathrm{op}}
}$$	

Let $V$ be a monoid and let a tensor monomorphism $\psi_V \colon V \otimes A \to A$
define on $A$ a structure of a $V$-module.
Let us call the bimonoid 
corresponding to the initial object in $\mathbf{Act}(A)^{\mathrm{op}}_{G_2'}(\psi_V)$ (if it exists) the \textit{$V$-universal acting bimonoid} on $A$.

\begin{remark}\label{RemarkDUALActAG2PsiVEquivDef}
	Again, by Proposition~\ref{PropositionDUALFromTensorMonoMonoidMorphism}, 
	the category $\mathbf{Act}(A)^{\mathrm{op}}_{G_2'}(\psi_V)$  coincides with the category
	$\mathbf{Act}(A)^{\mathrm{op}}_{G'G_2'}(G'\psi_V) = \mathbf{Act}(A)^{\mathrm{op}}_{G_1'G_3'}(G' \psi_V)$.
\end{remark}

\begin{theorem}\label{TheoremDUALBimonUnivActingExistence} Let $\mathcal C$ be a braided monoidal category 
	and let $\psi_V \colon V \otimes A \to A$ be a tensor monomorphism defining on an $\Omega$-magma~$A$
	a structure of a $V$-module for a monoid $V$ such that there exists ${}_\square\mathcal{C}(\psi_V)$.
	 Then ${}_\square\mathcal{B}(\psi_V):={}_\square\mathcal{C}(\psi_V)$ admits a unique monoid structure turning
	$\psi_V^\mathbf{Act}:=\psi_V^\mathbf{Meas}$ into an action, which is the initial object in $\mathbf{Act}(A)^{\mathrm{op}}_{G_2'}(\psi_V)$.
\end{theorem}
\begin{proof}Theorem~\ref{TheoremDUALBimonUnivActingExistence} is dual to Theorem~\ref{TheoremBimonUnivCoactingExistence}.
\end{proof}	
\begin{corollary}\label{CorollaryDUALBimonUnivActingExistence} Suppose that a braided monoidal category $\mathcal C$
	satisfies  Properties~\ref{PropertyDUALSmallColimits}, \ref{PropertyDUALExtrEpiMonoFactorizations}, \ref{PropertyDUALSubObjectsSmallSet}, \ref{PropertyDUALMonomorphism},  \ref{PropertyDUALEpimorphism},  \ref{PropertyDUALSwitchCoprodTensorIsAnEpimorphism}, \ref{PropertyDUALCofreeComonoid} of Section~\ref{SubsectionDUALCosupportActingConditions}.
	Let $\psi_V \colon V \otimes A \to A$ be a tensor monomorphism defining on an $\Omega$-magma~$A$
	a structure of a $V$-module for a monoid $V$.
	Then ${}_\square\mathcal{B}(\psi_V):={}_\square\mathcal{C}(\psi_V)$ admits a unique monoid structure turning
	$\psi_V^\mathbf{Act}:=\psi_V^\mathbf{Meas}$ into an action, which is the initial object in $\mathbf{Act}(A)^{\mathrm{op}}_{G_2'}(\psi_V)$.
\end{corollary}
\begin{proof} Apply Theorems~\ref{TheoremDUALComonUnivMeasExistence} and~\ref{TheoremDUALBimonUnivActingExistence}.
\end{proof}	

In other words, for every action $\psi \colon B \otimes A \to A$
such that $\psi =  \psi_V (\tau \otimes \id_A)$ for some morphism $\tau \colon B \to V$ (recall that by Proposition~\ref{PropositionDUALFromTensorMonoMonoidMorphism} the morphism $\tau$ is necessarily a monoid homomorphism) there exists a unique bimonoid homomorphism
$\varphi \colon B \to {}_\square\mathcal{B}(\psi_V)$ making the diagram below commutative:

\begin{equation*}\xymatrix{
		{B} \otimes A \ar[r]^(0.6){\psi}  \ar[d]_{\varphi\otimes{\id_A}} & A  \\
		{{}_\square\mathcal{B}(\psi_V)} \otimes A \ar[ru]_(0.6){\psi_V^\mathbf{Act}} &
}\end{equation*}

\begin{remark} Note that $\mathbbm{1} \otimes  A \mathrel{\widetilde\to} A $ is a terminal object in $\mathbf{Act}(A)^{\mathrm{op}}_{G_2'}(\psi_V)$,
	since for every bimonoid $B$ there exists the only bimonoid homomorphism $\mathbbm{1} \to B$, namely, the unit $u$.
\end{remark}

\begin{examples}\label{SweedlerBialgebra}
\hspace{1cm}
\begin{enumerate}
\item\label{ExampleSweedlerBialgebra}
When $\mathcal C = \mathbf{Vect}_\mathbbm{k}$ and $A$ is a unital associative algebra, the bimonoid ${}_\square\mathcal{B}(A,\End_\mathbbm{k}(A))$
is exactly the Sweedler universal acting bialgebra on $A$~\cite[Chapter VII]{Sweedler};	
\item\label{ExampleDUALUnivActionSets} 
	Example~\ref{SweedlerUnivMeasuring} (\ref{ExampleDUALUnivMeasSets}) shows that in the case $\mathcal{C} = \mathbf{Sets}$
	the monoid $V$ can be identified with the corresponding submonoid in $\mathbf{Sets}(A,A)$, the ordinary monoid of all 
	 maps $A \to A$. Hence ${}_\square\mathcal{B}(\psi_V) = V \cap \End_{\Omega}(A)$ where $\End_{\Omega}(A)$
	 is the ordinary monoid of all $\Omega$-magma endomorphisms of $A$.
\end{enumerate}
\end{examples}

\subsection{Universal acting Hopf monoids}\label{SubsectionDUALUnivHopfActingExistence}

Again fix an $\Omega$-magma $A$ in a braided monoidal category~$\mathcal C$. Consider the category $\mathbf{HAct}(A)$ where \begin{itemize}
	\item the objects are all actions $\psi \colon H\otimes A \to A$ for arbitrary Hopf monoids $H$;
	\item the morphisms from $\psi_1 \colon H_1\otimes A \to A$
	to $\psi_2 \colon H_2\otimes A \to A$  are Hopf monoid homomorphisms $\varphi \colon H_1 \to H_2$
	making the diagram below commutative:
	\begin{equation*}\xymatrix{
			{H_1} \otimes A \ar[r]^(0.6){\psi_1}  \ar[d]_{\varphi\otimes{\id_A}} & A  \\
			{H_2} \otimes A \ar[ru]_(0.6){\psi_2} &
	}\end{equation*}	
\end{itemize}

Let $V$ be a monoid and let a tensor monomorphism $\psi_V \colon V \otimes A \to A$
define on $A$ a structure of a $V$-module. Denote by $G_4'$ the forgetful functor $\mathbf{HAct}(A)\to \mathbf{Act}(A)$.
Let us call the Hopf monoid ${}_\square\mathcal{H}(\psi_V)$
corresponding to the initial object $\psi_V^\mathbf{HAct}$
in $$\mathbf{HAct}(A)^{\mathrm{op}}_{G_2'G_4'}(\psi_V)$$ (if it exists) the \textit{$V$-universal acting Hopf monoid} on $A$.

\begin{theorem}\label{TheoremDUALHopfMonUnivActingExistence}
	Let $\mathcal C$ be a braided monoidal category and let $\psi_V \colon V \otimes A \to A$
	be a tensor monomorphism
	defining on an $\Omega$-magma~$A$ a structure of a $V$-module for a monoid $V$.
	Suppose that the forgetful functor $\mathbf{Hopf}(\mathcal C) \to \mathbf{Bimon}(\mathcal C)$
	admits a right adjoint functor $H_r \colon \mathbf{Bimon}(\mathcal C) \to \mathbf{Hopf}(\mathcal C)$
	and there exists ${}_\square\mathcal{B}(\psi_V)$.
	Then the initial object $\psi_V^\mathbf{HAct}$
	in $$\mathbf{HAct}(A)^{\mathrm{op}}_{G_2'G_4'}(\psi_V)$$ indeed exists.
\end{theorem}
\begin{proof}
Theorem~\ref{TheoremDUALHopfMonUnivActingExistence} is dual to Theorem~\ref{TheoremHopfMonUnivCoactingExistence}.
\end{proof}
\begin{corollary}\label{CorollaryDUALHopfMonUnivActingExistence}  Suppose that a braided monoidal category $\mathcal C$
	satisfies  Properties~\ref{PropertyDUALSmallColimits}--\ref{PropertyDUALExtrEpiMonoFactorizations},
	\ref{PropertyDUALSubObjectsSmallSet}, \ref{PropertyDUALMonomorphism}, \ref{PropertyDUALEpimorphism},  \ref{PropertyDUALSwitchCoprodTensorIsAnEpimorphism}, \ref{PropertyDUALCofreeComonoid} of Section~\ref{SubsectionDUALCosupportActingConditions}. Let $\psi_V \colon V \otimes A \to A$
	be a tensor monomorphism
	defining on an $\Omega$-magma~$A$ a structure of a $V$-module for a monoid $V$ in~$\mathcal C$.
	Then the initial object $\psi_V^\mathbf{HAct}$
	in $$\mathbf{HAct}(A)^{\mathrm{op}}_{G_2'G_4'}(\psi_V)$$ indeed exists.
\end{corollary}
\begin{proof}
	Apply Theorems~\ref{TheoremDUALBimonHopfRightAdjoint}, \ref{TheoremDUALHopfMonUnivActingExistence} and Corollary~\ref{CorollaryDUALBimonUnivActingExistence}.
\end{proof}		

Again, if we unfold the definition, we obtain that for every action $\psi \colon H \otimes A \to A$
such that $\psi =  \psi_V (\tau \otimes \id_A)$ for some morphism $\tau \colon H \to V$ (recall that by Proposition~\ref{PropositionDUALFromTensorMonoMonoidMorphism} the morphism $\tau$ is necessarily a monoid homomorphism) there exists a unique Hopf monoid homomorphism
$\varphi \colon H \to {}_\square\mathcal{H}(\psi_V)$ making the diagram below commutative:

\begin{equation*}\xymatrix{
		{H} \otimes A \ar[r]^(0.6){\psi}  \ar[d]_{\varphi\otimes{\id_A}} & A  \\
		{{}_\square\mathcal{H}(\psi_V)} \otimes A \ar[ru]_(0.6){\psi_V^\mathbf{HAct}} &
}\end{equation*}

\begin{remark} Again, $\mathbbm{1} \otimes  A \mathrel{\widetilde\to} A $ is a terminal object
	in $$\mathbf{HAct}(A)^{\mathrm{op}}_{G_2'G_4'}(\psi_V),$$
	since for every Hopf monoid $H$ there exists the only Hopf monoid homomorphism $\mathbbm{1} \to H$, namely, the unit~$u$.
\end{remark}

\begin{examples}\label{UniversalActingHopf}
\hspace{1cm}
\begin{enumerate}
\item When $\mathcal C = \mathbf{Vect}_\mathbbm{k}$ and $A$ is an $\Omega$-algebra, the Hopf monoid ${}_\square\mathcal{H}(A,\End_\mathbbm{k}(A))$
is the universal acting Hopf algebra on $A$.	
\item		Let $\psi \colon  H \otimes A \to A$ be some Hopf monoid action on an $\Omega$-magma $A$
with an absolute value $|\psi|_{G_2'G_4'} \colon  (\mathop\mathrm{cosupp} \psi) \otimes A \to A$ in a braided monoidal category
$\mathcal C$ satisfying Properties~\ref{PropertyDUALSmallColimits}--\ref{PropertyDUALCofreeComonoid}, \ref{PropertyDUALSubObjectsSmallSet}, \ref{PropertyDUALMonomorphism} and~\ref{PropertyDUALExtrEpimorphism} of Section~\ref{SubsectionDUALCosupportActingConditions}.
Then by Propositions~\ref{PropositionLIOCoarserFinerEquivalent} and~\ref{PropositionLIOAbsValueLiftCriterion}
the ${}_\square\mathcal{H}(|\psi|_{G_2'G_4'})$-action $|\psi|^{\mathbf{HAct}}_{G_2'G_4'}$ on $A$ is equivalent to~$\psi$ and is universal among all actions equivalent to or coarser than $\psi$.
We call ${}_\square\mathcal{H}(|\psi|_{G_2'G_4'})$
\textit{the universal Hopf monoid} of $\psi$ and in the case $\mathcal C = \mathbf{Vect}_\mathbbm{k}$ for a field $\mathbbm{k}$
this Hopf monoid is exactly the universal Hopf algebra of $\psi$ introduced in~\cite{AGV1}.
\item\label{ExampleDUALUnivHopfActionSets} 
	Recall that by Example~\ref{SweedlerUnivMeasuring} (\ref{ExampleDUALUnivMeasSets}) in the case $\mathcal{C} = \mathbf{Sets}$
	the monoid $V$ can be identified with the corresponding submonoid in $\mathbf{Sets}(A,A)$. 
	Hence \begin{equation}\label{EqDUALUnivHopfActionSets} 
	{}_\square\mathcal{H}(\psi_V) = {\mathcal U}(V) \cap \Aut_{\Omega}(A)
	\end{equation} where ${\mathcal U}(V)$
	is the group of invertible elements of $V$ and $\Aut_{\Omega}(A)$
	is the group of all $\Omega$-magma automorphisms of $A$. 
	 Taking $V=\mathbf{Sets}(A,A)$,
	we obtain that $\Aut_{\Omega}(A)$-action is universal among all Hopf monoid (i.e. group) actions on~$A$.
\end{enumerate}
\end{examples}

\begin{remarks}
\hspace{0.1cm}
\begin{enumerate}
\item Note that~\eqref{EqDUALUnivHopfActionSets} resembles the group-like part of the formula for the universal acting
cocommutative Hopf algebra in~\cite[Theorem~5.3]{AGV1} (recall that in $\mathbf{Sets}$ all comonoids are trivial and therefore cocommutative).
\item The existence of universal measuring coalgebras for any pair of $\Omega$-magmas in $\Cc$ from a given class $\Xx$, turns $\Xx$ into a category over comonoids in $\Cc$ (see \cite{HylLopVas} and the forthcoming \cite{GV2}); in other words $\Xx$ can be given the structure of a semi-Hopf category in the sense of \cite{BCV}. If $\Xx$ contains just one object, we recover the universal acting bimonoid as described in Section~\ref{SubsectionDUALUnivBiActingExistence}. As we have just shown, the cofree Hopf algebra over this universal acting bimonoid is a universal acting Hopf monoid on the same $\Omega$-magma. This has been extended in \cite{GV}, showing that there exist (free and) cofree Hopf categories over a semi-Hopf category, which shows the existence of universal Hopf measurings. 
\end{enumerate}
\end{remarks}

\section*{Acknowledgments}

We are grateful to the referee for a thorough review and his/her extremely valuable insights and suggestions which greatly improved the exposition and the overall quality of the manuscript.


\begin{thebibliography}{99}

\bibitem{AbdulIovanov} Abdulwahid, A.\,H., Iovanov, M.\,C. Generators for comonoids and universal constructions. \textit{Arch. Math.} \textbf{106}  (2016), 21--33.


\bibitem{AHS} Ad\'amek, J., Herrlich, H., Strecker, G.\,E. Abstract and concrete categories. John Wiley, New York, 1990. Available online at \url{http://katmat.math.uni-bremen.de/acc/}.

\bibitem{AGV1} Agore, A.\,L., Gordienko, A.\,S., Vercruysse, J. On equivalences of (co)module algebra structures over Hopf algebras, \textit{J. Noncommut. Geometry} \textbf{15} (2021), 951--993. 

\bibitem{AGV2} Agore, A.\,L., Gordienko, A.\,S., Vercruysse, J.
$V$-universal Hopf algebras (co)acting on $\Omega$-algebras. \textit{Commun. Contemp. Math.},
\textbf{25}:1 (2023), 2150095-1 - 2150095-40. 

\bibitem{AGV3} Agore, A.\,L., Gordienko, A.\,S., Vercruysse, J. Dualities for universal co(acting) Hopf monoids, preprint 2024, arXiv:2406.17684v1.

\bibitem{ArdKaoMen} Ardizzoni, A., El Kaoutit, L., Menini, C. Categories of comodules and chain complexes of modules.
\textit{Int. J. Math.} \textbf{23} (2012), Article 1250109.

\bibitem{BK} Banerjee, A., Kour, S. Measurings of Hopf algebroids and morphisms in cyclic (co)homology theories, \textit{Adv. Math.} \textbf{442} (2024), 109581.

\bibitem{BCV}
E.\ Batista, S.\ Caenepeel and J.\ Vercruysse, Hopf categories, {\em Algebr. Represent. Theory} {\bf 19} (2016), 1173--1216.


\bibitem{BKS}
Bekaert, X., Kowalzig, N., Saracco, P. Universal enveloping algebras of Lie-Rinehart algebras: crossed products, connections, and curvature, \textit{Lett. Math. Phys.} \textbf{114} (2024), Article 140.

\bibitem{BhaChiGos}
Bhattacharjee, S., Chirv\u asitu, A., Goswami, D. 
Quantum Galois groups of subfactors.
\textit{Int. J. Math.} \textbf{33} (2022), Article 2250013.

\bibitem{BB}
B\"ohm, G., Brzezi\'nki, T. Cleft extensions of Hopf algebroids, \textit{Appl. Categ. Structures} \textbf{14} (2006), 431--469. 

\bibitem{CH}
Chirv\u asitu A. On epimorphisms and monomorphisms of Hopf algebras. \textit{J. Algebra} \textbf{323} (2010), 1593--1606.

\bibitem{ChiWalWan}
Chirv\u asitu, A., Walton, C., Wang, X.
On quantum groups associated to a pair of preregular forms.
 \textit{J. Noncommut. Geom.}  \textbf{13} (2019), 115--159.

\bibitem{DNR} D\u asc\u alescu, S., N\u ast\u asescu, C., Raianu, \c S.
Hopf algebras: an introduction. New York, Marcel Dekker, Inc., 2001.

\bibitem{ElduqueKochetov} Elduque, A., Kochetov M.\,V. Gradings on simple Lie algebras. \textit{AMS Mathematical Surveys and Monographs.} \textbf{189}, Providence, R.I., Halifax, NS, 2013, 336 pp.

\bibitem{GV}
P.\ Gro\ss kopf, J.\ Vercruysse,
Free and co-free constructions for Hopf categories, {\em J.\ Pure Appl.\ Algebra} \textbf{228} (2024), 107704.

\bibitem{GV2}
P.\ Gro\ss kopf, J.\ Vercruysse,
The Hopf category of Frobenius algebras, preprint 2024, arXiv:2406.18499v2.

\bibitem{HNUVVW}
Huang, H., Nguyen, V.C., Ure, C., Vashaw, K.B., Veerapen, P., Wang, X. Twisting Manin's universal quantum groups and comodule algebras,
{\em Adv. Math.} \textbf{445} (2024), 109651.

\bibitem{HNUVVW2}
Huang, H., Nguyen, V.C., Ure, C., Vashaw, K.B., Veerapen, P., Wang, X. Twisting of Graded Quantum Groups and Solutions to the Quantum Yang-Baxter Equation, {\em Transformation Groups} \textbf{29} (2024), 1459--1500.

\bibitem{HNUVVW3}
Huang, H., Nguyen, V.C., Vashaw, K.B., Veerapen, P., Wang, X. Quantum-symmetric equivalence is a graded Morita invariant,
{\em Proc. Amer. Math. Soc.}  \textbf{153} (2025), 1389--1409.

\bibitem{HWWW}
Huang, H., Walton, C.,  Wicks, E., Won, R. Universal quantum semigroupoids. \textit{J. Pure Appl. Algebra} \textbf{227} (2023), 107--193.

\bibitem{HylLopVas}
Hyland, M., L\'opez Franco, I., Vasilakopoulou, C. Hopf measuring comonoids and enrichment. \textit{Proc. Lond. Math. Soc.} \textbf{115}:5 (2017), 1118--1148. 

\bibitem{JoyStr}
Joyal, A., Street, R. Braided tensor categories, \textit{Adv. Math.} \textbf{102} (1993), 20--78.

\bibitem{MacLaneCatWork} Mac Lane, S. Categories for the working mathematician. \textit{Graduate texts in Mathematics} \textbf{5}, Springer-Verlag, New York, 1998.

\bibitem{Manin} Manin, Yu. I. Quantum groups and noncommutative geometry, Universite de Montreal, Centre de Recherches Mathematiques, Montreal, QC, 1988. 

\bibitem{NorthPeroux} North, P.\,R., P\'eroux, M. Coinductive control of inductive data types. \texttt{arXiv:2303.16793 [math.CT]}

\bibitem{PZ89} Patera, J., Zassenhaus, H. On Lie gradings. I, \textit{Linear Algebra Appl.}, \textbf{112}
(1989), 87--159.

\bibitem{Porst1}
Porst, H.-E. The formal theory of Hopf algebras. Part I: Hopf monoids in a monoidal category.  \textit{Quaestiones Mathematicae}, \textbf{38}:5 (2015), 631--682.

\bibitem{RV}
Raedschelders, T., Van den Bergh, M. The Manin Hopf algebra of a Koszul Artin-Schelter regular algebra is quasi-
hereditary, \textit{Adv. Math.} \textbf{305} (2017), 601--660.

\bibitem{Sweedler} Sweedler, M.\,E. Hopf Algebras, W. A. Benjamin New York, 1969.

\bibitem{Takeuchi} Takeuchi, M. Free Hopf algebras generated by coalgebras.
\textit{J. Math. Soc. Japan} \textbf{23} (1971), 561--582.

\bibitem{Tambara}
Tambara, D. The coendomorphism bialgebra of an algebra. \textit{J. Fac. Sci. Univ. Tokyo Math.} \textbf{37} (1990), 425--456.

\bibitem{Vasilakopoulou}  Vasilakopoulou, C. Enrichment of categories of algebras and modules. \texttt{arXiv:1205.6450 [math.CT]}


\end{thebibliography}
\end{document}